\documentclass[makeidx]{amsart}
\usepackage{latexsym}
\usepackage{graphicx}
\usepackage{psfrag}
\usepackage{subfigure}
\usepackage{amsmath}
\usepackage{amssymb}
\usepackage{mathrsfs}
\usepackage{epstopdf}
\newtheorem{theorem}{Theorem}[section]
\newtheorem{lemma}[theorem]{Lemma}
\newtheorem{prop}[theorem]{Proposition}
\newtheorem{corollary}[theorem]{Corollary}
 \newtheorem{definition}[theorem]{Definition}
 \newtheorem{note}[theorem]{Note}

\newenvironment{remark}{\noindent \textbf{Remark}.}{\hfill $\square$}

\renewcommand{\Re}{\mathop{\rm Re}\nolimits}
\renewcommand{\Im}{\mathop{\rm Im}\nolimits}

\newcommand{\sgn}{\mathop{\rm sgn}\nolimits}

\numberwithin{equation}{section} \makeindex
\begin{document}

\title[Global existence for a translating  bubble]{Global
existence for a translating near-circular Hele-Shaw bubble with surface tension}
\author{ J. Ye$^1$ and S.  Tanveer$^2$  }
\thanks{$^1$  Department of Mathematics, Ohio State University, Columbus, OH 43210 (jenny$_{-}$yeyj@math.ohio-state.edu).}
\thanks{$^2$  Department of Mathematics, Ohio State University, Columbus, OH 43210 (tanveer@math.ohio-state.edu).}

\maketitle

\today
\bigskip

\begin{abstract}This paper concerns
global existence for arbitrary nonzero surface tension of
bubbles in a Hele-Shaw cell that translate in the presence of a pressure
gradient.
When the cell width to bubble size is sufficiently large,
we show that a unique steady translating near-circular bubble symmetric
about the channel centerline exists, where the bubble translation speed in
the laboratory frame is found
as part of the solution.
We prove
global existence for symmetric sufficiently smooth initial conditions
close to this shape and show that the steady translating
bubble solution is an attractor within this class of disturbances.
In the absence of side walls, we prove stability of the steady
translating circular bubble without restriction
on symmetry of initial conditions.
These results hold for any
nonzero surface tension despite the fact that a local planar approximation
near the front of the bubble would suggest
Saffman Taylor instability.

We exploit a boundary integral approach
that is particularly suitable for analysis of nonzero viscosity ratio
between fluid inside and outside the bubble.
An important element of the proof was
the introduction of a weighted Sobolev norm that
accounts for stabilization due to advection of disturbances
from the front to the back of the bubble.
\end{abstract}

\bigskip

\section{introduction}
The displacement of a more viscous fluid by a less viscous one
in a
Hele-Shaw cell is a canonical
problem in a much wider class of Laplacian growth problems that include
dendritic crystal growth,
electrochemical growth, diffusion limited aggregation,
filtration combusion and tumor growth.
It has attracted many physicists and mathematicans.
In the recent two decades, there are many reviews about this subject
(Saffman \cite{PG2}, Bensimon {\it {et
al.}} \cite{BE1},  Homsy \cite{HO1},  Pelce \cite{PE1},  Kessler {\it
et al.} \cite{KE1},  Tanveer \cite{Tanveer3} \& \cite{Tanveer1},
Hohlov \cite{HY1}, and Howison \cite{HOW2} \& \cite{HO}).

There is a vast literature on zero surface tension problem though
the initial value problem in this case is ill-posed \cite{Howison},
\cite{FokasTanveer} and not always physically relevant
[See \cite{Tanveer1} for detailed discussion of this issue].  With surface
tension, there are rigorous local existence results for
general initial conditions
both for one and two phase problems
\cite{DR1},  \cite{ES2} using different approaches. Also there are  some
global existence and nonlinear stability
results \cite{PM}, \cite{FR} for one and two phase Hele-Shaw for
near-circular initial shapes in the absence of any forcing such as
fluid injection
or pressure gradient.
These have been generalized to
non-Newtonian one phase fluids \cite{ES1}. There are similar results
available for the two phase Stefan problem \cite{ES3}, \cite{PR},
which is mathematically close to but distinct from the
two-phase Hele-Shaw (also called Muskat problem)
being studied here. It is well recognized that
 global existence problem with surface tension
for arbitrary initial
shape is a difficult open problem\footnote{Note the "stable" problem
where
a more viscous fluid displaces a less viscous fluid is relatively simple
and will not be considered here;
there are many global results available in this case.}
though there is quite a substantial literature
involving formal asymptotic and numerical
computations (see cited reviews above).
Even the restricted problem of
stability of steadily propagating shapes such as a
semi-infinite finger \cite{Xie1}, \cite{Xie2} or a finite
translating bubble \cite{Xie2}
for nonzero surface tension remains an open problem for
rigorous analysis.
Translation causes
complications in global analysis due to a less viscous
fluid displacing a more viscous fluid --
a planar front is known to be unstable \cite{PG1} in this case.

This paper considers the motion of a bubble in a Hele-Shaw cell
subject to an external pressure gradient that causes the bubble to
translate. We scale the fluid velocity at $\infty$ in the laboratory
frame to be $1$; we choose $u_0$ so that the non-dimensional
velocity of the fluid   at ${+\infty}$ in the frame of a steady
bubble \footnote{This choice implies that the steady bubble
translates along the positive $x$-axis with non-dimensional speed $2+u_0$ in the
laboratory frame.} along the positive $x$-axis 
is $-(u_0+1)$. The analysis presented also
includes proof of existence and uniqueness of a steady bubble
solution together with determination of $u_0$. We choose the steady
bubble perimeter to be $2\pi$; this corresponds to
nondimensionalizing all length scales appropriately. The
non-dimensional half width of the Hele-Shaw cell will be denoted by
$\frac{\pi}{\beta}$.

The  two-phase Hele-Shaw problem
in the steady bubble frame is described mathematically
as follows:
$\Omega_2 (t) \subset\mathbb{R}^2$
is a simply connected bounded domain occupied by a
fluid with viscosity $\mu_2$ at time $t$, while a different
fluid of viscosity $\mu_1 > \mu_2 $\footnote{The assumption $\mu_1>\mu_2$ is not necessary in the analysis.} occupies
$\Omega_1(t) $, where $\Omega_1(t)\cup\Omega_2(t)$ constructs the strip which half width is $\frac{\pi}{\beta}$, i. e., $\{(x,y)|x\in\mathbb{R},-\frac{\pi}{\beta}<y<\frac{\pi}{\beta}\}$.
We define functions $\phi_1$  and $\phi_2$,
outside and inside $\Omega_2 $ such that
\begin{equation*}\left\{\begin{aligned}
\Delta \phi_1&=0 \mbox{  in  }{\Omega}_1,\\
\Delta \phi_2&=0 \mbox{  in  }\Omega_2,\\
\phi_1&\rightarrow -(u_0+1)x+O(1), \mbox{ as }(x,y)\rightarrow\infty,\\
\frac{\partial\phi_1}{\partial y}&\big(x,\pm\frac{\pi}{\beta}\big)=0, \mbox{ for }x\in\mathbb{R}.
\end{aligned}\right.\tag{O.1}
\end{equation*}

On the free boundary $\partial\Omega_1\cap\partial\Omega_2$ between two fluids,
we require two conditions:
\begin{equation*}\left\{\begin{aligned}
(2+u_0)x&+\phi_1-\frac{\mu_2}{\mu_1}\phi_2=\sigma\kappa,\\
\frac{\partial\phi_1}{\partial n}&=\frac{\partial\phi_2}{\partial n}=v_n,
\end{aligned}\right.\tag{O.2}
\end{equation*}
where $\sigma$ is the coefficient of
surface tension, ${\bf n}$ is the inward unit normal vector
on $\partial\Omega_1\cap\partial\Omega_2$, and $v_n$ is the normal velocity
of the interface. The first condition corresponds to jump in pressure
balanced by surface tension, while
the second is the usual kinematic condition requiring
that the normal motion of
a point on the interface equals
normal fluid velocity on either side of the interface.

The global existence analysis for arbitrary surface tension
is complicated by the far-field pressure gradient that causes
bubble translation since
a planar interface
under the same condition is susceptible to well-known
Saffman-Taylor instability. This difficulty arises both for
finite ($\beta \ne 0$) and infinite cell-width ($\beta =0$).
Locally, near the front of the bubble,
at sufficiently small scale a planar approximation would appear reasonable.
However, some formal arguments
\cite{BE1}, \cite{DE}.
supported by numerical calculations
have suggested that
stabilization occurs on a curved interface through
advection of disturbances from the front of the interface to the
sides. These conclusions are not universally accepted since
formal calculations \cite{JJX}
based on a multi-scale hypothesis suggest
that the steady state is linearly unstable
for sufficiently small surface tension.
Here we resolve this controversy rigorously in favor of stability
at least in the case of a Hele-Shaw bubble with distant sidewalls for
any nonzero surface tension.\footnote{It is to be noted that
the problem tackled here
is not equivalent to taking $O(1)$ sidewall separation and making
bubble size sufficiently small for fixed surface tension, since
if we scale down
bubble size, we must also scale down surface tension values to make
an equivalent problem. In the
small bubble limit any fixed surface tension dominates
translational effects; in our choice of length scale,
this would correspond only to
the simpler case of only sufficiently large $\sigma$.}
We have
introduced a weighted Sobolev space suitable for
controlling terms arising from bubble translation
for any nonzero surface tension $\sigma$.
We are unaware
of any previous work
for global control of small disturbances
superposed on a steadily translating curved interface in Hele-Shaw or
any other related problems.

In the present paper,
we  use a boundary integral formulation due to Hou {\it et al}
\cite{HL1}. This formulation has been widely used for numerical calculations
in a wide variety of free boundary problems involving
Laplace's equation.
Ambrose \cite{AD3} has recently
used this formulation to prove local existence for the Hele-Shaw flow of
general initial shapes \cite{AD3} without surface tension.
Given the wide use of boundary integral methods in computations,
one motivation for the present paper is to further
develop the mathematical machinery associated with this method
so as to be applicable to more general
existence problems.

Adapting the equal arc-length vortex sheet formulation of
Hou {\it et al} \cite{HL1} to the
present geometry,
the boundary curve between the two fluids of differing viscosities is
described parametrically at any time $t$ by $z=x(\alpha,t)+iy(\alpha,t)$, where
$\alpha$ is chosen so that
$z(\alpha+2
\pi,t)=z(\alpha,t)$. We introduce $\theta$ so that
$\frac{\pi}{2}+\alpha+\theta$ is the angle between the
tangent to the curve and the positive $x$-axis
as the boundary is traversed counter-clockwise
with increasing $\alpha$.
Hou {\it et al} \cite{HL2} observed that
a
choice\footnote{This choice or any other choice of
tangential speed of points on the interface has no effect on the interface
shape itself.}
of the tangent velocity $T$
is possible so that the rate of change of arc-length
$s_\alpha \equiv |z_\alpha| $
is independent of $\alpha$ and
corresponds to an equal arc-length interface
parametrization.
They also observed
that this choice
simplifies the evolution equation for $\theta$, and used it in their computational scheme. Note in this equal arc-length formulation  $z_\alpha=x_\alpha+i y_\alpha=\frac{L}{2\pi} e^{i\pi/2+i\alpha+i\theta}$, where $L$ is perimeter length of interface.
Then the unit tangent vector on the interface ${\bf t}=\big(-\sin(\alpha+\theta),\cos(\alpha+\theta)\big)$ and the unit normal vector pointing inward at bubble interface is ${\bf n}=\big(-\cos(\alpha+\theta),-\sin(\alpha+\theta)\big)$.
\begin{definition}
\label{def1.1}
Let $r\geq0$. The Sobolev
space $H^r_p $ is the set of all
$2\pi$-periodic function $f=\sum_{-\infty}^{\infty}\hat{f}(k)
e^{ik\alpha}$ such that
\begin{displaymath}
\|f\|_r=\sqrt{\sum_{k=-\infty}^{\infty}|k|^{2r}|\hat{f}(k)|^2+|\hat{f}(0)|^2}<\infty.
\end{displaymath}
\end{definition}
\begin{note}
\label{note1.2} For $f, g \in H^r_p
$, the Banach Algebra property $ \| f g \|_r \le C_r \| f
\|_r \| g \|_r $ for $r \ge 1 $ with some constant $C_r$ depending on
$r$ is easily proved and will be useful in the sequel.
Also, in what follows the $\hat{}$ symbol will reserved for
Fourier components.
\end{note}
\begin{definition}
\label{def1.3} The Hilbert transform, $\mathcal{H}$, of a function
$f \in H^0_p $ ({\it i.e.} $L_2$) with Fourier
Series $f=\sum_{-\infty}^{\infty}\hat{f}(k) e^{ik\alpha}$ is given
by
\begin{eqnarray}
\mathcal{H}[f] (\alpha)&=&\frac{1}{2
\pi}\mbox{PV}\int_{0}^{2\pi}f(\alpha')\cot{\frac{1}{2}(\alpha-\alpha')}d\alpha'\nonumber\\
&=&\sum_{k\neq 0}-i\sgn(k)\hat{f}(k) e^{ik\alpha}.\nonumber
\end{eqnarray}
\end{definition}

\begin{note}
\label{note1.4} For $f\in  H^1_p $,
the Hilbert transform commutes with differentiation.
We will denote derivative with respect to $\alpha$, either by $D_\alpha$ or
subscript $\alpha$. Also, for the sake of brevity of notation, the
time $t$ dependence will often be omitted, except where it might
cause confusion otherwise.
\end{note}
\medskip

\begin{definition}
\label{def1.7} We define the operator $\Lambda$ to be a derivative
followed by the Hilbert transform: $\Lambda=\mathcal{H}D_\alpha$.
Following Ambrose  \cite{AD3}, we also define commutator
$$ [\mathcal{H},f]g =
\mathcal{H}(fg)-f\mathcal{H}(g).$$
\end{definition}
\begin{note}\label{note1.8}
It is clear that
\begin{displaymath} \Big(\int_0^{2 \pi}\big(
f^2+f\Lambda f\big)d\alpha\Big)^{1/2}
\end{displaymath}
is  equivalent  to $H^{1/2}_p $ norm of
a real-valued $2\pi$-periodic function $f$.
Further, note the operator $\Lambda$ is
self-adjoint in $H^{1/2}_p $ Hilbert space.
\end{note}

\medskip

\begin{definition}
\label{def1.9}   We  define a linear integral operator $\mathcal{K}
[z]$, depending on $z$, as
\begin{equation}\label{1.8}
\mathcal{K}[z]f=\frac{1}{2\pi i}\int_{\alpha-\pi}^{\alpha+\pi}
 f(\alpha')\Big\{K(\alpha,\alpha'\big)-\frac{1}{2z_{\alpha}(\alpha')}\cot{\frac{1}{2}(\alpha-\alpha')}\Big\}d\alpha',
\end{equation}
where
for $\beta=0$,
\begin{equation}\label{kbeta0}{K}(\alpha,\alpha')=\frac{1}{z(\alpha)-z(\alpha')};
\end{equation}
for $\beta\neq0$,
\begin{equation}\label{kbeta}
{K}(\alpha,\alpha')=\frac{\beta}{4}\coth\Big[\frac{\beta}{4}\big(z(\alpha)-z(\alpha')\big)\Big]-\frac{\beta}{4}\tanh\Big[\frac{\beta}{4}\big(z(\alpha)-z^{\ast}(\alpha')\big)\Big].
\end{equation}
\end{definition}

\medskip

\begin{remark} For $2 \pi$-periodic functions $f$ and $z$, it is clear that the upper and lower
limits of the integral above can be replaced by $a$ and $a+2\pi$ respectively for arbitrary $a$. \end{remark}

\begin{definition}
\label{def1.10} We define a complex valued operator $\mathcal{G}[z]$,
depending on $z$, so that
\begin{equation}
\label{1.10} \mathcal{G}[z]\gamma =
z_\alpha\Big[\mathcal{H},\frac{1}{z_\alpha}\Big]\gamma +2i
z_\alpha\mathcal{K}\big[z\big]\gamma.
\end{equation}
It is also convenient to define a related
real operator $\mathcal{F} [z]$, depending on $z$, so that
\begin{equation}
\label{1.11} \mathcal{F} [z] \gamma = \Re \Big (
\frac{1}{i} \mathcal{G} [z] \gamma \Big ).
\end{equation}
\end{definition}

From the Hou {\it et al} \cite{HL2} equal arc-length formulation,  the Hele-Shaw equations (O.1)-(O.2)  reduce  to the following evolution equations for the boundary $\partial\Omega_1\cap\partial\Omega_2$:
 \begin{equation*}
\left\{\begin{aligned}
\theta_t(\alpha,t)&=\frac{2\pi}{L}U_\alpha(\alpha,t)+\frac{2\pi}{L}T(\alpha,t)\big(1+\theta_\alpha(\alpha,t)\big),\\
L_t(t)&=-\int_0^{2\pi}\big(1+\theta_\alpha(\alpha,t)\big) U(\alpha,t) d\alpha,
\end{aligned}\right.\tag{A.1}
\end{equation*}
where
$U$
is the normal interface velocity, determined from
\begin{multline}
\label{1.12} U(\alpha,t) =
\frac{2\pi}{L}\Re\Big(\frac{z_\alpha}{2\pi}\mbox{PV}\int_{\alpha-\pi}^{\alpha+\pi}
\gamma(\alpha')K(\alpha,\alpha')d\alpha'\Big)
+(u_0+1)\cos\left(\alpha+\theta(\alpha)\right)
\\
= \frac{\pi}{L} \mathcal{H} [\gamma] + \frac{\pi}{L}
\Re \left ( \mathcal{G} [z] \gamma \right
)+(u_0+1)\cos\left(\alpha+\theta(\alpha)\right),
\end{multline}
vortex sheet $\gamma$ and the tangent interface velocity are determined, respectively, by
 \begin{equation*}
 \gamma(\alpha,t)=- a_\mu\mathcal{F}[z]\gamma(\alpha,t)+\frac{L}{\pi}\big(1+\frac{\mu_2}{\mu_1+\mu_2}u_0\big)\sin(\alpha+\theta)+\frac{2\pi}{L}\sigma
\theta_{\alpha \alpha},\tag{A.2}
\end{equation*}
\begin{equation*}
T(\alpha,t)=\int_0^{\alpha}\big(1+\theta_{\alpha'}(\alpha',t)\big)U(\alpha',t)d\alpha'-\frac{\alpha}{2\pi}\int_0^{2\pi}\big(1+\theta_\alpha(\alpha,t)\big)
U(\alpha,t) d\alpha, \tag{A.3}
\end{equation*}
where $a_\mu = \frac{\mu_1 - \mu_2}{\mu_1+\mu_2} $

\medskip

For ({A.1})-(A.3), the initial conditions are
\begin{eqnarray}
\label{1.4} \theta(\alpha,0)=\theta_0(\alpha),\,\,L(0)=L_0.
\end{eqnarray}

\begin{note}\label{noteyzero} Since $\big(x_t(\alpha,t),y_t(\alpha,t)\big)=U{\bf n}+T{\bf t}$,
 (A.3)  implies that the interface evolution at $\alpha=0$ is given by
$\big(x_t(0,t),y_t(0,t)\big)=U(0,t) {\bf n}(0,t)$. In particular, this implies
\begin{equation}\label{y0}
y_t(0,t)=-U(0,t)\sin\big(\theta(0,t)\big), \mbox {\rm with ~initial
~~condition}~y (0, 0) = y_0.
\end{equation}

\end{note}

\begin{definition}
\label{def2.17} We denote the bubble area by $V$. From geometric
consideration,
\begin{eqnarray}
\label{5.5}  V = \frac{1}{2}\Im\int_0^{2\pi}z_\alpha
z^{\ast}d\alpha.
\end{eqnarray}
\end{definition}

\begin{remark}
It is well known (indeed easily seen from (O.1)) that the bubble area $V$
will remain invariant in time.
That this is also implied by the boundary integral
formulation (A.1) is not as obvious and is shown in \S 2.
\end{remark}

\begin{definition}
\label{def1.5} We introduce a family of  projections
$\{\mathcal{Q}_n\}$ such that
\begin{displaymath}
\mathcal{Q}_nf=f-\sum_{k=-n}^{n}\hat{f}(k)e^{ik\alpha}
\end{displaymath}
where $f=\sum_{-\infty}^{\infty}\hat{f}(k) e^{ik\alpha}$ and
$n\in\mathbb{Z}^{+}\cup\{0\}$. Henceforth, we will define
$\tilde\theta=\mathcal{Q}_1\theta$.
\end{definition}

\begin{definition}
\label{def1.6} We define $\dot{H}^r$ as a subspace of $H^r_p $ containing
real valued functions so
that $\phi\in\dot{H}^r$ implies $\mathcal{Q}_1\phi=\phi$. Note in
this subspace, $\|\phi\|_r=\|D^r_\alpha\phi\|_0$ for $r\ge1$.
\end{definition}

Without sidewalls, {\it i.e.} for $\beta =0$, our
main result is as follows:
\begin{theorem}\label{theo1.3}{ For any surface tension $\sigma>0$ and $r \ge 3$,
there exists $\epsilon>0$ such that
if $\| \theta_0\|_r<\epsilon$
and $|L_0-2\pi| < \epsilon <\frac{1}{2}$,
then   there
exists a unique   solution $\big(\theta,  L\big)\in C\big([0,\infty),
H^r_p \times\mathbb{R}\big)$  to the Hele-Shaw  problem
(A.1)-(A.3) with the initial condition (\ref{1.4}).
 Further, $\|\tilde\theta\|_r$ and $|\hat\theta (\pm1;t)|$  each
decay exponentially as $t \rightarrow \infty$, $|\hat\theta(0;t)|$
remains finite,  while $L$ approaches $2\sqrt{\pi V}$ exponentially
implying that
a steady translating circular bubble is
asymptotically stable for sufficiently small initial disturbances
in the $H^r_p $ space.}
\end{theorem}
\begin{remark} The proof is completed at the end of \S 4  (see Note \ref{translatingnote4.3}).
\end{remark}

We also consider the problem with finite cell-width ($\beta \ne 0$).
Here, we first prove the existence of a translating steady bubble;
more precisely
we have the following theorem:

\begin{theorem} \label{theo5.6} For any  surface tension $\sigma>0$ and
$r \ge 3$,
there exist for $\epsilon>0$, $\Upsilon>0$  two balls   $O_1=\big\{\beta\in
\mathbb{R}:0\leq\beta<\Upsilon\big\}$ and  $O_2=\big\{(u,v)\in
{H}_p^r\times \mathbb{R}\big|\|u\|_r< \epsilon, |v|<\epsilon\big\}$,
so that for sufficiently small $\epsilon$ and $\Upsilon$,
$\big(\theta^{(s)}, u_0\big)^{T}:O_1\rightarrow O_2$ is the unique
real valued
map
$\big(\theta^{(s)},u_0\big)$ determining the shape
and velocity of a steady translating bubble for $\beta\in O_1$.

Furthermore,
there exists $C$ independent of $\epsilon$ and $\Upsilon$ such that
$$\|\theta^{(s)}\|_r+|u_0|+\|\gamma^{(s)}-2\sin(\cdot)\|_{r-2}\le
C\beta^2,$$
and
$\theta^{(s)} $ is an odd function implying that
the bubble shape is symmetric about the channel centerline.
\end{theorem}

\begin{remark}
We will prove Theorem \ref{theo5.6} in \S 5.3.
Note results for steady bubble and finger
without restriction on $\beta$
but small $\sigma$ is available in  \cite{Xie1}, \cite{Xie2} and \cite{Xie3}.
Here, there is no restriction in $\sigma > 0$, but it is held fixed as
$\beta $ is made sufficiently small. Existence of at least one
steady
translating finger solution for $\sigma > 0$
has been proved earlier \cite{Su} using
different methods.
\end{remark}

For $\beta \ne 0$, we also consider the time evolution problem,
though only for initial conditions for which the bubble shape is
symmetric about the channel centerline. Symmetry implies $\theta$ is
an odd function of $\alpha$.
\begin{definition}\label{def5.9} We define unsteady perturbation
\begin{equation}
\Theta(\alpha,t)=\theta(\alpha,t)-\theta^{(s)}(\alpha).
\end{equation}
We also define
$\widetilde\Theta(\alpha,t)=\mathcal{Q}_1\Theta(\alpha,t)$.
\end{definition}

The main result for the evolution of a translating bubble with side
wall effects ($\beta\neq0$) is as follows:
\begin{theorem}\label{theo8}{ For any surface tension $\sigma>0$ and $r \ge 3$,
there exist $\epsilon,\,\Upsilon>0$ such that if $\|
\Theta(\cdot,0)\|_r<\epsilon$, $|L_0-2\pi| < \epsilon < \frac{1}{2}
$ and $0<\beta<\Upsilon$,  with
$\Theta(-\alpha,0)=-\Theta(\alpha,0)$,  then there exists a unique
solution $\big(\theta,  L\big)\in C\big([0,\infty), H^r_p
\times\mathbb{R}\big)$ with  $\theta(-\alpha,t)=-\theta(\alpha,t)$
to the Hele-Shaw  problem
 (A.1)-(A.3) with  initial  condition (\ref{1.4}). Furthermore, $\|
\Theta\|_r$
decays exponentially as $t \rightarrow \infty$,  while
$L$ approaches $2\sqrt{\pi V}$ exponentially.
Thus  the  translating steady bubble determined in Theorem \ref{theo5.6} is
asymptotically stable for sufficiently small symmetric
initial disturbances in the $H^r_p $ space.}
\end{theorem}
\begin{remark}
This theorem is proved in \S 6 (See  Note \ref{translatingnote5.7}).
\end{remark}

We organize the paper as follows.
 In \S 2,  we introduce equations (B.1)-(B.6) equivalent to (A.1)-(A.3).
It turns out that linearization of (A.1)-(A.3) about a steady shape gives
rise to neutrally stable modes,  including $\hat{\theta}(\pm1;t)$.
It is therefore convenient to project away these Fourier modes and
introduce instead a constraint to determine $\hat{\theta}(\pm 1;t)$ for
given ${\tilde \theta}$.
Further, we find it convenient to
replace the evolution equation for $L$ in (A.1) by
an area constraint relation (B.4) since it is otherwise more
difficult to obtain exponential control on
$L$ directly.
In \S 3, we prove several preliminary lemmas
about some integral operators. In \S 4,
we prove results for near-circular initial shape in the absence of
side walls ($\beta=0$), but without any symmetry assumptions.
In  \S 5, we consider the problem of determining a steady translating bubble
with side-wall effects ($\beta \ne 0$) and
complete the proof of Theorems
\ref{theo5.6}. In \S 6, we consider
the global evolution problem for $\beta \ne 0$ for initial shapes
symmetric about the channel centerline and complete the proof of
Theorem \ref{theo8}.
Because of technical problems
in controlling $\hat{\theta}(0;t)$ for nonzero $\beta$,
we have restricted our attention
to only  symmetric initial condition for which $\hat{\theta}(0;t)=0$.

\section{Equivalent evolution equations}

\begin{definition}
\label{def2.1}
We introduce functions
\begin{equation}
\label{eqomega}
\omega_0 (\alpha) = \int_0^\alpha e^{i \alpha'} d\alpha' ~~~~~\\,~~~
\omega(\alpha)=
\int_0^{\alpha}e^{i\alpha'+i\hat{\theta}(1;t)e^{i\alpha'}+i\hat{\theta}(-1;t)
e^{-i\alpha'}+i\tilde\theta(\alpha')}d\alpha'.
\end{equation}
\end{definition}

\begin{note}\label{notez}
Given the geometric description of $\theta$ in terms of the tangent angle,
it is clear that
\begin{equation}\label{zomega}
z(\alpha,t)=\frac{L}{2\pi}e^{i\frac{\pi}{2}+i\hat{\theta}(0;t)}\omega(\alpha,t)+z(0,t).
\end{equation}
Further, from (\ref{5.5}) and (\ref{zomega}), it follows that
\begin{equation}
\label{Veq}
V= \frac{L^2}{8 \pi^2} \Im ~\int_0^{2\pi}
\left (  \omega_\alpha \omega^* \right ) d\alpha
\end{equation}
The above relation implies
equation (B.4) in the sequel.
\end{note}

For $\beta\neq0$, it is seen that
$y(0,t)$ is not decoupled from (A.1)-(A.3); thus
(\ref{y0}) has to be solved at the same time as (A.1)-(A.3). We will
show (A.1)-(A.3) and (\ref{y0}) with the initial
conditions (\ref{1.4}), $y(0,0)=y_0$ is equivalent to the following evolution
system for $\big(\tilde\theta(\alpha,t),\hat{\theta}(0;t),y(0,t)\big)\in\dot{H}^r\times\mathbb{R}^2$:
\begin{equation*}\left\{\begin{aligned}
\tilde{\theta}_t(\alpha,t)&=\frac{2\pi}{L}\mathcal{Q}_1\big(U_\alpha+T(1+\theta_\alpha)\big),\\
\frac{d\hat{\theta}(0;t)}{dt}&
=\frac{1}{L}\int_0^{2\pi}T(\alpha,t)\big(1
+\theta_\alpha(\alpha,t)\big)d\alpha
\end{aligned}\right.\tag{B.1}
\end{equation*}
\begin{equation*}
y_t(0,t)=-U(0,t)\sin\big(\theta(0,t)\big),
\tag{B.2}
\end{equation*}
where
\begin{equation}\label{alltheta}
\theta=\hat\theta(0;t)+\hat{\theta}(-1;t)e^{-i\alpha}+\hat{\theta}(1;t)e^{i\alpha}+\tilde\theta,
\end{equation}
with $\gamma(\alpha,t)$, $L(t)$, $T(\alpha,t)$ and $\hat\theta(\pm
1;t)$ determined by
\begin{equation*}
(I+a_\mu\mathcal{F}[z])\gamma=
\frac{2\pi}{L}\sigma\theta_{\alpha\alpha}
+
\frac{L}{\pi}\big(1+\frac{\mu_2}{\mu_1+\mu_2}
u_0\big)\sin\big(\alpha+\theta\big),\tag{B.3}
\end{equation*}
\begin{equation*}
L =\sqrt{\frac{8\pi^2 V}{\Im\int_0^{2\pi}
\omega_\alpha (\alpha, t) \omega^{\ast} (\alpha, t) d\alpha}},
\mbox {\rm where~} V = \frac{L_0^2}{8 \pi^2}
\Im \left \{
\int_0^{2 \pi} \omega_\alpha (\alpha, 0) \omega^* (\alpha, 0) d\alpha
\right \},
\tag{B.4}
\end{equation*}
\begin{equation*}
T=\int_0^{\alpha}(1+\theta_{\alpha'})U(\alpha')d\alpha'-\frac{\alpha}{2\pi}\int_0^{2\pi}(1+\theta_{\alpha})U(\alpha)d\alpha,\tag{B.5}
\end{equation*}
\begin{equation*}
\int_0^{2\pi}\exp\Big(i\frac\pi{2}+i\alpha+i\hat{\theta}(-1;t)e^{-i\alpha}
+i\hat{\theta}(1;t)e^{i\alpha}+i\tilde\theta(\alpha,t)\Big)d\alpha=0,
\tag{B.6}\end{equation*}
and $U$ determined by (\ref{1.12}). The initial condition is
\begin{eqnarray}
\label{translating5.1}
\tilde\theta(\alpha,0)=\mathcal{Q}_1\theta_0, \mbox{  }\hat{\theta}(0;0)=\hat{\theta}_0(0) \mbox{ and }y(0,0)=y_0.
\end{eqnarray}

\begin{definition}
\label{def2.4}
Let $r\geq 3$.
 We define  open balls :
 $$\mathcal{B}_{\epsilon}^r=\left\{u\in
 \dot{H}^r| \|u\|_r<\epsilon\right\};$$
 $$S_M=\{y\in\mathbb{R}||y|<M\},$$ for some $M$  independent of $\beta$.
\end{definition}
\begin{remark}
We will eventually choose  $\epsilon>0$ to be small enough for  Theorem \ref{theo1.3} and Theorem \ref{theo8} to
apply.
\end{remark}

For  the constraint (B.6), we have the following result:
\begin{prop}
\label{prop2.6} There exists $\epsilon_1>0$ so that   (B.6)
implicitly defines a unique $C^1$ function $G:\big\{u\in
\dot{H}^1|
\|u\|_1<\epsilon_1\big\}\rightarrow \mathbb{R}^2$ satisfying
$\big(\Re\hat\theta(1;t),\Im\hat{\theta}(1;t)\big)=G\big(\tilde\theta(t)\big)$
with $G(0)=0$ and $G_{\tilde\theta}(0)=0$.
Moreover, $G$ satisfies the following estimates  for
all $u, u_1,u_2\in\big\{u\in
\dot{H}^1|
\|u\|_1<\epsilon_1\big\}$:
\begin{eqnarray}
\label{2.7}
|G(u)|&\leq &\frac{1}{2}\|u\|_1,\\
|G(u_1)-G(u_2)|&\leq& \frac{1}{2}\|u_1-u_2\|_1.
\label{2.8}
\end{eqnarray}
Furthermore, if ${\tilde \theta}$ is odd, then the corresponding
${\hat \theta} (1; t)$ is purely imaginary.
\end{prop}
\begin{proof}
The proof of the first part appears in \cite{JY} (See Proposition 2.4).
Furthermore, if ${\tilde \theta} (-\alpha) = - {\tilde \theta} (\alpha)$,
then on complex conjugation of (B.6), replacing integration variable $\alpha\rightarrow-\alpha$ and local uniqueness
of the mapping $G$, it follows that
${\hat \theta} (1; t) =-\hat{\theta}^{\ast}(1;t)$, hence it is imaginary.
\end{proof}

\begin{note}\label{notetheta1}
Note that calculation of $\hat{\theta} (1; t)$ (and therefore
of ${\hat \theta} (-1; t) = {\hat \theta}^* (1; t)$)
from ${\tilde \theta}$
in Proposition \ref{prop2.6} allows compuation of
$$\mathcal{Q}_0 \theta = {\tilde \theta} (\alpha,t)
+ {\hat \theta} (1; t)
e^{i \alpha}
+ {\hat \theta} (-1; t)
e^{-i \alpha}
$$
and this is an odd function of $\alpha$ for odd ${\tilde \theta}$.
Also, note that
having determined $\gamma$,  $\hat{\theta}(1;t)$ and $\hat{\theta}(-1;t)$,
(\ref{1.12}) and (B.6) determine $U$ and
 $T$ needed in (B.1)-(B.2).
\end{note}

\begin{prop}\label{proparea}  Suppose for $r \ge 3$,
$\big(\theta(\alpha,t), \,L(t), \,y(0,t)\big) \in C^1\left
([0,S], H^r_p \times \mathbb{R} \times S_M\right )$ with
$|L-2\pi|<\frac{1}{2}$ is a solution to the  system
(A.1)-(A.3), (\ref{y0}) with initial
conditions (\ref{1.4}), $y(0,0)=y_0$.
Then the corresponding bubble area
$ V$ is invariant with time.
\end{prop}

\medskip

\begin{proof}

 Taking the derivative with respect to $t$ on both sides of (\ref{5.5}),
 it is readily seen that
\begin{equation}\label{variantS}
\frac{d V}{dt}=\frac{1}{2}\Im\int_0^{2\pi}\big(z_{\alpha}z_t^{\ast}-
z_t z_\alpha^\ast\big)d\alpha
=-\frac{L}{2\pi}\int_0^{2\pi} U d\alpha.
 \end{equation}
Using (\ref{1.12}), we have
\begin{equation}
\frac{d V}{dt}=-\Re\Big(\int_0^{2\pi}\frac{z_\alpha(\alpha)}{2\pi}
\mbox{PV}\int_{\alpha-\pi}^{\alpha+\pi}\gamma(\alpha')
K(\alpha,\alpha')d\alpha'd\alpha\Big).
\end{equation}
Since
\begin{equation*}
\Re\Big(\mbox{PV}\int_0^{2\pi}
\frac{z_\alpha(\alpha)}{z(\alpha)-z(\alpha')}d\alpha\Big) =
\log\big|z(2\pi)-z(\alpha')\big|-\log\big|z(0)-z(\alpha')\big|
=0,
\end{equation*}
the Proposition follows.
\end{proof}

\begin{lemma}
\label{lem2.7} For $r\ge3$ and sufficiently small $\epsilon_1$, the following
statements {\bf (i.)} and {\bf (ii.)} are equivalent:

\noindent{\bf (i.)}
$(\theta, L, y(0,t))\in C^1\Big([0,S], H^r_p
\times\mathbb{R}\times S_M\Big)$ satisfies (A.1) and (\ref{y0}) with
initial conditions (\ref{1.4}) and $y(0, 0)=y_0$, where
$\theta$ is real-valued,
$\|\mathcal{Q}_1\theta\|_1<\epsilon_1$ and
$|L-2\pi| < \epsilon_1 <\frac{1}{2}$, while
$\gamma$, $T$ and $U$ are determined by (A.2), (A.3) and (\ref{1.12}).

\noindent{(\bf ii.)}
$\big(\tilde{\theta},\hat{\theta}(0;t),y(0,t)\big)
\in C^1\Big([0,S], \dot{H}^r
\times\mathbb{R}\times S_M\Big)$ satisfies (B.1)-(B.2),
initial conditions (\ref{translating5.1}),
with $\|\tilde\theta\|_1<\epsilon_1$,
where
$\gamma$, $T$,
$\hat{\theta}(\pm1;t)$, $L$ and $U$ are determined by (B.3)-(B.6),
(\ref{1.12}) and
$$\theta = {\tilde \theta}+{\hat \theta} (0; t) + {\hat \theta} (1; t)
e^{i \alpha} + {\hat \theta} (-1; t) e^{-i \alpha}. $$
\end{lemma}

\begin{proof} The first part involves essentially the same arguments as
Lemma 2.5 in \cite{JY}, except that (\ref{Veq})
is used to derive (B.4) with $V$ determined from
initial conditions (see
Proposition \ref{proparea}).
\medskip

For the second part, assume
$\big(\tilde{\theta}(\alpha,t),\hat{\theta}(0;t), y(0,t)\big)\in
C^1\Big([0,S],H^r_p \times\mathbb{R}\times S_M\Big)$ is a solution to (B.1)-(B.2)
with $\|\tilde\theta\|_1<\epsilon_1$ and
$\gamma$,  $T$, $\hat{\theta}(\pm1;t)$, $L$ and $U$ are
determined from (B.3)-(B.6),  (\ref{1.12}).
From Lemma 2.5 in \cite{JY},
$\theta=\tilde{\theta}+\hat\theta(0;t)+\hat{\theta}(1;t)
e^{i\alpha}+\hat{\theta}(-1;t)e^{-i\alpha},
$ is real valued solution to
the equation for $\theta$ in (A.1), where
$\gamma$, $T$ and $U$ are determined
by (A.2), (A.3) and (\ref{1.12})
 for $t\in[0, S]$.

As far evolution of $L$, we note that
taking time derivative of (B.4),
we have
\begin{eqnarray}\label{translating5.2}
\frac{L_t}{2\pi}\Im\int_0^{2\pi}\omega_\alpha\omega^{\ast}d\alpha
+\frac{L}{2\pi}\Im\int_0^{2\pi}\omega_\alpha\omega^{\ast}_td\alpha=0.
\end{eqnarray}
Using integration by parts, we also have
\begin{eqnarray}
\frac{L}{2\pi}\omega_t&=&\frac{Li}{2\pi}\int_0^\alpha
e^{i\zeta+i\theta(\zeta)}\theta_t(\zeta)d\zeta\nonumber\\
&=&\left(iU(\alpha)+T(\alpha)\right)\omega_\alpha-iU(0)
+\frac{1}{2\pi}\omega
\int_0^{2\pi}(1+\theta_\alpha)Ud\alpha.\nonumber
\end{eqnarray}
In Proposition \ref{proparea}, we noted
$\int_0^{2\pi}U(\alpha,t)d\alpha=0$.
Plugging the above formula into (\ref{translating5.2}), we obtain
\begin{equation}
\label{Lev}
\frac{L_t}{2 \pi} \left (
\Im \int_0^{2 \pi} \omega_\alpha \omega^* d\alpha \right )
+\frac{1}{2\pi}\Big(\Im\int_0^{2\pi}\omega_\alpha
\omega^{\ast}d\alpha\Big)\Big(\int_0^{2\pi}(1+\theta_\alpha)Ud\alpha\Big)=0.
\end{equation}
Furthermore, if $\| {\tilde \theta} \|_1 <
\epsilon$ is sufficiently small, then
using $\Im\int_0^{2\pi}e^{i\alpha}\int_0^\alpha
e^{-i\alpha'}d\alpha'd\alpha=2\pi$, by Sobolev's embedding theorem
and Proposition \ref{prop2.6}, we have
\begin{align}
\label{Imomlbound}
&\Big|\Im\int_0^{2\pi}\omega_\alpha
\omega^{\ast}d\alpha-\Im\int_0^{2\pi}e^{i\alpha}\int_0^\alpha
e^{-i\alpha'}d\alpha'd\alpha\Big|\\
\leq&\Big|\int_0^{2\pi}e^{i\alpha}\big(e^{i\theta}-1\big)
\int_0^\alpha
e^{-i\alpha'-i\theta(\alpha')}d\alpha'd\alpha+\int_0^{2\pi}e^{i\alpha}\int_0^\alpha
e^{-i\alpha'}\big(e^{-i\theta(\alpha')}-1\big)d\alpha'd\alpha\Big|
\nonumber \\
\leq &16\sqrt{2}\pi^2|\theta|_{\infty}\leq C\|\tilde\theta\|_1.
\nonumber
\end{align}
This implies that
$\Im\int_0^{2\pi}\omega_\alpha\omega^{\ast} d\alpha\neq 0$ and so (\ref{Lev})
implies
$$L_t=-\int_0^{2\pi}(1+\theta_\alpha)Ud\alpha, $$
which is evolution equation for $L$ in (A.1).
\end{proof}

\medskip

\begin{remark}
Because of the equivalence shown above, it turns out to be more
convenient to study solutions to the system (B.1)-(B.6), where $U$
is determined from (\ref{1.12}). Further, without loss of generality, we take $\hat\theta(0;0)=0$ since it only determines the origin of $\alpha$.
\end{remark}

\section{Preliminary Lemmas  }
\begin{definition}
\label{def1.11}
 We decompose $\coth$ and $\cot$ functions into the singular and regular
parts at the origin:
\begin{eqnarray}
\coth(w)&=&\frac{1}{w}+l_1(w),\nonumber\\
\cot(w)&=&\frac{1}{w}+l_2(w).\nonumber \end{eqnarray}
\end{definition}

\medskip

We decompose operator
\begin{equation}
\label{Kdecomp}
\mathcal{K}=
\mathcal{K}_1+\mathcal{K}_2,
\end{equation}
where
\begin{eqnarray*}
\mathcal{K}_1[z]f&=&\frac{1}{2\pi
i}\int_{\alpha-\pi}^{\alpha+\pi}f(\alpha')\Big\{\frac{1}{z(\alpha)-z(\alpha')}
-\frac{1}{z_{\alpha}(\alpha')}\cot{\frac{1}{2}(\alpha-\alpha')}\Big\}d\alpha',\nonumber\\
\mathcal{K}_2[z]f&=&\frac{1}{2\pi
i}\int_{\alpha-\pi}^{\alpha+\pi}f(\alpha')\Big\{\frac\beta4
l_1\big(\frac{1}4\beta({z(\alpha)-z(\alpha')})\big)-\frac{\beta}{4}\tanh\Big[\frac{\beta}{4}\big(z(\alpha)-z^{\ast}(\alpha')\big)\Big]\Big\}d\alpha'.\nonumber
\end{eqnarray*}

\begin{definition}
\label{defG1}
Related to $\mathcal{G}$ and $\mathcal{F}$,
we define operators $\mathcal{G}_1$, $\mathcal{F}_1 $ so that
\begin{equation}
\label{eqG1}
\mathcal{G}_1 [z] \gamma = z_\alpha \left [ \mathcal{H}, \frac{1}{z_\alpha}
\right ] \gamma + 2 i z_\alpha \mathcal{K}_1 [z] \gamma ,\,\,\,\,\, \mathcal{G}_2 [z] \gamma = 2 i z_\alpha \mathcal{K}_2 [z] \gamma , \mbox
~~~\mathcal{F}_1 [z] \gamma = \Re \left ( \frac{1}{i} \mathcal{G}_1 [z] \gamma
\right ).
\end{equation}
\end{definition}
\begin{note}
\label{noteg1}
It is readily checked that for any $f \in H^0_p$,
\begin{equation}
\label{eqnoteg1}
\frac{\omega_{0,\alpha}}{\pi} PV \int_0^{2 \pi}
\frac{f (\alpha') d\alpha'}{\omega_0 (\alpha) - \omega_0 (\alpha')}
= \mathcal{H} [f] (\alpha) + i {\hat f} (0),
\end{equation}
implies that
\begin{equation}
\label{eqnoteg2}
\mathcal{G}_1 [ \omega_0 ] f = i {\hat f} (0),
\end{equation}
which is imaginary for
real valued $f$.
\end{note}

\begin{definition}\label{defxi} We define operators $\Xi_e, \Xi_s, \Xi_c$
so that
\begin{eqnarray*}
\Xi_e [u](\alpha)&=&e^{iu(\alpha)}-1-iu(\alpha) ,\nonumber \\
\Xi_s [u; a] (\alpha) &=& \sin \left (u (\alpha) + \alpha+ a \right ) -
\sin (\alpha + a) - u (\alpha) \cos (\alpha+ a), \nonumber \\
\Xi_c [u; a] (\alpha) &=& \cos \left (u (\alpha) + \alpha+ a \right ) -
\cos (\alpha + a) + u (\alpha) \sin (\alpha+ a ), \nonumber
\end{eqnarray*}
for a real function $u\in H^{r}_p $ with $r\geq 1$.
\end{definition}

\medskip

In the rest of this section, we find some
estimates for integral operators and functions  in terms of
$\tilde\theta$ and $\hat\theta(0;t)$, which will be useful
later. Recall tangent angle of the
curve is
$\frac{\pi}{2}+\alpha+\theta(\alpha)=\frac{\pi}{2}+\alpha+\tilde\theta(\alpha)+\hat\theta(0;t)+\hat{\theta}(-1;t)e^{-i\alpha}+\hat{\theta}(1;t)e^{i\alpha}$,
where $\hat{\theta}(1;t)$ and $\hat{\theta}(-1;t)$ are determined through
$G(\tilde\theta)$.

\medskip

\begin{lemma}(See Lemma 3.1 in \cite{JY})
\label{lem3.1} Assume $ \| \tilde \theta \|_1 < \epsilon_1$ where
$\epsilon_1$ is small enough for Proposition  \ref{prop2.6} to apply.
Then $\omega$ determined from ${\tilde \theta} \in {\dot H}^r$ through
(\ref{eqomega}) satisfies the following estimates
for
$r \ge 1$,
 \begin{equation}
 \label{1ststate}
\|\omega_\alpha\|_r\leq
C_1(\|\tilde\theta\|_r+1)\exp\left(C_2\|\tilde\theta\|_{r-1}\right),~~~~~~~~~
\Big\|\frac{1}{\omega_\alpha}\Big\|_r\leq
C_1(\|\tilde\theta\|_r+1)\exp\left(C_2\|\tilde\theta\|_{r-1}\right),
\end{equation}
where constants $C_1$ and $C_2$, depend only on $r$, and particularly for $r=1$,
$C_2=0$.

Similarly, if $z$ determined by $\big({\tilde\theta},\hat{\theta}(0;t), L\big) \in \dot{H}^r\times\mathbb{R}^2$, then for
$r \ge 1$,
 \begin{equation}
 \label{zstate}
\|z_\alpha\|_r\leq
C_1L(\|\tilde\theta\|_r+1)\exp\left(C_2\|\tilde\theta\|_{r-1}\right),
\end{equation}
where constants $C_1$ and $C_2$, depend only on $r$, and particularly for $r=1$,
$C_2=0$.

Further, if $\omega^{(1)}, \omega^{(2)}$ correspond respectively to
${\tilde \theta}^{(1)}, {\tilde \theta}^{(2)} \in \dot{H}^{r} $,
where $\| {\tilde \theta}^{(1)} \|_1 $, $\| {\tilde \theta}^{(2)}
\|_1 < \epsilon_1$, then for $r \ge 1$,
\begin{eqnarray}
\label{3rdstate}
 \|\omega_{\alpha}^{(1)} - \omega_\alpha^{(2)} \|_{r} \leq
 C_1 \| {\tilde \theta}^{(1)} - {\tilde \theta}^{(2)}  \|_r
\exp{\left ( C_2 \left ( \|{\tilde \theta}^{(1)} \|_{r} + \| {\tilde
\theta}^{(2)} \|_r
\right )\right)},\\
\left\|\frac{1}{\omega_{\alpha}^{(1)}} -
\frac{1}{\omega_\alpha^{(2)}}\right \|_{r} \leq
 C_1 \| {\tilde \theta}^{(1)} - {\tilde \theta}^{(2)}  \|_r
\exp{\left (C_2 \left ( \|{\tilde \theta}^{(1)} \|_{r} + \| {\tilde
\theta}^{(2)} \|_r \right ) \right )},
\label{4thstate}
\end{eqnarray}
while for $r\ge2$,
\begin{eqnarray}
\label{5thstate} \|\omega_\alpha^{(1)} - \omega_\alpha^{(2)}\|_r
&\leq& C_1\left(\|\tilde\theta^{(1)}-\tilde\theta^{(2)}\|_r
+\|\tilde\theta^{(2)}\|_r\|\tilde\theta^{(1)}-\tilde\theta^{(2)}\|_{r-1}\right)
\\
&\times&
\exp\left(C_2\big(\|\tilde\theta^{(1)}\|_{r-1}
+\|\tilde\theta^{(2)}\|_{r-1}\big)\right),\nonumber\\
\label{6thstate} \left\|\frac{1}{\omega_\alpha^{(1)}} -\frac{1}{
\omega_\alpha^{(2)}}\right\|_r &\leq&
C_1\left(\|\tilde\theta^{(1)}-\tilde\theta^{(2)}\|_r
+\|\tilde\theta^{(2)}\|_r\|\tilde\theta^{(1)}-\tilde\theta^{(2)}\|_{r-1}\right)
\\
&\times&
\exp\left(C_2\big(\|\tilde\theta^{(1)}\|_{r-1}
+\|\tilde\theta^{(2)}\|_{r-1}\big)\right),
\nonumber
\end{eqnarray}
where the constants $C_1$ and $C_2$ depend only on $r$.
 \end{lemma}

\begin{lemma}
\label{coroXi} If $F$ is an entire function  of order one\footnote{An entire function $ f$ of order $m$  satisfies
$$|f(z)|\le e^{C|z|^m}, \mbox{ for }z\in\mathbb{C}.$$}  with $F(u)=\sum_{j=j_0}^\infty a_j u^j$ for $j_0=1$ or $2$. Then for
$u\in H^{r+1}_p $ with $r\ge 1$, $F\big(u(\alpha)\big)$ satisfies

(i) $j_0=1$:
 \begin{eqnarray*}
\left\|F\big(u(\cdot)\big)\right\|_0&\leq&
C_1 \exp\left( C_2\|u\|_{1}\right)\|u\|_1, \\
\left\|F\big(u(\cdot)\big)\right\|_{r+1}&\leq&
C_1 \exp\left( C_2\|u\|_{r}\right)\|u\|_{r+1};
\end{eqnarray*}

(ii) $j_0=2$:
 \begin{eqnarray*}
\left\|F\big(u(\cdot)\big)\right\|_0&\leq&
C_1 \exp\left( C_2\|u\|_{1}\right)\|u\|_1^2, \\
\left\|F\big(u(\cdot)\big)\right\|_{r+1}&\leq&
C_1 \exp\left( C_2\|u\|_{r}\right)\|u\|_{r+1}\|u\|_r,
\end{eqnarray*}
where the constants $C_1$ and $C_2$ depend only on $r$.

Further, if  both $u^{(1)}$ and $ u^{(2)}$ belong to $ H^{r+1}_p $,
then for $r \ge 1$,

(i) $j_0=1$:
\begin{eqnarray*}
 \left\|F\big(u^{(1)}(\cdot)\big) -F\big(u^{(2)}(\cdot)\big)\right \|_{0} &\leq&
 C_1 \| u^{(1)} -u^{(2)}  \|_1
\exp{\left [ C_2 \left ( \|u^{(1)} \|_{1} + \| u^{(2)} \|_1
\right )\right]}\\
\left\|F\big(u^{(1)}(\cdot)\big) -F\big(u^{(2)}(\cdot)\big)\right \|_{r+1} &\leq&
 C_1 \Big(\| u^{(1)} -u^{(2)}  \|_{r+1}+\| u^{(1)} -u^{(2)}
\|_{r}\|u^{(2)}\|_{r+1}\Big)\\
 &&\times
\exp{\left [ C_2 \left ( \|u^{(1)} \|_{r} + \| u^{(2)} \|_r
\right )\right]};
\end{eqnarray*}

(ii) $j_0=2$:
\begin{eqnarray*}
 \left\|F\big(u^{(1)}(\cdot)\big) -F\big(u^{(2)}(\cdot)\big)\right \|_{0} &\leq&
 C_1 \| u^{(1)} -u^{(2)}  \|_1
\Big\{\exp{\left [ C_2 \left ( \|u^{(1)} \|_{1} + \| u^{(2)} \|_1
\right )\right]}-1\Big\}\\
\left\|F\big(u^{(1)}(\cdot)\big) -F\big(u^{(2)}(\cdot)\big)\right \|_{r+1} &\leq&
 C_1 \Big(\| u^{(1)} -u^{(2)}  \|_{r+1}\|u^{(1)}\|_r+\| u^{(1)} -u^{(2)}
\|_{r}\|u^{(2)}\|_{r+1}\Big)\\
 &&\times
\exp{\left [ C_2 \left ( \|u^{(1)} \|_{r} + \| u^{(2)} \|_r
\right )\right]},
\end{eqnarray*}
where the constants $C_1$ and $C_2$ depend only on $r$.
\end{lemma}

\begin{proof} The proof is fairly routine and is
is relegated to the appendix.
\end{proof}

\begin{note}\label{notecoroXi}
In particular, $\Xi_e$, $\Xi_s$ and $\Xi_c$ satisfy Lemma \ref{coroXi} with $j_0=2$. $\sin(\alpha+a+u)-\sin(\alpha+a)$ also satisfies Lemma \ref{coroXi} with $j_0=1$.
\end{note}

The following divided differences are
useful.
\begin{definition}
\label{def3.2}
For $z \in H_p^r$, we define operators $q_1$ and $q_2$ so that
\begin{displaymath}
q_1[z](\alpha,\alpha')=\frac{z(\alpha)-z(\alpha')}{
\alpha-\alpha'}=\int_0^1
D z (t\alpha+(1-t)\alpha')dt,
\end{displaymath}
\begin{eqnarray}
q_2[z](\alpha,\alpha')&=&\frac{z(\alpha)-z(\alpha')-
z_{\alpha}(\alpha)(\alpha-\alpha')}{(\alpha-\alpha')^2}
=\int_0^1(t-1)D^2 z ((1-t)\alpha+t\alpha')dt , \nonumber
\end{eqnarray}
where $D$ and $D^2 $ denote first and second derivatives with respect
to the argument.
\end{definition}

\medskip

\begin{prop}
\label{prop3.4} There exists $\epsilon_1 > 0$
so that $\|\tilde\theta\|_1\leq\epsilon_1$ implies
\begin{eqnarray}\label{translating5.4}
\big|q_1[\omega](\alpha,\alpha')\big|\geq \frac{1}{8},
\end{eqnarray}
and
\begin{eqnarray}
\label{translating5.3}\big |q _1 [z] (\alpha, \alpha')\big | \ge
\sqrt{\frac{\pi V}{24}} ~~,~~{\rm for} ~~ 0 < |\alpha
-\alpha'| \le \pi,
\end{eqnarray}
which implies that the curve $z (\alpha)$ is non-self-intersecting.
\end{prop}
\begin{proof}
The first part follows from Proposition 3.3 in \cite{JY}.
Since $\Im \int_0^2\pi \omega_{0,\alpha} \omega^*_0 d\alpha = 2 \pi$,
using (\ref{Imomlbound}), we obtain
for
$C\epsilon_1\leq \pi$,
\begin{equation}
\label{ImomBound}
\pi\leq\Im\int_0^{2\pi}\omega_\alpha
\omega^{\ast}d\alpha\leq 3\pi.
\end{equation}
From (B.4), we obtain
\begin{eqnarray}\label{translating5.5}
\sqrt{\frac{8\pi V}{3}}\leq L\leq\sqrt{8\pi V}.
\end{eqnarray}

\medskip

Combining (\ref{translating5.4}) and (\ref{translating5.5}),  if
$\|\tilde\theta\|_1<\epsilon_1$, then
\begin{eqnarray*}
\big|q_1[z](\alpha,\alpha')\big|=\frac{L}{2\pi}\big|q_1[\omega](\alpha,\alpha')\big|\geq
\sqrt{\frac{\pi V}{24}}, \mbox{ for all
}0<|\alpha-\alpha'|\leq \pi.
\end{eqnarray*}

 \end{proof}

\medskip

\begin{lemma}{(See Lemma 5 in \cite{AD1})}
\label{lem3.3} Assume $z$ and $\omega$ are related through
(\ref{eqomega}) and (\ref{zomega}). Let $z_\alpha \in H^{j}_p
$ for $j \ge 0$. Then for any real $a$, $D_\alpha^j q_1,
D_{\alpha'}^j q_1 \in H^0 [a, a+2 \pi]$ in both variables $\alpha$
or $\alpha'$ and satisfy the bounds
$$ \| D_\alpha^j q_1[z] \|_{0} \leq CL \|\omega_\alpha\|_{j} ~~~~,~~~~~
\| D_{\alpha'}^j q_1[z] \|_{0} \leq CL \|\omega_\alpha\|_{j}$$
with $C$ only depending on $j$ (in particular independent of $a$).
Further if $ z_{\alpha \alpha} \in H^{j}_p
$ for $j \ge 0$, then $D_\alpha^j q_2, D_{\alpha'}^j
q_2 \in H^0 [a, a+2 \pi]$ in both variables $\alpha$ and $\alpha'$
and satisfy
$$ \| D_\alpha^j q_2 [z] \|_{0} \leq C L\|\omega_{\alpha\alpha} \|_{j}
~~~~,~~~~~ \| D_{\alpha'}^j q_2
[z] \|_{0} \leq CL
\|\omega_{\alpha \alpha} \|_{j}$$ with $C$ only depending on $j$.
\end{lemma}

\medskip

\begin{lemma}
\label{lem3.5}
Let  $\omega^{(1)}, \omega^{(2)} \in H^{j+1}_p$ for
$j\geq 0$.
Suppose
$$ |q_1[\omega^{(1)}](\alpha,\alpha')|\geq \frac{1}{8}, \mbox{ for }
0<|\alpha-\alpha'|\leq \pi.$$
Then for $j=0$, there exists constant $C_1$ independent of $\alpha$ such that
\begin{equation*}
\Big(\int_{\alpha-\pi}^{\alpha+\pi}\Big|
\frac{q_2[\omega^{(2)}](\alpha,\alpha')}{q_1[\omega^{(1)}](\alpha,\alpha')}
\Big|^2d\alpha'\Big)^{\frac{1}{2}}\leq
C_1\|\omega^{(2)}_\alpha\|_{1}.
\end{equation*}
Further, for $j\ge3$,
\begin{multline}
\Big(\int_{\alpha-\pi}^{\alpha+\pi}\Big|D_\alpha^j
\frac{q_2[\omega^{(2)}](\alpha,\alpha')}{q_1[\omega^{(1)}](\alpha,\alpha')}
\Big|^2d\alpha'\Big)^{\frac{1}{2}}\\\leq
C_2\left(\|\omega^{(2)}_\alpha\|_{j+1}
+\|\omega^{(2)}_\alpha\|_{j-1}\|\omega^{(1)}_\alpha\|_{j}\right)
(\|\omega^{(1)}_\alpha\|_{j-1}^{j-1}+1), \label{5.3.2}
\end{multline}
where $C_2$  depends on $j$ alone, but not on $\alpha$.
\end{lemma}

\medskip

\begin{proof}
We note that
$$ D_\alpha^j \frac{q_2}{q_1} = \sum_{l=0}^j C_{j,l} D_\alpha^{j-l} q_2
D^{l}_\alpha \frac{1}{q_1}. $$
Using Lemma \ref{lem3.3} with $L=2\pi$
it follows that for $ l \ge 1$
$$\left\| D^l_\alpha \frac{1}{q_1} \right\|_0 \le C_1\|q_1\|_l  \left (1+\|  q_1 \|_{l-1}^{l-1} \right ) \le
C_1\|\omega_\alpha^{(1)}\|_l\left (
\| \omega^{(1)}_\alpha \|_{l-1}^{l-1}+1 \right ),$$
and
\begin{equation*}
\left\|D_\alpha^j \frac{q_2}{q_1}\right\|_0\leq C \sum_{l=1}^{j-1}
\left\| D_\alpha^{j-l} q_2 \right\|_{0} \left\| D^{l}_\alpha
\frac{1}{q_1}\right\|_\infty
+ C\left\| D_\alpha^j \frac{1}{q_1}\right \|_0 \left\| q_2 \right\|_{\infty}
+ C \left\| D_\alpha^j q_2 \right\|_0\left \| \frac{1}{q_1 } \right\|_\infty.
\end{equation*}
The lemma immediately follows from Lemma \ref{lem3.3} on using
$\| \frac{1}{q_1}\|_\infty \le C$ and $\left\| D_\alpha^{l} \frac{1}{q_1} \right\|_{\infty} \leq C\left\| \frac{1}{q_1} \right\|_{l+1} $.
\end{proof}

\begin{lemma}
\label{insertnew}
Assume $\omega^{(1)}, \omega^{(2)} \in H_p^{j+1}$ for $j \ge 0$. Assume
further that
$$\Big | q_1 [\omega^{(1)} ] (\alpha, \alpha') \Big |
\ge \frac{1}{8} ~~{\rm for} ~
0 < |\alpha - \alpha'| \le \pi .$$
Then
 for $j=0$, there exists constant $C_1$ independent of $\alpha$ such that
\begin{equation*}
\Big(\int_{0}^{2 \pi}\Big|
\frac{q_1[\omega^{(2)}](\alpha,\alpha')}{q_1[\omega^{(1)}](\alpha,\alpha')}
\Big|^2d\alpha'\Big)^{\frac{1}{2}}\\
\leq
C_1\|\omega^{(2)}_\alpha\|_{0}.
\end{equation*}
Further, for $j\ge3$,
\begin{multline}
\label{eqq1q1}
\Big(\int_{0}^{2 \pi}\Big|D_\alpha^j
\frac{q_1[\omega^{(2)}](\alpha,\alpha')}{q_1[\omega^{(1)}](\alpha,\alpha')}
\Big|^2d\alpha'\Big)^{\frac{1}{2}}\\
\leq
C_2\left(\|\omega^{(2)}_\alpha\|_{j}
+\|\omega^{(2)}_\alpha\|_{j-2}\|\omega^{(1)}_\alpha\|_{j}\right)
(1+\|\omega^{(1)}_\alpha\|_{j-1}^{j-1}),
\end{multline}
where $C_2$ depends on $j$ only.
\end{lemma}
\begin{proof}
The proof is almost identical to that of
Lemma \ref{lem3.5}. It uses
Lemma \ref{lem3.3} and the lower bound on $q_1 [\omega^{(1)}]$.
We note that integrand on the left of (\ref{eqq1q1}) is
$2 \pi$-periodic in $\alpha'$,  noting that factors of $(\alpha-\alpha')$
in $q_1[\omega^{(1)}]$ and $q_1 [\omega^{(2)}]$
cancel each other.  We are therefore free to replace the
upper and lower bound in the integral in $\alpha'$ by
$\alpha + \pi$ and $\alpha - \pi$ respectively for which
$|q_1|$ is bounded below as needed.
\end{proof}

\begin{lemma}
\label{coro3.11}
Assume $f, g \in H^{j}_p$,
for
$j\geq 0$,
with Fourier components ${\hat f} (0), {\hat g} (0) = 0 $
and $h  \in H^0_p $.
Suppose
\begin{equation}
\label{glowerb}
\big |\int_{0}^1 g(t\alpha+(1-t)\alpha')dt\big|\geq \frac{1}{8}, \mbox{ for }
0 \le |\alpha'-\alpha| \le \pi.
\end{equation}
Then  for $j=0$, there exists constant $C_1$ independent of $\alpha$ such that
\begin{eqnarray*}
\int_0^{2\pi} \Big| h(\alpha')
\frac{\int_{\alpha'}^\alpha f(\tau)d\tau}{\int_{\alpha'}^\alpha g(\tau)d\tau}
\Big | d \alpha' \leq
C_1\|h\|_0\|f\|_{0}.
\end{eqnarray*}
Further, for $j\ge3$,
\begin{eqnarray*}
\int_0^{2\pi} \Big| h(\alpha') D_\alpha^j
\frac{\int_{\alpha'}^\alpha f(\tau)d\tau}{\int_{\alpha'}^\alpha g(\tau)d\tau}
\Big | d \alpha' \leq
C_2\|h\|_0\left(\|f\|_{j}
+\|f\|_{j-2}\|g\|_{j}\right)(1+\|g\|_{j-1}^{j-1}),
\end{eqnarray*}
where $C_2$ depends on $j$ only.
\end{lemma}
\begin{proof}
We define
$$\omega^{(1)} (\alpha ) = \int_0^\alpha g(s) ds, \mbox ~~
\omega^{(2)} (\alpha) = \int_0^\alpha f(s) ds .$$
Clearly, $\omega^{(1)}, \omega^{(2)} \in H^{k+1}_p$
since ${\hat g} (0)=0={\hat f} (0)$.
We note
$$
\frac{\int_{\alpha'}^\alpha f(\tau)d\tau}{\int_{\alpha'}^\alpha
g(\tau)d\tau}
= \frac{q_1 [ \omega^{(2)} ] (\alpha, \alpha') }{
q_1 [\omega^{(1)}] (\alpha, \alpha') }.
$$
Further, we note that the given condition on lower bound involving $g$
becomes
$$ \Big | q_1 [\omega^{(1)} ] (\alpha, \alpha') \Big | \ge \frac{1}{8} .$$
Using Lemma
\ref {insertnew}, the proof follows using Cauchy Schwartz inequality.
\end{proof}

\begin{lemma}
\label{insertnew2}
Assume $\omega^{(1)}, \omega^{(2)} \in H_p^{j+2}$ with $j\ge0$.
Suppose
$$\Big | q_1 [\omega^{(1)} ] (\alpha, \alpha') \Big |
\ge \frac{1}{8} ~~,~~
\Big | q_1 [\omega^{(2)} ] (\alpha, \alpha') \Big |
\ge \frac{1}{8} ~~{\rm for} ~
0 < |\alpha - \alpha'| \le \pi. $$
 Then $j=0$, for  any $a \in \mathbb{R}$,  there exists constant $C_1$ independent of $\alpha$ and $a$ such that
\begin{equation*}
\left \{ \Big(\int_{a}^{a+2\pi}
\Big |  
\frac{q_2[\omega^{(2)}](\alpha',\alpha)}{q_1[\omega^{(2)}](\alpha',\alpha)}
-\frac{q_2[\omega^{(1)}](\alpha',\alpha)}{q_1[\omega^{(1)}](\alpha',\alpha)}
\Big |^2 d\alpha' \right \}^{1/2}
\leq
C_1\|\omega^{(2)}_\alpha-\omega^{(1)}_\alpha \|_{1}
 \left (1+   \|\omega^{(1)}_\alpha\|_{1}\right \} .
\end{equation*} Further, for $j\ge3$,
\begin{multline*}
\left \{ \Big(\int_{a}^{a+2\pi}
\Big | D_\alpha^j \left (
\frac{q_2[\omega^{(2)}](\alpha',\alpha)}{q_1[\omega^{(2)}](\alpha',\alpha)}
-\frac{q_2[\omega^{(1)}](\alpha',\alpha)}{q_1[\omega^{(1)}](\alpha',\alpha)}
\right )
\Big |^2 d\alpha' \right \}^{1/2} \nonumber \\
\leq
C\left(\|\omega^{(2)}_\alpha-\omega^{(1)}_\alpha \|_{j+1}
+\|\omega^{(2)}_\alpha-\omega^{(1)}_\alpha \|_j\|\omega^{(1)}_\alpha\|_{j+1}
\right)
 \left (1+   \|\omega^{(1)}_\alpha\|_{j}^j
+ \|\omega^{(2)}_\alpha \|_{j} ^j\right \} ,
\end{multline*}
where $C$ depends on $j$ alone, but not on $a$ and $\alpha$.
\end{lemma}
\begin{proof} We note from the definitions of $q_1$ and $q_2$ that
the nonperiodic term
$\frac{1}{\alpha-\alpha'}$ that
appears in
each $\frac{q_2}{q_1}$ in the integrand
cancels each other out and we are left with
integrating a
$2 \pi$-periodic function in $\alpha'$; hence
there is no dependence on $a$, and we may choose
$a = \alpha-\pi$ in the proof.
The rest of the proof is similar to
that of Lemma \ref{lem3.5}. We note that
$$ \frac{q_2 [\omega^{(2)} ] }{ q_1 [\omega^{(2)} ]
 }
- \frac{q_2 [\omega^{(1)} ] }{ q_1 [\omega^{(1)}]
}  =
\frac{q_2 [\omega^{(2)}-\omega^{(1)}]}{q_1[\omega^{(2)} ]}
-\frac{q_2 [\omega^{(1)}] q_1 [\omega^{(2)}-\omega^{(1)}]}{
q_1 [\omega^{(1)} ] q_1[\omega^{(2)} ]}
$$
and that the denominators are bounded away from zero. We use
Lemmas
\ref{lem3.3}, \ref{lem3.5} and the
Banach algebra property for $\| \cdot \|_j $ norms in $\alpha'$ for $j \ge 1$.
For $j = 0$, the result follows from
$$
\left\| \frac{q_2 [\omega^{(1)}] }{q_1 [\omega^{(1)}] }\right \|_{L^\infty}
\le
C \left\| \frac{q_2 [\omega^{(1)}] }{q_1 [\omega^{(1)}] }\right \|_{1},
$$
where the norms are taken in $\alpha' $.
\end{proof}

\begin{lemma}\label{leminsert3}
For $\| {\tilde \theta} \|_1 < \epsilon$ sufficiently
small, $\omega$ determined from ${\tilde \theta}$ through
(\ref{eqomega}), then for ${\tilde \theta}, J  \in H^r_p $ for $r \ge 3$
and any $a$,  there exists constant $C_r$ only depending on $r$ such that
$$ \left\| \frac{1}{\pi} \mbox{PV}
\int_{a}^{a+2\pi} \frac{\omega_\alpha (\alpha) J (\alpha') d\alpha'}{
\omega (\alpha) - \omega (\alpha') } \right\|_r
\le C_r \left [ \| J \|_r +  \| J \|_0 \| {\tilde \theta} \|_{r+1}
\exp \left ( C_r \| {\tilde \theta} \|_{r} \right ) \right ].
$$
\end{lemma}
\begin{proof}
We note from (\ref{eqnoteg2}) that $J \in H^0_p$,
$$ \frac{\omega_{0_\alpha}}{\pi}  \mbox{PV}\int_{a}^{a+2\pi}
\frac{J (\alpha') d\alpha'}{\omega_0 (\alpha) - \omega_0 (\alpha') }
= \mathcal{H} [J] (\alpha) + i {\hat J} (0) $$
and we know that $\| \mathcal{H} J \|_r = \| J \|_r $.
Therefore, the integrand may be written as
$$ i {\hat J} (0) + \mathcal{H} [J] (\alpha) + \frac{1}{\pi}
\int_{\alpha-\pi}^{\alpha+\pi} d\alpha' J (\alpha')
\left \{ \frac{q_2 [\omega] (\alpha, \alpha')}{q_1 [\omega] (\alpha, \alpha')}
- \frac{q_2 [\omega_0] (\alpha, \alpha')}{q_1 [\omega_0] (\alpha, \alpha')}
\right \}.
$$
The proof follows from applying
Lemmas  \ref{insertnew2}, \ref{lem3.1}, Proposition \ref{prop3.4} and
using Cauchy Schwartz
inequality and noting
$\omega=\omega_0$, when ${\tilde \theta}=0$.
\end{proof}

\begin{lemma}\label{leminsertnew4}
Assume ${\tilde \theta} \in \dot{H}^{r+1}$, $J  \in H^r_p $
 and $\omega^{[3]} \in
H^{r+1}_p $
for $r \ge 3$. Assume $\|{\tilde \theta}\|_1 < \epsilon$
is sufficiently small and $\omega$ is determined from
${\tilde \theta} $ through (\ref{eqomega}).
Then for
any $a$, there exists constant $C_r$ only depending on $r$ such that
\begin{multline*} \left\|
\omega_\alpha
 \mbox{PV}\int_{a}^{a+2\pi} \frac{J (\alpha') q_1 [\omega^{[3]} ] (\alpha, \alpha')
d\alpha'}{
\left ( \omega (\alpha) - \omega (\alpha') \right ) q_1 [\omega] (\alpha,
\alpha') } \right\|_r\\
\le C_r \left\{\| \omega^{[3]}_\alpha \|_{r}
\left [ \| J \|_r + \| J \|_0 \| {\tilde \theta} \|_{r+1} \exp
 \left (C_r\|{\tilde \theta} \|_{r} \right )  \right ]+\| \omega^{[3]}_\alpha \|_{r+1}\exp
 \left (C_r\|{\tilde \theta} \|_{r} \right )\right\}.
\end{multline*}
\end{lemma}
\begin{proof} We note that
\begin{multline*}
\omega_\alpha  \mbox{PV}
\int_{a}^{a+2\pi} \frac{J (\alpha') q_1 [\omega^{[3]} ] (\alpha, \alpha')
d\alpha'}{
\left ( \omega (\alpha) - \omega (\alpha') \right ) q_1 [\omega] (\alpha,
\alpha') }
= -D_\alpha \int_{a}^{a+2 \pi}
\frac{J (\alpha') q_1 [\omega^{[3]} ] (\alpha, \alpha') }{
 q_1 [\omega] (\alpha, \alpha')} d\alpha' \\
+ \frac{\omega_\alpha^{[3]}}{\omega_\alpha} \mbox{PV}
\int_{a}^{a+2\pi} \frac{J(\alpha') \omega_\alpha (\alpha) d\alpha'}{
\omega (\alpha)
-\omega (\alpha') }.
\end{multline*}
We rely on Lemmas \ref{insertnew} and \ref{insertnew2}, as well as
Cauchy Schwartz inequality, and Banach algebra property of
$\| \cdot\|_r$ norm for $r \ge 1$
to complete the proof.
\end{proof}

\begin{lemma}\label{lem6.2}
Suppose for $r\geq2$,  $z\in H^r_p $
corresponds to ${\tilde \theta} \in \dot{H}^{r-1}$ through
(\ref{eqomega}) and (\ref{zomega}) and
$\| {\tilde \theta} \|_1<\epsilon_1$, where $\epsilon_1$ is small enough
for Propositions  \ref{prop2.6}  and \ref{prop3.4} to apply.
Further assume $|L - 2\pi | \le \frac{1}{2} $ and
$y(0,t)\in S_M$.  Then there exists $\Upsilon>0$ such that
if  $0\leq\beta<\Upsilon$, then
$\mathcal{K}[z]:H^0_p \rightarrow
H^{r-2}_p $, and in particular, there are
positive constants $C_1$ depending on $r$ only such that
\begin{equation}
\label{6.4} \|\mathcal{K}[z]f\|_{r-2} \leq C_1 \|f\|_0(1+\beta^2)
{(1+\|\omega_\alpha\|_{r-1}^{r-2})}.
\end{equation}
Further, $\mathcal{K}[z]: H^1_p
\rightarrow H^{r-1}_p $, and
\begin{equation}
\label{6.5} \|\mathcal{K}[z]f\|_{r-1} \leq C_1\|f\|_1(1+\beta^2)
{(1+\|\omega_\alpha\|_{r-1}^{r-1})}.
\end{equation}
\end{lemma}
\begin{proof}
We will deal with $\mathcal{K}_1$ and $\mathcal{K}_2$ separately. By
Lemma 6 in \cite{AD1}, we have
\begin{eqnarray}\label{6.6}
\|\mathcal{K}_1[z]f\|_{r-2} &\leq& {C_1}\|f\|_0
{(1+\|\omega_\alpha\|_{r-1}^{r-2})},\\
\label{6.7} \|\mathcal{K}_1[z]f\|_{r-1} &\leq &{C_1}\|f\|_1
{(1+\|\omega_\alpha\|_{r-1}^{r-1})},
\end{eqnarray}
where the positive constants $C_1$   both  depend on $r$.

Now consider $D^{r-1}_\alpha \mathcal{K}_2[z] f$, given by:
\begin{multline}\label{6.8}
\frac{1}{2\pi
i}\int_{\alpha-\pi}^{\alpha+\pi}f(\alpha')D^{r-1}_\alpha\Big\{\frac\beta4
l_1\big(\frac{1}4\beta({z(\alpha)-z(\alpha')})\big)-\frac\beta4\tanh\big[\frac\beta4\big(z(\alpha)-z^{\ast}(\alpha')\big)\big]\Big\}d\alpha'\\
=\frac{1}{2\pi
i}\int_{\alpha-\pi}^{\alpha+\pi}f(\alpha')D^{r-1}_\alpha\frac\beta4
l_1\big(\frac{1}4\beta({z(\alpha)-z(\alpha')})\big)d\alpha'\\
-\frac{1}{2\pi
i}\int_{\alpha-\pi}^{\alpha+\pi}f(\alpha')D^{r-1}_\alpha\frac\beta4\tanh\Big\{\frac\beta4\big[\big(z(\alpha)-z(\alpha')\big)+2i\big(y(\alpha',t)-y(0,t)\big)+2iy(0,t)\big]\Big\}d\alpha'.
\end{multline}

Equation (\ref{6.8}) involves upto
$r-1$ derivative of $z$.  From (\ref{translating5.5}),
\begin{equation}\label{z3.21}
\big|z(\alpha)-z(\alpha')\big|= \frac{L}{2\pi} \Big|\int_{\alpha'}^\alpha
e^{i\zeta+i\theta(\zeta)}d\zeta\Big|
\le \frac{L}{2} < 2\pi
\end{equation}
\begin{equation}\label{zstar}
z(\alpha,t)-z^{\ast}(\alpha',t)=\big(z(\alpha,t)-z(\alpha',t)\big)+2i\big(y(\alpha',t)-y(0,t)\big)+2iy(0,t).
\end{equation}
From (\ref{z3.21}), (\ref{zstar}) and $|y(0,t)|<M$,
there exists $\Upsilon>0$ small enough so that if $0\leq\beta<
 \Upsilon<1$, then
 $\Big|\beta\big(z(\alpha)-z(\alpha')\big)\Big|\leq \pi$,
and
$\Big|\beta\big[\big(z(\alpha)-z(\alpha')\big)+2i\big(y(\alpha',t)-y(0,t)\big)
+2iy(0,t)\big]\Big|< C\beta $.
Since $l_1$ and $\tanh$
analytic,  we conclude that
\begin{eqnarray}
\|\mathcal{K}_2[z]f\|_{r-1} \leq C_1\beta^2\|f\|_0
{\big(1+\|\omega_\alpha\|_{r-1}^{r-1}\big)},\label{6.9}
\end{eqnarray}
where $C_1$  depends only on  $r$.
Combining (\ref{6.6}), (\ref{6.7}) and (\ref{6.9}), we complete
the proof.
\end{proof}

\begin{note}\label{mathcalK2} Note from (\ref{eqG1}) and (\ref{6.9}), for $r\ge1$ and $|L-2\pi|<\frac{1}{2}$, by Lemma \ref{lem3.1}, it follows that
\begin{eqnarray}\label{G23.23}
\|\mathcal{G}_2[z]f\|_{r-1} \leq C_1\beta^2\|f\|_0
\exp{\big(C_2\|\tilde\theta\|_{r-1}\big)},
\end{eqnarray}
where $C_1$ and $C_2$  depend only on  $r$.
\end{note}

\begin{lemma}(See Lemma 3.8 in \cite{JY})
\label{lem3.8} If $f \in H^1_p $,
and $\omega^{(1)}$,  $\omega^{(2)}$ correspond respectively to
$\tilde{\theta}^{(1)}$ and $\tilde{\theta}^{(2)}$, each in $
\dot{H}^{1}$, with $\|\tilde\theta^{(1)}\|_1$,
$\|\tilde\theta^{(2)}\|_1<\epsilon_1$, then for sufficient small
$\epsilon_1$,
\begin{eqnarray}
\|\mathcal{K}_1[\omega^{(1)}]f-\mathcal{K}_1[\omega^{(2)}]f\|_0 &\leq&
C_1
\|f\|_0
\|\tilde\theta^{(1)}-\tilde\theta^{(2)}\|_1.\nonumber
\end{eqnarray}
Suppose  ${\tilde \theta}^{(1)}, {\tilde \theta}^{(2)} \in \dot{H}^r$. Then
for $r \ge 1$,
\begin{eqnarray*}
&&\|\mathcal{K}_1[\omega^{(1)}]f-\mathcal{K}_1[\omega^{(2)}]f\|_r \\
&\leq& C_1\exp\Big({C_2}\big(\|\tilde\theta^{(1)}\|_{r}+
\|\tilde\theta^{(2)}\|_{r}\big)\Big)
\|\tilde\theta^{(1)}-\tilde\theta^{(2)}\|_r\|f\|_1,
\end{eqnarray*}
while for
for $r\ge 3$,
\begin{eqnarray*}
&&\|\mathcal{K}_1[\omega^{(1)}]f-\mathcal{K}_1[\omega^{(2)}]f\|_r \\
&\leq& C_1\exp\Big({C_2}\big(\|\tilde\theta^{(1)}\|_{r-1}+
\|\tilde\theta^{(2)}\|_{r-1}\big)\Big)\Big(\big(\|\tilde\theta^{(1)}\|_r+\|\tilde\theta^{(2)}\|_r\big)\|\tilde\theta^{(1)}-\tilde\theta^{(2)}\|_{r-1}\\
&&\hspace{6cm}+
\|\tilde\theta^{(1)}-\tilde\theta^{(2)}\|_{r}\Big)\|f\|_1,
\end{eqnarray*}
where constants $C_1$ and $C_2$ depend on $r$ only.
\end{lemma}

\medskip

\begin{lemma}\label{lem6.3}
Let $0\leq\beta<\Upsilon$.
Let  $f \in H^1_p $,  and $z^{(1)}$, $z^{(2)}$
correspond respectively
to $\big(\tilde\theta^{(1)}, L^{(1)}(t), \hat\theta^{(1)}(0;t)\big)$
and $\big(\tilde\theta^{(2)}, L^{(2)}(t), \hat\theta^{(2)}(0;t)\big)$
(see (\ref{zomega})).
Further,  assume $\|\tilde\theta^{(1)}\|_1$,
$\|\tilde\theta^{(2)}\|_1<\epsilon_1$,
$|L^{(1)}-2\pi|<\frac{1}{2}$, $|L^{(2)}-2\pi|<\frac{1}{2}$
and $y^{(1)}(0,t)=\Im z^{(1)}(0,t) $, $ y^{(2)}(0,t)=\Im z^{(2)}(0,t)$
belong to $ S_M$.
Then for $\epsilon_1$ and $\Upsilon$ small enough for Proposition
\ref{prop2.6} and Lemma
\ref{lem6.2} to apply, there exists
constant $C_1$ depending only on $r$ so that
\begin{equation*}
\|\mathcal{G}_2[z^{(1)}]f-\mathcal{G}_2[z^{(2)}]f\|_0 \leq C_1\beta^2
\|f\|_0\big(\|\theta^{(1)}-\theta^{(2)}\|_1
+|L^{(1)}(t)-L^{(2)}(t)|+|y^{(1)}(0,t)-y^{(2)}(0,t)|\big).
\end{equation*}

If  ${\tilde \theta}^{(1)}, {\tilde \theta}^{(2)} \in \dot{H}^r$, then
for $r \ge 1$,
\begin{multline*}
\|\mathcal{G}_2[z^{(1)}]f-\mathcal{G}_2[z^{(2)}]f\|_r  \leq
C_1\beta^2\|f\|_1\exp\Big({C_2}\big(\|\tilde\theta^{(1)}\|_{r}+
\|\tilde\theta^{(2)}\|_{r}\big)\Big)
\Big(\|\theta^{(1)}-\theta^{(2)}\|_r\\
+|L^{(1)}(t)-L^{(2)}(t)|+|y^{(1)}(0,t)-y^{(2)}(0,t)|\Big),
\end{multline*}
for constants $C_1$ and $C_2$  depending on $r$ only.

Further, if  $L^{(1)}$ and $L^{(2)}$ correspond to the same area $V$ through (B.4), then
\begin{equation}
\label{6.14}
|L^{(1)}-L^{(2)} |\leq
C
\|\tilde\theta^{(1)}-\tilde\theta^{(2)}\|_1,
\end{equation}
with $C$  depending on area $ V$ alone.
\end{lemma}

\begin{proof} Note Definition \ref{defG1}.
The first part of the proof uses the regularity of
functions $l_1$ and $\tanh$ away from the poles and uses
(\ref{zomega}) and Lemma \ref{lem3.1};
the second part uses
(B.4) and Lemma \ref{lem3.1}, taking into account the implied lower bound
in (\ref{ImomBound}) for $\| {\tilde \theta} \|_1 < \epsilon_1$. See
\cite{thesis} for more details.
\end{proof}

\begin{lemma}{(See Lemma 8 in \cite{AD1})}
\label{lem3.9}
For  $\psi\in H^r_p $ with $r\geq 1$, the
operator $[\mathcal{H},\psi]$ is bounded from
$H^0_p $ to
$H^{r-1}_p $. And we have
\begin{eqnarray*}
\|[\mathcal{H},\psi]f\|_{r-1} \leq C \|f\|_0 \|\psi\|_r,
\end{eqnarray*}
 where $C$ depends on $r$.
\end{lemma}

\medskip

\begin{lemma}{(See Lemma 3.10 in \cite{JY})}
\label{lem3.10}
 For $r>\frac{1}{2}$
and  $\psi \in H^r_p $, the operator
 $[\mathcal{H},\psi]$ is bounded from $H^1_p $
to $H^r_p $, and
 \begin{eqnarray*}
 \|[\mathcal{H},\psi]f\|_r\leq C \|f\|_1\|\psi\|_r,
 \end{eqnarray*}
 where $C$ depends on $r$.
 \end{lemma}

 \begin{lemma}
\label{lem3.11}  Assume $0\leq\beta<\Upsilon$,
$f \in H^1_p $ and let $z^{(1)}$ and $z^{(2)}$
correspond respectively to
$\big(\tilde\theta^{(1)}, L^{(1)}(t), \hat\theta^{(1)}(0;t)\big)$
and $\big(\tilde\theta^{(2)}, L^{(2)}(t), \hat\theta^{(2)}(0;t)\big)$
(see (\ref{zomega})). Further,  assume $\|\tilde\theta^{(1)}\|_1$,
$\|\tilde\theta^{(2)}\|_1<\epsilon_1$,  $|L^{(1)}-2\pi|<\frac{1}{2}$,
$|L^{(2)}-2\pi|<\frac{1}{2}$ and $y^{(1)}(0,t)=\Im z^{(1)}(0,t) $,
$ y^{(2)}(0,t)=\Im z^{(2)}(0,t)$ belong to $ S_M$.
Then for sufficient small $\epsilon_1$ and $\Upsilon$
so that Proposition \ref{prop2.6} and
and Lemmas
\ref{lem6.2} and \ref{lem6.3} apply, there exists constants $C_1$ so that
\begin{multline*}
\|\mathcal{G}[z^{(1)}]f-\mathcal{G}[z^{(2)}]f\|_0 \leq C_1
\|f\|_0
\Big\{\|\tilde\theta^{(1)}-\tilde\theta^{(2)}\|_1
+\beta^2 \Big[|L^{(1)}(t)-L^{(2)}(t)|+\|\theta^{(1)}-\theta^{(2)}\|_1\\
+|y^{(1)}(0,t)-y^{(2)}(0,t)|\Big]\Big\},
\end{multline*}

\medskip

Furthermore, if ${\tilde \theta}^{(1)}, {\tilde \theta}^{(2)} \in \dot{H}^r$,
then for $r \ge 1$,
\begin{multline*}
\|\mathcal{G}[z^{(1)}]f-\mathcal{G}[z^{(2)}]f\|_r  \leq
C_1\|f\|_1\exp\Big({C_2}\big(\|\tilde\theta^{(1)}\|_{r}+
\|\tilde\theta^{(2)}\|_{r}\big)\Big)
\Big\{\|\tilde\theta^{(1)}-\tilde\theta^{(2)}\|_r\\
+\beta^2\Big[|L^{(1)}(t)-L^{(2)}(t)|+\|\theta^{(1)}-\theta^{(2)}\|_r+|y^{(1)}(0,t)-y^{(2)}(0,t)|\Big]\Big\},
\end{multline*}
where the constants $C_1$ and $C_2$  depend on $r$.\end{lemma}

\begin{proof} The proof follows from
Lemmas \ref{lem3.8}, \ref{lem6.3} and \ref{lem3.10}, once we
note the relation (\ref{1.10}).
\end{proof}

\begin{prop}
\label{propgamma}
Assume $0 \le \beta < \Upsilon$, $z$
corresponds to $\big(\tilde\theta, L(t), \hat\theta(0;t)\big)$
through (\ref{eqomega}), (\ref{zomega})
for $r\ge3$ with $\tilde\theta\in\dot{H}^{r}$.  Further assume
$\|\tilde \theta\|_1<\epsilon_1$, $|L-2\pi|<\frac{1}{2}$ and
$y(0,t)=\Im z(0,t)$ belongs to $S_M$, and $|u_0|<1$.
Then for sufficiently small
$\epsilon_1$ and $\Upsilon$
(so that Proposition \ref{prop2.6} and
Lemmas
\ref{lem6.2} and \ref{lem6.3} apply),
there
exists unique solution $ \gamma \in \{u\in H^{r-2}_p
|\hat{u}(0)=0\}$ satisfying  (B.3).
For constants $C_0$ and $C$,  solution $\gamma$ satisfy estimates
\begin{eqnarray*}
\| \gamma \|_0  &\leq&
C_0\big(\sigma \|\tilde \theta \|_2+1 \big),\nonumber \\
\| \gamma \|_{1} &\le& C \big ( \sigma \| {\tilde \theta } \|_{3}
+ 1+\|\tilde\theta\|_{0}  \big).
\end{eqnarray*}

Let  $z^{(1)}$ and $z^{(2)}$
correspond respectively to
$\big(\tilde\theta^{(1)}, L^{(1)}(t), \hat\theta^{(1)}(0;t)\big)$
and $\big(\tilde\theta^{(2)}, L^{(2)}(t), \\
\hat\theta^{(2)}(0;t)\big)$
(see (\ref{zomega})). Further assume  $\|\tilde\theta^{(1)}\|<\epsilon_1$,
$\|\tilde\theta^{(2)}\|<\epsilon_1$, $|L^{(1)}-2\pi|<\frac{1}{2}$,
$|L^{(2)}-2\pi|<\frac{1}{2}$ and $y^{(1)}(0,t)=\Im z^{(1)}(0,t) $,
$ y^{(2)}(0,t)=\Im z^{(2)}(0,t)$ belong to $ S_M$.
Then for sufficient small $\epsilon_1$ and $\Upsilon$, the corresponding
$\gamma^{(1)}$ and $\gamma^{(2)}$ determined from (B.3)
satisfies
\begin{equation}
\label{eq1prop319}
\|\gamma^{(1)}-\gamma^{(2)}\|_{0}\leq  C\Big(\| {\theta}^{(1)}-
{ \theta}^{(2)} \|_{2}+|L^{(1)}(t)-L^{(2)}(t)|+
\beta^2|y^{(1)}(0,t)-y^{(2)}(0,t)| \Big).
\end{equation}
Further, if  ${\tilde \theta}^{(1)}, {\tilde \theta}^{(2)} \in
\dot{H}^r$,
then the corresponding $(\gamma^{(1)}, U^{(1)}, T^{(1)})$ and
$(\gamma^{(2)}, U^{(2)}, T^{(2)})$ determined from (B.3),
(\ref{1.12}) and (B.5) satisfy
\begin{multline}
\label{eq2prop319}
\|\gamma^{(1)}-\gamma^{(2)}\|_{r-2}\leq  C_1\exp
\big(C_2(\|\tilde\theta^{(1)}\|_r+\|\tilde\theta^{(2)}\|_r)\big)
\Big(\|  \theta^{(1)}-
{\theta}^{(2)} \|_{r}\\+|L^{(1)}(t)-L^{(2)}(t)|+
\beta^2|y^{(1)}(0,t)-y^{(2)}(0,t)|
\Big),
\end{multline}
\begin{multline}\label{U3.32}
\|U^{(1)}-U^{(2)}\|_{r-2}\leq
C_1\exp\big(C_2(\|\tilde\theta^{(1)}\|_r+\|\tilde\theta^{(2)}\|_r)\big)
\Big(\| {\ \theta}^{(1)}-
{\theta}^{(2)} \|_{r}\\+|L^{(1)}(t)-L^{(2)}(t)|
+\beta^2|y^{(1)}(0,t)-y^{(2)}(0,t)| \Big),
\end{multline}
\begin{multline}\label{T3.33}
\|T^{(1)}-T^{(2)}\|_{r-1}\leq   C_1
\exp\big(C_2(\|\tilde\theta^{(1)}\|_r+\|\tilde\theta^{(2)}\|_r)\big)
\Big(\| {\ \theta}^{(1)}-
{\theta}^{(2)} \|_{r}\\+|L^{(1)}(t)-L^{(2)}(t)|
+\beta^2|y^{(1)}(0,t)-y^{(2)}(0,t)|\Big),
\end{multline}
where $C_1$ and $C_2$ depend on  $r$ only.
\end{prop}

\begin{proof}
Since
$\mathcal{F}_1 [\omega_0] \gamma =\hat{\gamma}(0)=0$
(see \ref{eqG1} and Note \ref{noteg1}), (B.3) implies
\begin{equation}
\label{gammainv}
\left [I + a_\mu \left ( \mathcal{F} [z] -\mathcal{F}_1 [\omega_0]
\right ) \right ] \gamma = \frac{2\pi\sigma}{L}
\theta_{\alpha \alpha}+\frac{L}{\pi}\big(1+\frac{\mu_2 u_0}{\mu_1+\mu_2}\big)
\sin\big(\alpha+\theta(\alpha)\big).
\end{equation}
Therefore, if $\tilde\theta \in \dot{H}^2$, then by Notes \ref{noteg1} and \ref{mathcalK2},
Lemma  \ref{lem3.11} (note that Lemma \ref{lem3.11}
still holds for $\mathcal{G}_1$.)
imply
\begin{multline}\label{fgamma}
 \| \mathcal{F} [z] \gamma -\mathcal{F}_1 [\omega_0 ] \gamma \|_0\leq \|\mathcal{G}_2[z]\gamma\|_0+
 \|\mathcal{G}_1[z]\gamma-\mathcal{G}_1[\omega_0]\gamma\|_0\\
\le C_1\big(\| {\tilde \theta} \|_1+C_2\beta^2\big) \|\gamma \|_0,
\end{multline}
So, for sufficiently small $\epsilon_1$ and $\Upsilon>0$, if $\|\tilde
\theta\|_1 \leq \epsilon_1$ and $0\leq\beta<\Upsilon$, then
$$ \left [1 + a_\mu \left ( \mathcal{F} [z] -\mathcal{F}_1
[\omega_0]  \right ) \right ]^{-1} $$ exists and is bounded independent
of any parameters. Therefore, it follows from (\ref{gammainv}) that
$$ \| \gamma \|_0
\le C_0( \sigma \|\tilde \theta \|_2+ 1).$$
Further, by Note \ref{mathcalK2} and Lemma \ref{lem3.11} again, we have
$$ \| \mathcal{F} [ z]\gamma  - \mathcal{F}_1 [\omega_0] \gamma \|_{r-2}
\le C_1\big(\exp(C_2\|\tilde\theta\|_{r-2}) \| {\tilde \theta} \|_{r-2}
+\beta^2\exp(C_2\|\tilde\theta\|_{r-2})\big)
\| \gamma \|_1,  $$ where $C_1$ and $C_2$ depend only on $r$. Therefore,
for $r \ge 3$, it follows from  (B.3) that
\begin{multline}\label{gamma1} \| \gamma \|_{r-2} \le
C \sigma \| {\tilde \theta } \|_{r}+C
\big(1+\|\tilde\theta\|_{r-3}\big)\\
+ C_1\big(\exp(C_2\|\tilde\theta\|_{r-2}) \| {\tilde \theta} \|_{r-2}
+\beta^2\exp(C_2\|\tilde\theta\|_{r-2})\big)
\| \gamma \|_1
\end{multline} which $C$, $C_1$ and $C_2$ depend on  $r$, which implies
for sufficiently small $\epsilon_1$ and $\Upsilon$ that
\begin{equation}\label{gammar-2}
 \| \gamma \|_{1} \le C \sigma \| {\tilde \theta } \|_{3}+
C \big(1+\|\tilde\theta\|_{0}\big).
 \end{equation}
From  (B.3), we obtain
\begin{multline*}
 \| \gamma^{(1)} - \gamma^{(2)} \|_{r-2}  \le
C\Big(\frac{|L^{(1)}-L^{(2)}|}{L^{(1)} L^{(2)}}+
\frac{1}{L^{(2)}} \big(\| \tilde\theta^{(1)}- \tilde\theta^{(2)} \|_{r}
+|\hat{\theta}^{(1)}(0;t)-\hat{\theta}^{(2)}(0;t)|\big)\\
+\left\|\mathcal{F}[z^{(1)}]\gamma^{(1)}-\mathcal{F}[z^{(2)}]\gamma^{(2)}
\right\|_{r-2},
\end{multline*}
and using Lemma \ref{lem3.11}, we have
\begin{multline}
\label{mathcalF} \| \mathcal{F} [z^{(1)} ] \gamma^{(1)} -
\mathcal{F} [z^{(2)} ] \gamma^{(2)} \|_{r-2} \le  \left \|
\mathcal{F} [z^{(1)}] (\gamma^{(1)}-\gamma^{(2)})
- \mathcal{F}_1 [\omega_0 ] (\gamma^{(1)}- \gamma^{(2)}) \right  \|_{r-2} \\
+ \left\| \mathcal{F} [z^{(1)} ] \gamma^{(2)} - \mathcal{F}
[z^{(2)} ]\gamma^{(2)}
\right\|_{r-2} \\
\le C_1\exp\big(C_2\|\tilde\theta^{(1)}\|_r\big)
\Big( \| {\tilde \theta}^{(1)} \|_{r-2}
+\beta^2\exp(C_2\|\tilde\theta^{(1)}\|_{r-2})\Big)\| \gamma^{(1)}-
\gamma^{(2)} \|_1\\ + \left\| \mathcal{F} [z^{(1)} ]
\gamma^{(2)} - \mathcal{F}
[z^{(2)} ]\gamma^{(2)}
\right\|_{r-2}
\end{multline}
with $C_1$ and $C_2$ depending on $r$.
Hence by Lemma \ref{lem3.11}  again, the
fourth  and fifth statements in the proposition follow.

From (\ref{1.12}), it follows that
\begin{multline*}
 \| U^{(1)} - U^{(2)} \|_{r-2}  =
\Big\|\frac{\pi}{L^{(1)}}\mathcal{H}[\gamma^{(1)}]-
\frac{\pi}{L^{(2)}}\mathcal{H}[\gamma^{(2)}]\Big\|_{r-2}
 +\Big\|\frac{\pi}{L^{(1)}}\mathcal{G}[z^{(1)}]\gamma^{(1)}
-\frac{\pi}{L^{(2)}}\mathcal{G}[z^{(2)}]\gamma^{(2)}\Big\|_{r-2}\\
+(|u_0|+1)\Big\|\cos\big(\alpha+\theta^{(1)}(\alpha)\big)-
\cos\big(\alpha+\theta^{(2)} (\alpha)\big)\Big\|_{r-2},
\end{multline*}
by Lemmas \ref{lem3.1} and \ref{lem3.11}, it is easy to obtain (\ref{U3.32}).

Also from (B.5), we have
\begin{equation*}
 \| T^{(1)} - T^{(2)} \|_{r-1}  \leq
\Big\|(1+\theta_{1,\alpha})U^{(1)}-(1+\theta_{2,\alpha})U^{(2)}\Big\|_{r-2},
\end{equation*}
by (\ref{U3.32}), we get (\ref{T3.33}).
\end{proof}

\section{Global existence for near-circular translating bubble
without side-walls ($\beta =0$)}

In this section, we consider bubble solutions in the absence of side walls
($\beta=0$) for near-circular initial shapes.
It is readily checked that a time-independent
solution that satisfies (B.1), (B.3)-(B.6)
is $\theta = 0$,
$\gamma = 2 \sin \alpha$, $u_0=0$, $V=\pi$\footnote{This is consistent,
as it must be, with our choice length scale
$L=L^{(s)} = 2\pi$ as the perimeter length of a steady bubble.}
this describes
a steady circular bubble translating
along the positive $x$-axis in the laboratory frame with speed
$2+u_0 = 2$.
The uniqueness of this steady state, at least locally in the neighborhood
of this solution, is established in a more general
context
in the steady state analysis of \S 5  for $\beta \ge 0$. Note in
that case steady bubbles are not circular
and move along the positive $x$-axis in the lab frame with speed
$2 + u_0 (\beta)$.

However, if we overlook the equation for ${\hat \theta}_t (0; t)$
which only affects parametrization $\alpha$ of the boundary,
the remaining equations in (B.1), (B.3)-(B.6)
are seen to be satisfied even for
$\theta = \theta^{(s)} \equiv {\hat \theta} (0; t)$,  $
\gamma = \gamma^{(s)} \equiv 2 \sin \left (\alpha + {\hat \theta} (0; t)
\right )$, with  $u_0 = 0$ and $V=\pi$.
Geometrically, this still
corresponds to the same translating steady circular bubble, despite
the time dependence of
${\hat \theta} (0; t)$ does not affect the circular
shape and
the normal speed $U=0$ at the interface, as it must be
in the frame of the
steady bubble.

In studying the time evolution of near-circular interface, it
 turns out to be more convenient to use the time-dependent
$\gamma^{(s)}$ and define
a perturbed vortex sheet strength
$
\Gamma (\alpha, t) \equiv \gamma (\alpha, t) - \gamma^{(s)}(\alpha,t)$.

Using (B.3) and the property $\mathcal{G} [\omega_0] \gamma^{(s)} =0$ (see
Note \ref{noteg1}), it follows that
\begin{multline}
\label{eqGamma}
(I+a_\mu\mathcal{F}[\omega])\Gamma=-a_\mu
\left [\mathcal{F}[\omega]\gamma^{(s)}
- \mathcal{F} [\omega_0 ] \gamma^{(s)} \right ]
+\frac{2\pi-L}{L}\sigma\theta_{\alpha\alpha}+\sigma\theta_{\alpha\alpha}\\
+\frac{L-2\pi}{\pi}\sin\big(\alpha+\theta\big)+2\Big(\sin\big(\alpha+\theta)
-\sin\big(\alpha+\hat\theta(0;t)\big)\Big).
\end{multline}
Further,
from expression for $\gamma^{(s)}$ and property
$\mathcal{G}_1 [\omega_0] \gamma^{(s)} =0$ (see Note \ref{noteg1} ),
the normal velocity $U$ in (\ref{1.12}) for $\beta=0$
may be re-expressed as
\begin{equation}
\label{eqUnew}
U  = \frac{\pi}{L} \mathcal{H} [ \Gamma ]
+ \Re \left [ \frac{\pi}{L} \mathcal{G} [\omega] - \frac{1}{2}
\mathcal{G} [\omega_0] \gamma^{(s)} \right ]
+ \cos (\alpha + \theta ) - \cos (\alpha + {\hat \theta} (0; t).
\end{equation}

\begin{prop}
\label{translatingprop6.1}
If  $\tilde\theta\in\dot{H}^r$ with
$\|\tilde\theta\|_1<\epsilon_1$  and $|\hat\theta(0;t)|<\infty$,
then for sufficiently small $\epsilon_1$,
there
exists a unique solution $ \Gamma \in \{u\in H^{r-2}_p
|\hat{u}(0)=0\}$ for
$r\ge 3$ satisfying (\ref{eqGamma}).
This solution $\Gamma$ satisfies the estimates
\begin{eqnarray}
\| \Gamma \|_0  &\leq&  C\|
\tilde\theta \|_2,\label{Gamma1} \\
\|\Gamma\|_{r-2}&\leq &C_1\exp(C_2\|\tilde\theta\|_{r-2})
\|\tilde\theta\|_r, \label{Gamma2} \\
\left\| \Gamma - \sigma\frac{2 \pi}{L} \theta_{\alpha \alpha} \right\|_{r-2}
&\leq & C_1 \exp \left ( C_2 \| {\tilde \theta} \|_{r-2} \right )
\| {\tilde \theta} \|_{r-1}, \label{Gamma3}
\end{eqnarray}
where  $C_1$ and  $C_2$ depend only on $r$.

Let $\Gamma^{(1)} $  and $\Gamma^{(2)} $ correspond to $({\tilde
\theta}^{(1)}, \hat\theta^{(1)}(0;t) )$ and $({\tilde \theta}^{(2)} , \hat\theta^{(2)}(0;t) ) $ respectively. Assume $\|\tilde\theta^{(1)}\|_1<\epsilon_1$ and $\|\tilde\theta^{(2)}\|_1<\epsilon_1$. If $\tilde\theta^{(1)},\tilde\theta^{(2)}\in\dot{H}^r$ with $r\ge3$,   then for sufficient small $\epsilon_1$,
\begin{equation}\label{gamma4.6}
 \| \Gamma^{(1)} - \Gamma^{(2)} \|_{r-2} \le C_1\exp\big(C_2(\|\tilde\theta^{(1)}\|_r+\|\tilde\theta^{(2)}\|_r)\big)
\Big(\big \| {\tilde \theta}^{(1)} - {\tilde \theta}^{(2)}\big \|_r+\big|\hat{\theta}^{(1)}(0;t)-\hat{\theta}^{(2)}(0;t)\big|\Big),
\end{equation}
where  $C_1$ and $C_2$ depend on $r$ alone.
\end{prop}

\begin{proof}
In statements (\ref{eq1prop319}) and (\ref{eq2prop319})
in Proposition \ref{propgamma}, we take $\beta=0$,
$\gamma^{(2)} = \gamma$, ${\tilde \theta}^{(1)} =
{\tilde \theta}$, $L^{(1)} = L$,
$$\gamma^{(2)} = \gamma^{(s)} =
2 \sin \left (\alpha + {\hat \theta} (0; t) \right ), ~~{\tilde \theta}^{(2)}
= 0, ~~L^{(2)} = 2 \pi
$$
and use Lemma  \ref{lem3.11} to obtain statements
(\ref{Gamma1}) and (\ref{Gamma2}).

(\ref{eqGamma}) can be written as
\begin{equation*}
\Gamma-\sigma \frac{2\pi}L \theta_{\alpha\alpha}=-a_\mu
\left [\mathcal{F}[\omega]\gamma
- \mathcal{F} [\omega_0 ] \gamma\right ]
+\frac{L-2\pi}{\pi}\sin\big(\alpha+\theta\big)+2\Big(\sin\big(\alpha+\theta)
-\sin\big(\alpha+\hat\theta(0;t)\big)\Big).
\end{equation*}
Hence, by Lemma \ref{lem3.11} with $\beta=0$, Lemmas \ref{lem6.3} and \ref{coroXi} (see Note \ref{notecoroXi}), we obtain (\ref{Gamma3}).

The statement (\ref{gamma4.6}) follows in a similar
manner from (\ref{eq2prop319}).

\end{proof}

\medskip

When there is no side wall effect ($\beta=0$),
it is readily checked from (B.1), (B.3)-(B.6)
that $y(0,t)$\footnote{We ignored
in all cases $x(0, t) = \Re z (0, t)$ which
does not affect the evolution of the shape function $\theta$.}
does not affect the evolution of
${\tilde \theta}$ or ${\hat \theta} (0; t)$.
So, in this section we will ignore (B.2) all together, since translations
do not affect the shape and if
necessary, $y(0, t)$ can be calculated from
(B.2) at the end.

The main result in this section is the following proposition:
\begin{prop}\label{translationprop5.2}{ For $\sigma>0$,
there exists $\epsilon>0$ such that for
$r\geq3$, if $\|\mathcal{Q}_1\theta_0\|_r<\epsilon$,
 then there
exists a unique solution
$\big(\tilde\theta,\hat\theta(0;t)\big)\in C\big([0,\infty),\dot{H}^r\times
\mathbb{R}\big)$  to the Hele-Shaw  problem
(B.1), (B.3)-(B.6) satisfying initial conditions
(\ref{translating5.1}).
Further, $\|\tilde\theta\|_r$, $|\hat\theta (\pm1;t)|$ and $|L-2\pi|$  each
decay exponentially as $t \rightarrow \infty$, $|\hat\theta(0;t)|$
remains finite. Thus the circular translating steady bubble is
asymptotically stable for sufficiently small initial disturbances in the
$H^r_p $ space.}
\end{prop}

\begin{note}\label{translatingnote4.3} Proof of Proposition \ref{translationprop5.2} is given at the end of \S 4. Note also
Proposition \ref{translationprop5.2} and Lemma \ref{lem2.7} imply  Theorem \ref{theo1.3}.
\end{note}

\subsection{{\it A priori} estimates}

Before we consider  global solutions
to the system (B.1), (B.3)-(B.6)
for initial condition (\ref{translating5.1}). First some
additional estimates are needed for the terms that arise in
the evolution equations.

\begin{definition}\label{translatingdef4.4}
We define operator $\mathfrak{W}$ so that
\begin{equation*}
\mathfrak{W}[f](\alpha)=\frac{1}{2\pi
}\int_{0}^{2\pi}\gamma^{(s)}(\alpha')\frac{\int_{\alpha'}^{\alpha}\mathcal{Q}_0\big(f(\zeta)\omega_{0_\zeta}(\zeta)\big)d\zeta}{\omega_0(\alpha)
-\omega_0(\alpha')}d\alpha'.
\end{equation*}
\end{definition}

\begin{lemma}
\label{lemmaW}
For $f \in H_p^{k}$, there exists constant $C_1$ only dependent on $k$ so that
$$\| \mathfrak{W} [f] \|_k \le C_1 \| f \|_{k} $$
\end{lemma}
\begin{proof}
We take
$\omega_{0_\alpha} $ and
$Q_0 \left [f \omega_{0_\alpha} \right ] $
to be
$g(\alpha)$ and $f(\alpha)$ in Lemma
\ref{coro3.11} respectively and define $ h =\gamma^{(s)}$.
Note that for this choice, the condition
${\hat g} (0) = 0
={\hat f} (0)$ as well as the
lower bound constraint on
$g = \omega_{0,\alpha} = e^{i \alpha}$ is satisified.
The proof follows since $\| . \|_{L^\infty}$ bounds in $\alpha$
on $ D_\alpha^j \mathfrak{W} [f] $ imply
$\| . \|_j $ bounds in the Lemma statement.
\end{proof}

From (\ref{eqGamma}), after some algebraic manipulation, it follows that
\begin{equation}
\label{gamma4.3}
\Gamma(\alpha,t)=\frac{2\pi}{L}\sigma\theta_{\alpha\alpha}+\Gamma_L(\alpha,t)
+{N}_1(\alpha,t)+{N}_2(\alpha,t)+{N}_3(\alpha,t),
\end{equation}
where
\begin{multline}\label{LGamma}
\Gamma_L(\alpha,t)=2\mathcal{Q}_0\theta(\alpha,t)\cos\big(\alpha+\hat\theta(0;t)\big)+\frac{L-2\pi}{\pi}\sin\big(\alpha+\hat\theta(0;t)\big)\\
-a_\mu\Re\Big(\frac{\partial}{\partial\alpha}\big\{\mathfrak{W}[\mathcal{Q}_0\theta](\alpha)\big\}\Big),
\end{multline}
\begin{multline}
\label{N14.4}
N_1
=
a_\mu\Re\Big(-\frac{1}{i}
\mathcal{G}[\omega]\Gamma+\frac{1}{i}
\mathcal{G}[\omega_0]\Gamma\Big)
+\frac{L-2\pi}{\pi}\big(\sin(\alpha+\theta)-\sin(\alpha+\hat\theta(0; t))\big)\\
+a_\mu\Re \Big(
i(e^{i\mathcal{Q}_0\theta}-1)
\Big\{ \frac{\omega_{0_\alpha}}{\omega_\alpha}
\left [ \mathcal{G}[\omega]\gamma^{(s)} - \mathcal{G} [\omega_0] \gamma^{(s)}
\right ]
-2 \left (
\frac{\omega_{0_\alpha}}{\omega_\alpha} - 1 \right ) \cos \left (
\alpha + {\hat \theta} (0; t) \right )
\Big\}\Big) \\
+2 \Xi_{s} \left [ Q_0 \theta ; {\hat \theta} (0; t) \right ]
\end{multline}
\begin{equation}\label{N24.5}
N_2=-2a_\mu\Re\Big(\frac{1}{i}\frac{\partial}{\partial\alpha}\big\{\mathfrak{W}[\Xi_e[\mathcal{Q}_0\theta]](\alpha)\big\}\Big),
\end{equation}
and
\begin{multline}
\label{N34.6}
N_3 =
\Re
\left (
\frac{a_\mu \omega_{0_\alpha}}{i \pi \omega_\alpha}
\int_{\alpha-\pi}^{\alpha+\pi}
\gamma^{(s)} (\alpha') \frac{q_1 [\omega-\omega_0] (\alpha, \alpha')}{
q_1 [\omega_0] (\alpha, \alpha') } \left [
\frac{q_2 [\omega] (\alpha, \alpha')}{q_1 [\omega] (\alpha, \alpha')}
-
\frac{q_2 [\omega_0] (\alpha, \alpha')}{q_1 [\omega_0] (\alpha, \alpha')}
\right ] \right ) \\
+
\Re
\left (
\frac{a_\mu \omega_{0_\alpha} }{i \pi}
\left [ \frac{1}{\omega_\alpha} - \frac{1}{\omega_{0_\alpha}} \right ]
\int_{\alpha-\pi}^{\alpha+\pi} d\alpha'
\frac{\gamma^{(s)} (\alpha') q_1 [\omega-\omega_0] (\alpha, \alpha')
\omega_{0_\alpha} (\alpha) }{
q_1 [\omega_0 ] (\alpha, \alpha')
\left [\omega (\alpha) -\omega_0 (\alpha') \right ] } \right ).
\end{multline}

Further, from (\ref{1.12}) it follows that normal velocity
\begin{equation}\label{translating6.6}
U(\alpha,t)=\frac{2\pi^2}{L^2}\sigma\mathcal{H}(\theta_{\alpha\alpha})(\alpha)+U_L(\alpha,t)+\frac{1}{2}\mathcal{H}\Big(
N_1(\cdot)+N_2(\cdot)+N_3(\cdot)\Big)(\alpha)+N_4(\alpha),
\end{equation}
where
\begin{multline*}U_L(\alpha,t)
=\frac{1}{2}\mathcal{H}[\Gamma_L](\alpha,t)+\frac{L-2\pi}{L}\cos\big(\alpha+\hat\theta(0;t)\big)
-\mathcal{Q}_0\theta\sin\big(\alpha+\hat\theta(0;t)\big)\\
-\Re\Big(\frac{1}{i}\frac{\partial}{\partial\alpha}\big(\mathfrak{W}[\mathcal{Q}_0\theta](\alpha)\big)\Big),
\end{multline*}
and
\begin{align}\label{eqN4}
N_4 (\alpha)
=&
\Re\Big(\frac{\pi}{L}\mathcal{G}[\omega]\Gamma
-\frac{1}{2}\mathcal{G}[\omega_0]\Gamma
\Big)
+\frac{2\pi-L}{L}\Re\Big(\frac{1}{2}
\left [ \mathcal{G}[\omega]\gamma^{(s)} - \mathcal{G} [\omega_0] \gamma^{(s)}
\right ] \Big)  \\
&+\Re\left((e^{i\mathcal{Q}_0\theta}-1)
\left\{\frac{\omega_{0_\alpha}}{2\omega_\alpha}
\big(\mathcal{G}[\omega]\gamma^{(s)} -
\mathcal{G} [\omega_0] \gamma^{(s)} \big) -
\left ( \frac{\omega_{0_\alpha}}{\omega_\alpha} -1 \right )
\cos\big(\alpha+\hat\theta(0;t)\big)\right\}\right) \nonumber\\
&-
\Re\Big(\frac{\omega_{0_\alpha}}{2\pi
}\mbox{PV}\int_{\alpha-\pi}^{\alpha+\pi}\gamma^{(s)}(\alpha')
\frac{q_1[\omega-\omega_0](\alpha,\alpha')}{q_1[\omega_0](\alpha,\alpha')}
\Big(\frac{1}{\omega(\alpha)-\omega(\alpha')}
-\frac{1}{\omega_0(\alpha)-\omega_0(\alpha')}
\Big)d\alpha'\Big) \nonumber\\
&+\frac{2\pi-L}{2L}\mathcal{H}[\Gamma
-\frac{2\pi}{L}\sigma\theta_{\alpha\alpha}]
+\Re\Big(\frac{\partial}{\partial\alpha}
\big(\mathfrak{W}[\Xi_e[\mathcal{Q}_0\theta]](\alpha)\big)\Big)
+\Xi_c \left [ Q_0 \theta; {\hat \theta} (0; t) \right ]
\nonumber \end{align}

Using (\ref{translating6.6}) and (B.5), from (B.1)
we obtain
\begin{equation}
\label{RHSC1}
{\tilde \theta}_t =
\frac{2\pi}{L}\mathcal{Q}_1\big(U_\alpha+T(1+\theta_\alpha)\big)\\
=\mathcal{A} [{\tilde \theta} ] (\alpha,t)  + \mathcal{A}_N [ {\tilde \theta},
{\hat \theta} (0; \cdot), L ](\alpha,t) +
\mathfrak{N}\big[\tilde\theta,\hat{\theta}(0;\cdot)\big](\alpha,t),
\end{equation}
where the operators $\mathcal{A}$ and $\mathcal{A}_N$
acting on
real valued functions  ${\tilde \theta} \in {\dot H}^r $ for $r \ge 3$
are defined by
\begin{equation}
\label{Adef}
\mathcal{A} [ {\tilde \theta} ] (\alpha,t) =
\sum_{k=2}^\infty e^{i k \alpha} \left ( -\sigma d(k) {\hat \theta} (k; t)
+ m (k) {\hat \theta} (k+1; t) \right )
+ c.c.,
\end{equation}
\begin{multline}
\label{AdefN} \mathcal{A}_N [ {\tilde \theta}, {\hat \theta}
(0;\cdot), L] (\alpha,t) = \sum_{k=2}^\infty e^{i k \alpha} \left \{
\left ( \frac{-8 \pi^3}{L^3} +1 \right ) \sigma ~d(k)~{\hat \theta}
(k; t) + e^{-i {\hat \theta} (0; t) } \left (\frac{2\pi}{L} - 1
\right )
m (k) {\hat \theta} (k+1; t) \right \} \\
+ \left ( e^{-i {\hat \theta} (0; t)} - 1 \right )
\sum_{k=2}^\infty e^{ik\alpha}
m(k) {\hat \theta} (k+1; t)
+ c.c. ,
\end{multline}
where $c.c.$ indicates complex conjugate of explicitly shown terms on
the right side in each of (\ref{Adef}), (\ref{AdefN})
\footnote{Note that
while $L$ is shown as an independent argument of
$\mathcal{A}_N $, in the evolution equation (\ref{RHSC1}), itself,
$L$
is
determined from ${\tilde \theta}$ through (B.4) and (\ref{eqomega}).}
and
\begin{equation}
\label{dkmkdef}
d(k) = \frac{1}{2} k (k^2-1) ~~,~~m(k) = (1+a_\mu)
\frac{(k^2-1) (k+1)}{k (k+2)} ,
\end{equation}
and
\begin{equation}
\label{eqmathfrakN}
\mathfrak{N}
\big[\tilde\theta,\hat{\theta}(0;\cdot)\big](\alpha,t)=
\frac{2\pi}{L}\mathcal{Q}_1\Big\{\Big(\frac{1}{2}\mathcal{H}\Big(N_1(\cdot)+N_2(\cdot)+N_3(\cdot)\Big)(\alpha)
+N_4(\alpha)\Big)_\alpha+N_5(\alpha)\Big\},
\end{equation}
where
\begin{eqnarray}
\label{eqN5}
N_5(\alpha)&=&\int_0^\alpha
\Big[\frac{1}{2}\mathcal{H}\Big(N_1(\cdot)+N_2(\cdot)+N_3(\cdot)\Big)(\alpha')
+N_4(\alpha')\Big]d\alpha'\\
&& -\frac{\alpha}{2\pi}\int_0^{2\pi}
\Big[\frac{1}{2}\mathcal{H}\Big(N_1(\cdot)+N_2(\cdot)+N_3(\cdot)\Big)(\alpha)
+N_4(\alpha)\Big]d\alpha\nonumber\\
&&+\int_0^\alpha\theta_\alpha(\alpha')
U(\alpha')d\alpha'-\frac{\alpha}{2\pi}\int_0^{2\pi}\theta_\alpha(\alpha)U(\alpha)d\alpha
\nonumber\\
&&+\Big(\int_0^\alpha\theta_\alpha(\alpha')U(\alpha')d\alpha'-\frac{\alpha}{2\pi}\int_0^{2\pi}\theta_\alpha(\alpha)U(\alpha)d\alpha\Big)
\theta_\alpha(\alpha) .\nonumber
\end{eqnarray}
It is straightforward to check from (\ref{dkmkdef})
that for any $k \ge 2$,
\begin{equation}
\label{dkmkbound}
\frac{3}{8} k^3 \le d(k) \le \frac{1}{2} k^3 ~~,~~\frac{9}{16}
(1+a_\mu) k \le m (k) \le (1+a_\mu) k.
\end{equation}

After some calculation, we also find from
(B.1) that
\begin{equation}\label{translating6.9}
{\hat \theta}_t (0; t) =
\frac{1}{L}\int_0^{2\pi}T(\alpha,t)\Big(1+\theta_\alpha(\alpha,t)\big)
\Big)d\alpha = \mathfrak{N}_0 [ {\tilde \theta}, {\hat \theta} (0;
\cdot) ](t),
\end{equation}
where the functional $\mathfrak{N}_0$  of real valued $\big (
{\tilde \theta} (\alpha, t), {\hat \theta} (0; t) \big ) $ is
defined by\footnote{Note that the Fourier component ${\hat \theta}
(1; t)$ appearing in the summation is being determined indirectly
from ${\tilde \theta}$ through (B.6) (see Proposition
\ref{prop2.6}).}
\begin{multline}
\label{eqmathfrakN0}
\mathfrak{N}_0 \big[\tilde\theta,\hat{\theta}(0;\cdot)\big](t)
=\int_0^{2\pi}\int_0^\alpha\Big(\big(\frac{2\pi^2}{L^3}
-\frac{1}{4\pi}\big)\sigma\mathcal{H}(\theta_{\alpha\alpha})(\alpha')
+\big(\frac{1}{L}-\frac{1}{2\pi}\big)U_L(\alpha')\Big)d\alpha'd\alpha\\
-\pi\big(\frac{1}{L}-\frac{1}{2\pi}\big)\int_0^{2\pi}U_L(\alpha)d\alpha
+\frac{1}{L}\int_0^{2\pi}N_5(\alpha)d\alpha+B_0\big[\tilde\theta,\hat\theta(0;\cdot)\big](t),
\end{multline}
with the functional $B_0$ defined by
\begin{equation}
\label{eqB0} B_0 \big[{\tilde \theta},\hat\theta(0;\cdot)\big ](t)
=\sum_{k=1}^\infty \Big( \frac{\sigma k}{2} {\hat \theta}(k;t)
-e^{-i\hat\theta(0;t)} (1+a_\mu) \frac{k+1}{k(k+2)}
 \hat{\theta}(k+1;t)\Big)
+ c.c..
\end{equation}

With respect to the functional $B_0 [{\tilde \theta} (\alpha, t),
{\hat \theta} (0; t) ]$, the following statement readily follows.
\begin{lemma}
\label{lemB0} With ${\tilde \theta} \in {\dot H}_1 $ and $\| {\tilde
\theta} \|_2 < \epsilon$ sufficiently small, then
$$ \Big|
B_0 \big[{\tilde \theta} (\alpha, t) ,{\hat \theta} (0; t)\big ]
\Big | \le C \| {\tilde \theta} \|_2.
$$
Further $B_0^{(1)} $ and $B_0^{(2)} $ correspond to respectively to
$\big({\tilde \theta}^{(1)}, {\hat \theta}^{(1)} (0;t)\big )$ and
and $\big({\tilde \theta}^{(2)}, {\hat \theta}^{(2)} (0;t)\big )$,
then
\begin{equation*}
\big| B_0^{(1)} - B_0^{(2)} \big| \le C \left \{ \|{\tilde
\theta}^{(1)} (\cdot, t) - {\tilde \theta}^{(2)} (\cdot, t) \|_2 +
|{\hat \theta}^{(1)} (0; t) - {\hat \theta}^{(2)} (0; t) | \|
{\tilde \theta}^{(1)}(\cdot,t) \|_1 \right \}.
\end{equation*}
\end{lemma}
\begin{proof} The proof follows easily from the expression (\ref{eqB0}),
and Proposition \ref{prop2.6} relating ${\hat \theta} (1; t) $ to
${\tilde \theta}$.
\end{proof}

We have the following estimates for the nonlinear terms
$N_j$, $j=1,\cdots,5$:
 \begin{lemma}
\label{N1lem4.6}
If for $r\ge 3$, ${\tilde \theta} \in \dot{H}^r $ and
$\|\tilde\theta\|_1 <\epsilon_1$,
then for sufficiently small $\epsilon_1$,
$N_j$, $j=1,\cdots,5$,
defined by (\ref{N14.4}),
(\ref{N24.5}), (\ref{N34.6}), (\ref{translating6.6}), (\ref{eqN4})
and (\ref{eqN5}) satisfy
\begin{eqnarray}
\| N_j \|_{r-1}  &\leq&  C_1\exp(C_2\|\tilde\theta\|_{r-1}) \|
\tilde\theta \|_{r-1}\|\tilde\theta\|_r,\label{N14.10}
\end{eqnarray}
where $C_1$ and  $C_2$ depend only on $r$.
Further let $N_j^{(1)}$ and $N_j^{(2)}$ correspond to
$\big({\tilde \theta}^{(1)},\hat\theta^{(1)}(0;t)\big)$ and $\big({\tilde \theta}^{(2)},\hat\theta^{(2)}(0;t)\big)$ respectively,
each in $\dot{H}^r\times\mathbb{R}$ with $\| {\tilde \theta}^{(1)} \|_1$ and
$\| {\tilde \theta}^{(2)} \|_1 $ $ < \epsilon_1$.
Then for sufficiently small $\epsilon_1$,
\begin{multline}\label{N14.8}
\left\|N^{(1)}_j-N_j^{(2)} \right\|_{r-1}
\le C_1\exp{\Big(C_2\big(\| {\tilde
\theta}^{(1)}\|_{r-1} +
\|{\tilde \theta}^{(2)}\|_{r-1}\big)\Big)}\Big\{\big(\| {\tilde
\theta}^{(1)}\|_{r-1} + \|{\tilde \theta}^{(2)}\|_{r-1}\big) \\
\times\Big(\big\| {\tilde
\theta}^{(1)} - {\tilde \theta}^{(2)}\big \|_r
+\left|\hat\theta^{(1)}(0;t)-\hat\theta^{(2)}(0;t)\right| \Big )
+\big(\| {\tilde
\theta}^{(1)}\|_r + \|{\tilde \theta}^{(2)}\|_r\big) \big\| {\tilde
\theta}^{(1)} - {\tilde \theta}^{(2)}\big \|_{r-1}\Big\},
\end{multline}
 \end{lemma}
 \begin{proof}
For estimating $N_1$ we use Lemmas \ref{coroXi} (see Note \ref{notecoroXi}), \ref{lem3.11}, \ref{lem6.3}
(in particular (\ref{6.14}) for $L^{(1)}=L$, $L^{(2)}=2\pi$, the latter
corresponding to ${\tilde \theta}=0$)
and
Proposition \ref{translatingprop6.1}.
For $N_2$, we use Lemmas \ref{coroXi} (see Note \ref{notecoroXi}) and \ref{lemmaW}.
For $N_3$, we use Lemmas \ref{lem3.1}, \ref{insertnew}, \ref{insertnew2} and
\ref{leminsertnew4}
together with Cauchy-Schwartz inequality to get the desired bound.

For (\ref{LGamma}), by Lemmas \ref{lem6.3} and \ref{lemmaW}, we have
\begin{equation}
\label{estimateGammaL}
\|\Gamma_L\|_{r-3}\le C\|\tilde\theta\|_{r-3}.
\end{equation}
For $N_4$ we rely on (\ref{estimateGammaL}), Lemmas \ref{coroXi} (see Note \ref{notecoroXi}), \ref{lem6.3} (equation (\ref{6.14})
in particular), \ref{lem3.11}, \ref{lemmaW} and Proposition
\ref{translatingprop6.1}.
$N_5$ uses bounds similar to $N_j$ for $j=1,\cdots,4$ as well
as bounds on $U$ (In Proposition \ref{propgamma}, we choose $U^{(1)}=U$, $U^{(2)}=0$, $\tilde\theta^{(1)}=\tilde\theta$, $\tilde\theta^{(2)}=0$, and $L^{(1)}=L$, $L^{(2)}=0$ in (\ref{U3.32}) to get the bound of $U$.).
\end{proof}

\begin{corollary}
\label{mathfrakNcor}
If for $r\ge 3$, ${\tilde \theta} \in \dot{H}^r $ and
$\|\tilde\theta\|_1 <\epsilon_1$,
then for sufficiently small $\epsilon_1$,
the function $\mathfrak{N} $,
and the functional $\mathfrak{N}_0 $, defined in
(\ref{eqmathfrakN}) and
(\ref{eqmathfrakN0})
satisfy the following estimates
\begin{eqnarray}
\| \mathfrak{N} \|_{r-1}
&\leq&  C_1\exp( C_2\|\tilde\theta\|_r)
\| \tilde\theta \|_r \|\tilde\theta\|_{r+1},
\label{N14.10.new.1} \\
|\mathfrak{N}_0 |&\leq&C_1\exp(C_2\|\tilde\theta(\cdot,t)\|_3)
\|\tilde\theta(\cdot,t)\|_3^2 + C_1 \| {\tilde \theta}(\cdot,t) \|_2
.\nonumber
\end{eqnarray}
where $C_1$ and  $C_2$ depend only on $r$.
Further, let $\left (\mathfrak{N}^{(1)}, \mathfrak{N}_0^{(1)} \right ) $
and $\left (\mathfrak{N}^{(2)}, \mathfrak{N}_0^{(2)} \right ) $
correspond to
$\big({\tilde \theta}^{(1)},\hat\theta^{(1)}(0;t)\big)$ and $\big({\tilde \theta}^{(2)},\hat\theta^{(2)}(0;t)\big)$ respectively,
each in $\dot{H}^r\times\mathbb{R}$ with $\| {\tilde \theta}^{(1)} \|_1$ and
$\| {\tilde \theta}^{(2)} \|_1 $ $ < \epsilon_1$.
Then for sufficiently small $\epsilon_1$,
\begin{multline}\label{N14.8.new}
\left\|\mathfrak{N}^{(1)}-\mathfrak{N}^{(2)} \right\|_{r-1}
\le C_1\exp{\Big(C_2\big(\| {\tilde
\theta}^{(1)}\|_r +
\|{\tilde \theta}^{(2)}\|_r\big)\Big)}\Big\{\big(\| {\tilde
\theta}^{(1)}\|_r + \|{\tilde \theta}^{(2)}\|_r\big) \\
\times\Big(\big\| {\tilde
\theta}^{(1)} - {\tilde \theta}^{(2)}\big \|_{r+1}
+\left|\hat\theta^{(1)}(0;t)-\hat\theta^{(2)}(0;t)\right| \Big )
+\big(\| {\tilde
\theta}^{(1)}\|_{r+1} + \|{\tilde \theta}^{(2)}\|_{r+1}\big) \big\| {\tilde
\theta}^{(1)} - {\tilde \theta}^{(2)}\big \|_r\Big\},
\end{multline}
\begin{multline}\label{N14.8.new2}
\Big | \mathfrak{N}_0^{(1)}-\mathfrak{N}_0^{(2)} \Big | \le
C_1\exp{\Big(C_2\big(\| {\tilde \theta}^{(1)}\|_{3}+ \|{\tilde
\theta}^{(2)}\|_{3} \big) \Big)}  \Big\{\big\| {\tilde \theta}^{(1)}
- {\tilde \theta}^{(2)}\big \|_{3} +\| {\tilde
\theta}^{(1)}\|_{3}|\hat\theta^{(1)}(0;t)-\hat\theta^{(2)}(0;t)|
\Big\},
\end{multline}
where $C_1$ and $C_2$ depend on $r$.
\end{corollary}
\begin{proof}
On using Lemmas \ref{lemB0} and \ref{N1lem4.6}, the the proof
follows from the expressions of $\mathfrak{N} $ and $\mathfrak{N}_0
$ in terms of $N_1,\cdots, N_5$.
\end{proof}

\subsection{Weighted Sobolev Space and Estimates}

For any surface tension $\sigma$, we choose the integer $K$ by
\begin{enumerate}
\item if $\sigma\geq1$, then $K=2$;
\item if $0<\sigma<1$, then
$K\geq \sqrt{1+\frac{6}{\sigma}}$.
\end{enumerate}
We define the weight $w(\sigma, k)$ so that
\begin{equation}
\label{defwsk}
w(\sigma, k) = \sigma^{K-|k|} ~~{\rm for} ~2 \le |k| \le K (\sigma),
w (\sigma, k) = 1 ~~{\rm for }~|k| > K (\sigma).
\end{equation}

\medskip

\begin{definition}\label{transdef5.8}
 Let $r\geq 0$. We define a family of weighted Sobolev norm in $\dot{H}^r$ by
 \begin{equation}
\label{defnormwr}
 \|u\|_{w,r}^2=\sum_{k=2}^{\infty} w^2 (\sigma, k) |k|^{2r} |\hat{u}(k)|^2
 + \sum_{k=-2}^{-\infty} w^2 (\sigma, k) |k|^{2r} |\hat{u}(k)|^2,
\end{equation}
and the corresponding inner-product:
\begin{equation}
\label{definnerprod}
(v, u)_{w,r}
 =\sum_{k=2}^{\infty} w^2 (\sigma, k) |k|^{2r} {\hat v}^* (k) \hat{u}(k)
 +\sum_{k=-2}^{-\infty} w^2 (\sigma, k) |k|^{2r} {\hat v}^* (k) \hat{u}(k).
\end{equation}
\end{definition}

 \medskip

\begin{note}
If $u$ and $v$ are real valued, be in $\dot{H}^r$, the inner-product reduces to
\begin{equation}
\label{defrealinnerprod}
(v, u)_{w,r}
 =2 \Re \left [
\sum_{k=2}^{\infty} w^2 (\sigma, k) |k|^{2r} {\hat v}^* (k)
\hat{u}(k) \right ].
\end{equation}
\end{note}

 \begin{remark}
 It is clear that for any fixed $\sigma > 0$,
the two norms
 $\|\cdot \|_{w,r}$ and $\|\cdot \|_r$ are equivalent.
 \end{remark}

The following two lemmas involve useful inner product estimates
involving $\mathcal{A} $ and $\mathcal{A}_N $:
\begin{lemma}
\label{lemAinner}
For any $ r \ge 0$ and $v \in {\dot H}^{r+3/2} $,
$$ \left ( v,  -\mathcal{A} [{\tilde \theta} ] \right )_{w,r}
\ge \frac{15 \sigma}{64} \| v \|_{w, r+3/2}^2.
$$
\end{lemma}
\begin{proof}
It is convenient to define
$$ \delta= \sup_{k \ge 2} \frac{m(k) w(k, \sigma)}{
\sigma d^{1/2} (k) d^{1/2} (k+1) w(k+1,\sigma) }  .
$$
Since $(1+a_\mu) \le 2$, it is not difficult to conclude from
the explicit expressions of $d(k)$ and $m(k)$ that
in all cases, $\delta\le \frac{3}{8}$.
Then, it follows from Cauchy Schwartz inequality that
$$ \sum_{k=2}^\infty
k^{2r} w^{2k} (\sigma, k) m(k)
\Re \left\{( {\hat v}^* (k) {\hat v} (k+1) \right \}
\le \frac{3}{8} \sigma \sum_{k=2}^\infty k^{2r} w^{2k} (\sigma, k) d(k)
| {\hat v} (k)|^2.
$$
It follows that
$$ \left ( v, - \mathcal{A} [v ] \right )_{w,r}
\ge \frac{5 \sigma}{8} \sum_{k=2}^\infty k^{2r} w^{2k} (\sigma, k)
d(k)
| {\hat v} (k)|^2
\ge \frac{15 \sigma}{64} \| v \|_{w,r+3/2}^2 .
$$
\end{proof}

With respect to the operator $\mathcal{A}_N $, we have the following estimate:
\begin{lemma}
\label{ANnorm}
For $r \ge 3$, assume real $f, f_1, f_2
\in {\dot H}^{r}$ and
$a, a_1, a_2,L, L_1, L_2 $ are real numbers satisfying
constraint $|L- 2 \pi| \le \frac12$, $|L_j - 2 \pi| \le \frac12$ for $j=1,2$.
Then there exists constant $C_r$ only dependent on $r$ so that
\begin{equation*}
\| \mathcal{A}_N [ f , a,  L] \|_{w,r-3/2}
\le C_r \sigma \left ( |L-2\pi | \| f \|_{w, r+3/2}
+ |a |
\| f\|_{w, r-1/2} \right),
\end{equation*}
\begin{multline*}
\| \mathcal{A}_N [ f_1 , a_1,  L_1 ] -
\mathcal{A}_N [ f_2 ,a_2, L_2 ] \|_{w,r-3/2}
\le C_r \sigma \left ( |L_1 - L_2 | \| f_2 \|_{w, r+3/2}
+ |a_1-a_2 |
\| f_2 \|_{w, r-1/2} \right .\\
\left. + |L_1-2\pi| \| f_1 - f_2 \|_{w,r+3/2}
+ |a_1| \| f_1 - f_2 \|_{w,r-1/2} \right ).
\end{multline*}
\end{lemma}
\begin{proof}
From the definition of $\mathcal{A}_N$, it follows that
\begin{multline*}
\| \mathcal{A}_N [ f, a, L] \|^2_{w, r-3/2} \le
2 \Big | 1 - \frac{8\pi^3}{L^3} \Big |^2
\sum_{k=2}^\infty \sigma^2
k^{2r-3} d^2 (k) w^2 (k, \sigma) |{\hat f}(k) |^2 \\
+ 2\left ( \Big | 2 \sin \frac{a}{2} \Big |^2
+ \Big | \frac{2\pi}{L} - 1 \Big |^2 \right )
\sum_{k=2}^\infty  k^{2r-3} m^2 (k) w^2(k,\sigma)|{\hat f}(k+1) |^2\\
\leq C_r\sigma^2 \Big( | L-2\pi|^2\|f\|_{r+3/2}^2
+ (|a|^2+|L-2\pi|^2)\sup_{k\ge 2}\frac{m^2(k)w^2(k,\sigma)}{\sigma^2(k+1)^2w^2(k+1,\sigma)}
\|f\|_{r-1/2}^2\Big)\\
\leq C_r \Big( | L-2\pi|^2\|f\|_{r+3/2}^2
+ |a|^2
\|f\|_{r-1/2}^2\Big).\\
\end{multline*}
Therefore, from bounds on $d(k)$ and $m(k)$, it follows that
\begin{multline*}
\| \mathcal{A}_N [ f_1-f_2, a_1, L_1] \|_{w, r-3/2} \le
C_r \sigma \left ( |L_1-2\pi| \| f_1 - f_2 \|_{w,r+3/2}
+ |a_1| \| f_1-f_2 \|_{w, r-1/2} \right ) .
\end{multline*}
Further, since
\begin{multline*}
\mathcal{A}_N [ f ,a_1, L_1]-
\mathcal{A}_N [ f , a_2, L_2]
= \sigma \left ( \frac{8 \pi^3}{L_2^3} - \frac{8 \pi^3}{L_1^3} \right )
\sum_{k=2}^\infty e^{ik\alpha} d(k) {\hat f}(k)  \\
+ \left \{ \left ( \frac{2 \pi}{L_1} - \frac{2 \pi}{L_2} \right )
+\frac{2\pi}{L_2} \left ( e^{i a_1 } -e^{i a_2} \right ) \right \}
\sum_{k=2}^\infty e^{ik\alpha}m(k) {\hat f}(k+1),
\end{multline*}
the results follow from the definition of $ \| \cdot \|_{w, r} $ on
using the restriction on $L_1, L_2$.
\end{proof}

\subsection{Linear Evolution and space-time estimates}

\begin{definition}
\label{defHsigmar}
For $r \ge 3$, we define the space of real valued
functions
$$H^r_{\sigma}  \equiv
C\big([0,\infty),\dot{H}^r\big) \cap
L^2\big([0,\infty),\dot{H}^{r+3/2}\big),$$ equipped with the norm
$\| \cdot \|_{H_\sigma^r}$ defined by
$$\|u\|^2_{H_\sigma^r}=\sup_{t\ge0} e^{t\sigma} \| u(\cdot,t)\|_{w,r}^2
+\frac{\sigma}{4}
\int_0^\infty e^{\sigma t} \| u(\cdot,t)\|_{w,r+3/2}^2 dt.$$
\end{definition}

We now consider linear evolution equation
\begin{equation}\label{eqnv}
v_t (\alpha,t)-\mathcal{A} [v] (\alpha,t)
= f(\alpha,t) ~~{\rm with} ~~
v(\cdot ,  0) = v_0 \in {\dot H}^r,
\end{equation}
where $f \in H_\sigma^{r-3}$.

\medskip

\begin{lemma}(A priori linear energy estimates)
\label{nonhomogeneous}
Suppose $r \ge 3$, $f \in H_\sigma^{r-3} $ and
Then a solution $v (., t) \in {\dot H}^r$ to (\ref{eqnv})
will satisfy
the following energy inequality for any $t$:
\begin{equation*}
e^{\sigma t} \| v (\cdot, t) \|^2_{w,r} + \frac{\sigma}{4}
\int_0^t e^{\sigma \tau} \| v (\cdot, \tau) \|^2_{w+3/2, r} d\tau
\le \| v_0 \|_{w,r}^2 + \frac{8}{3 \sigma^2}
\| f \|^2_{H_\sigma^{r-3}},
\end{equation*}
and thus
\begin{equation*}
\| v \|^2_{H_\sigma^r} \le \| v_0 \|_{w,r}^2 + \frac{8}{3 \sigma^2}
\| f \|^2_{H_\sigma^{r-3}}
\end{equation*}
\end{lemma}
\begin{proof}
Taking the $(\cdot, \cdot)_{w,r}$ inner-product on both sides of (\ref{eqnv}) with
$v$, we obtain
\begin{equation}
\label{eqvf}
\frac{d}{dt} \| v \|_{w,r}^2 -
2 \left ( v (\cdot, t) , \mathcal{A} [ v ] \right )_{w,r}
= 2 \left ( v (\cdot, t) , f(\cdot,t)\right )_{w,r}.
\end{equation}
>From Lemma \ref{lemAinner}, this implies
\begin{eqnarray*}
\frac{d}{dt} \| v \|_{w,r}^2 + \frac{15 \sigma}{32}
\| v (., t) \|_{w, r+3/2}^2
\le
2 \| v (., t) \|_{w, r+3/2} \| f (., t) \|_{w,r-3/2}   .
\end{eqnarray*}
Noting that
$$|k|^{r+3/2} \ge 2^{1/2} |k|^{r+1} \ge 2 |k|^{r+1/2}
\ge 2^{3/2} |k|^r ~~{\rm for}~~
k \ge 2$$
implies that
$$\| v \|_{w,r+3/2} \ge 2^{1/2} \| v \|_{w,r+1} \ge
2 \| v \|_{w, r+1/2} \ge 2^{3/2} \| w \|_r. $$
It follows that
on using Cauchy Schwartz inequality,
\begin{equation*}
\frac{d}{dt} \| v \|_{w,r}^2 + \sigma \| v \|_{w,r}^2 +
\frac{\sigma}{4} \| v (., t) \|_{w,r+3/2}^2
\le \frac{32}{3\sigma}
\| f (., t) \|^2_{w,r-3/2}.
\end{equation*}
Integration gives the desired energy inequality.
Noting that this is true for any $t$, and using the definition of
$\| \cdot\|_{H^r_\sigma}$, we obtain
the given bounds on $\| v \|_{H_\sigma^r}$.
\end{proof}

\begin{remark}
Proof of existence of a solution to the linear equation
(\ref{eqnv}) for given real valued
$f \in { H}^{r-3}_\sigma $
and the initial condition $v_0
\in {\dot H}^r$, satisfying the given conditions follows in
a standard manner. Note  that we can introduce a sequence
of Galerkin approximants $v_N (\alpha, t)$ containing a
finite number of Fourier modes. This will satisfy
the energy bounds
in Lemma \ref{nonhomogeneous}, independent of $N$.
These approximants clearly solve linear ODEs
for which the unique solutions exist globally. In the Hilbert
space $L^2 \left ([0, S], { H}^{r+3/2} \right )$, there exists
a subsequence of $v_N \rightarrow v $ weakly. Therefore for
almost all $t \in [0, S]$, this subsequence denoted again
by $v_N (\cdot, t) \rightarrow v (\cdot, t)$ strongly in $\dot{H}^r$. From
the energy bound, the limit $v(\cdot, t)$ is bounded in $\dot{H}^{r}$
for any $t \in [0, S]$, and $v \in L^2 ([0, S], H^{r+3/2} )$.
It is also easy to check that
the limiting solution satisfies  (\ref{eqnv}) in a classical sense
for sufficiently large $r$.
This proves existence of a global classical solution for
any $t$ noting that $r \ge 3$ since $r$ is arbitrary.
The uniqueness of this solution
follows from the energy bound itself.
\end{remark}

\begin{definition}
\label{defS1S2}
It is convenient to define a linear operator $e^{t \mathcal{A}} $ so that
$$v = e^{t \mathcal{A}} v_0  $$
is the unique solution $v (\alpha, t) \in {\dot H}^r $ satisfying
(\ref{eqnv}) for $f=0$, with the initial condition $v (\alpha, 0) = v_0$.
\end{definition}

\begin{note}
\label{noteduhammel}
It is easily seen that $e^{t \mathcal{A}} $ is a semi-group.
Further, using Duhammel principle, the solution $v (\alpha, t) \in {\dot H}^r$
satisfying (\ref{eqnv}) for $v_0=0$ may be expressed as
\begin{equation}
\label{duhammel}
v (\alpha, t) = \int_0^t e^{(t-\tau) \mathcal{A}} f (\alpha, \tau) d\tau.
\end{equation}
\end{note}

\begin{remark}
The energy bounds in Lemma
\ref{nonhomogeneous}
imply that
\begin{equation}
\label{S0S1bound}
\| e^{t \mathcal{A}} v_0 \|_{H_\sigma^r}
\le \| v_0 \|_{w,r} , \mbox ~~
\left\| \int_0^t e^{(t-\tau) \mathcal{A} } f (\cdot, \tau) d\tau
\right\|_{H_\sigma^r}
\le \frac{2\sqrt{2}}{\sqrt{3} \sigma} \| f \|_{H_\sigma^{r-3}} .
\end{equation}
\end{remark}

\subsection{Nonlinear evolution, contraction map and proof of Proposition
\ref{translationprop5.2} }

We express the evolution equation (\ref{RHSC1}) in the integral
form:
\begin{equation}
\label{evoltiltheta}
{\tilde \theta} (\alpha,t)
= e^{t \mathcal{A}} {\tilde \theta}_0 +
\int_0^t d\tau e^{(t-\tau) \mathcal{A}}
\left \{\mathfrak{N} [ {\tilde \theta} (\cdot, \tau), {\hat \theta} (0; \tau)]+
\mathcal{A}_N [{\tilde \theta} (\cdot, \tau),{\hat \theta} (0;\tau),
L(\tau) ] \right\}\equiv \mathcal{S}_1\big[{\tilde \theta}, {\hat \theta}(0;\cdot)\big] (\alpha, t)  ,
\end{equation}
\begin{equation}
\label{evolhat}
{\hat \theta} (0; t)  = \int_0^t \mathfrak{N}_0\Big[ \mathcal{S}_1(\alpha,\cdot),
{\hat \theta} (0;\cdot)\Big] (\tau)d\tau \equiv \mathcal{S}_2 [{\tilde \theta}, {\hat \theta}(0;\cdot )] ( t)   ,
\end{equation}
where $L=L(t)$ is  determined in terms of
${\tilde \theta} (., t)$
through (B.4) and (\ref{eqomega}).

Equations (\ref{evoltiltheta}) and (\ref{evolhat}) will be the basis
of a contraction mapping theorem for $({\tilde \theta} , {\hat
\theta} (0;t) )$ in an small  ball in the space
$$\mathcal{D} \equiv
H_{\sigma}^r \times \mathbf{C} [ 0, \infty ) $$ equipped with the
norm $\| \cdot\|_{\mathcal{D}} $ so that
\begin{equation}
\label{defBnorm} \left\| \big ( {\tilde \theta} , {\hat \theta}
(0;\cdot) \big ) \right\|_{ \mathcal{D}} = \| {\tilde \theta}
\|_{H^r_\sigma} + \big | {\hat \theta} (0;\cdot) \big |_\infty.
\end{equation}

First, we define a mapping in $\mathcal{D}$ by
$$\mathcal{S}\big[\tilde\theta,\hat\theta(0;\cdot)\big]\equiv
\left(\begin{aligned}
\mathcal{S}_1\big[\tilde\theta,\hat\theta(0;\cdot)\big](\alpha,t)\\
\mathcal{S}_2\big[\tilde\theta,\hat\theta(0;\cdot)\big](t)\end{aligned}\right).
$$

Secondly, we estimate the nonlinear terms in the space $H_\sigma^r$.
\begin{lemma}\label{NN0bounds} For $r \ge 3$ and $\sigma>0$, assume
$ \big ( {\tilde \theta} (\alpha, t), {\hat \theta} (0;t ) \big )$
satisfy the condition
\begin{equation}
\label{ballcond} \| {\tilde \theta} \|_{H_\sigma^r} \le \epsilon
~~~,~~~ | {\hat \theta} (0;\cdot) |_\infty \le \epsilon .
\end{equation}
Then for $\mathcal{A}_N [{\tilde \theta},{\hat \theta} (0;\cdot),
L(\cdot) ] (\alpha,t)$, $\mathfrak{N} [ {\tilde \theta}, {\hat
\theta} (0;\cdot) ] (\alpha, t)$ and scalar $\mathfrak{N}_0 [
{\tilde \theta}, {\hat \theta} (0;\cdot)](t)$, determined from
(\ref{AdefN}), (\ref{eqmathfrakN}) and (\ref{eqmathfrakN0})
respectively,\footnote{Note that $\Gamma$ and $L$ appearing in the
expressions are determined in terms of ${\hat \theta}$ and ${\hat
\theta} (0;t)$ through (\ref{gamma4.3}) and (B.4) on using
(\ref{eqomega}) and (\ref{zomega}).}
\begin{equation*}
\big\| \mathcal{A}_N [{\tilde \theta},{\hat \theta} (0;\cdot), L ]
+\mathfrak{N} [ {\tilde \theta}, {\hat \theta} (0;\cdot) ]
\big\|_{H_\sigma^{r-3}} \le c_1\| {\tilde \theta} \|_{H_\sigma^r}
\left ( \| {\tilde \theta} \|_{H_\sigma^r} + | {\hat \theta}
(0;\cdot)|_\infty \right ) ,
\end{equation*}
\begin{equation*}
\Big | \int_0^t  \mathfrak{N}_0 [ {\tilde \theta} , {\hat \theta}
(0;\cdot)](\tau)d\tau \Big |_\infty \le c_2 \| {\tilde \theta}
\|_{H_\sigma^3} .
\end{equation*}
Further, if both $ \left ( {\tilde \theta}^{(1)} (\alpha, t), {\hat
\theta}^{(1)} (0;t) \right )$ and $ \left ( {\tilde \theta}^{(2)}
(\alpha, t), {\hat \theta}^{(2)} (0;t) \right )$ satisfy
(\ref{ballcond}), then the corresponding $\left
(\mathcal{A}_N^{(1)}, \mathfrak{N}^{(1)}, \mathfrak{N}_0^{(1)}
\right )$ and $\left ( \mathcal{A}_N^{(2)},\mathfrak{N}^{(2)},
\mathfrak{N}_0^{(2)} \right )$ satisfy
\begin{equation*}
\|\mathcal{A}_N^{(1)}-\mathcal{A}_N^{(2)}\|_{H_\sigma^{r-3}}+\|
\mathfrak{N}^{(1)} - \mathfrak{N}^{(2)} \|_{H_\sigma^{r-3}} \le
c_3\epsilon \left ( \| {\tilde \theta}^{(1)}-{\tilde \theta}^{(2)}
\|_{H_\sigma^r} + \big| {\hat \theta}^{(1)} (0;\cdot) -{\hat
\theta}^{(2)} (0;\cdot)\big |_\infty \right ),
\end{equation*}
\begin{multline*}
\Big | \int_0^t \big ( \mathfrak{N}_0^{(1)} -\mathfrak{N}_0^{(2)}
\big )d\tau \Big |_\infty \le c_4 \left ( \| {\tilde \theta}^{(1)} -
{\tilde \theta}^{(2)} \|_{H_\sigma^3} +\epsilon \big | {\hat
\theta}^{(1)} (0; \cdot) - {\hat \theta}^{(2)} (0; \cdot) \big
|_\infty \right ).
\end{multline*}
\end{lemma}
\begin{proof} We note the bounds for $\mathcal{A}_N$, $\mathfrak{N}$ and $\mathfrak{N}_0$ in Lemma \ref{ANnorm} and
Corollary \ref{mathfrakNcor}. It follows from the equivalence of $
\| \cdot\|_r$ and $ \| \cdot \|_{w,r}$ norms and the definition of $
\| . \|_{H_\sigma^r} $ norm that
\begin{multline*}
e^{\sigma t/2}
\big\| \mathcal{A}_N [{\tilde \theta},{\hat \theta} (0;\cdot), L ]
+\mathfrak{N} [ {\tilde \theta}, {\hat \theta} (0;\cdot) ]
\big\|_{w,r-3} \\
\le Ce^{\sigma t/2}\big(\|\tilde\theta\|_{w,r}\|\tilde\theta\|_1+\|\tilde\theta\|_{w,r-1}\|\tilde\theta\|_{w,r-2}+|\hat\theta(0;\cdot)|_\infty\|\tilde\theta\|_{w,r-2}\big)\\
\le C\| {\tilde \theta} \|_{H_\sigma^r} \left ( \|
{\tilde \theta} \|_{H_\sigma^r} + | {\hat \theta} (0;\cdot)|_\infty
\right ).
\end{multline*} Further, it follows that
\begin{multline*}
\int_0^\infty e^{\sigma t} \|  \mathcal{A}_N [{\tilde \theta},{\hat
\theta} (0;t), L ] +\mathfrak{N} [{\tilde \theta} (\cdot, t), {\hat
\theta} (0; t) ]
\|^2_{w,r-3/2} dt \\
\le C \sup_{t} \left [ e^{\sigma t} \| {\tilde \theta} (\cdot, t)
\|^2_{w,r} +|\hat\theta(0;t)|^2 \right ]
\int_0^\infty e^{\sigma t} \| {\tilde \theta} \|^2_{w,r+3/2} dt \\
\le C \| {\tilde \theta} \|_{H_\sigma^r}^2 \left ( \| {\tilde
\theta} \|_{H_\sigma^r}^2 + | {\hat \theta} (0;\cdot)|^2_\infty
\right ).
\end{multline*}
Therefore the bounds for $\| \mathcal{A}_N+\mathfrak{N}
\|_{H_\sigma^{r-3}} $ follows. For $\mathfrak{N}_0$, we use
Corollary \ref{mathfrakNcor} again to note
\begin{equation*}
\Big|\int_0^t  \mathfrak{N}_0 [{\tilde \theta}, {\hat \theta} (0;.)
](\tau)d\tau\Big|_\infty  \le C \int_0^\infty \big(\| {\tilde
\theta} (\cdot, \tau) \|_{w,3}^2+ \| {\tilde \theta} (\cdot, \tau)
\|_{w,2}\big) d\tau  \le c_2\| {\tilde \theta} \|_{H_\sigma^3}.
\end{equation*}
The statements for the differences of $\mathfrak{N}$,
$\mathfrak{N}_0$ for different $({\tilde \theta}, {\hat \theta} (0;
t))$ follow from parallel statements in Lemma \ref{ANnorm} and
Corollary \ref{mathfrakNcor}.
\end{proof}

We have the following contraction properties in a ball
$$\mathcal{V}_\epsilon\equiv\big\{(u,v)\in\mathcal{D}|\|u\|_{H^r_\sigma}\le\epsilon,|v|_\infty\le\epsilon\big\}.$$
\begin{lemma}
\label{lemcontract} Let $\sigma>0$, $r \ge 3$. Assume  $\left (
{\tilde \theta} , {\hat \theta} (0; t) \right ) \in
\mathcal{V}_\epsilon$ and $c_1$, $c_2$,  $c_3$, $c_4$ are as defined
in Lemma \ref{NN0bounds}. If  for sufficiently small $\epsilon$,  $
\| {\tilde \theta}_0 \|_{w,r} <
\min\{\frac{\epsilon}{2},\frac{\epsilon}{2c_1}\} $ and ${\hat
\theta} (0; 0) =0$, then
$$ \mathcal{S} \big[{\tilde \theta},
{\hat \theta} (0;\cdot) \big ] \in \mathcal{V}_\epsilon.
$$
Further, if each of
$\left ( {\tilde \theta}^{(1)} , {\hat \theta}^{(1)} (0;t) \right ) $
and $\left ( {\tilde \theta}^{(2)} , {\hat \theta}^{(2)} (0;t)\right ) $
 belongs to $\mathcal{V}_\epsilon$, then 
there exists $c_5$ depending on $c_1, \cdots, c_4$, such that
\begin{multline*}
\Big\| \mathcal{S} \big[ {\tilde \theta}^{(1)}, {\hat \theta}^{(1)}
(0;\cdot) \big ] -\mathcal{S} \big[{\tilde \theta}^{(2)}, {\hat
\theta}^{(2)} (0;\cdot) \big]\Big\|_{\mathcal{D}}  \le c_5
\epsilon\Big\| \big ( {\tilde \theta}^{(1)}-{\tilde \theta}^{(2)},
 {\hat \theta}^{(1)} (0;\cdot) -{\hat \theta}^{(2)} (0;\cdot) \big)
\Big\|_{\mathcal{D}}.
\end{multline*}
\end{lemma}
\begin{proof}
Define $c_6=\frac{2\sqrt{2}}{\sqrt{3}\sigma}c_1$. By (\ref{S0S1bound}) and Lemma \ref{NN0bounds}, we have
\begin{equation*}
\Big\|\mathcal{S}_1\big[\tilde\theta,\hat\theta(0;\cdot)\big]\Big\|_{H^r_\sigma}\leq
\|\tilde\theta_0\|_{w,r}+c_6\big(\|\tilde\theta\|_{H^r_\sigma}^2+\|\tilde\theta\|_{H^r_\sigma}|\hat\theta(0;\cdot)|_\infty\big)\le\epsilon,
\end{equation*}
if $c_6\epsilon<\frac{1}{4}$. We also
have
\begin{equation*}
\Big|\mathcal{S}_2\big[\tilde\theta,\hat\theta(0;\cdot)\big]\Big|_{\infty}\leq
c_2\Big\|\mathcal{S}_1\big[\tilde\theta,\hat\theta(0;\cdot)\big]\Big\|_{H^3_\sigma}\leq
c_2\Big(\|\tilde\theta_0\|_{w,r}+c_6\|\tilde\theta\|_{H^r_\sigma}^2+c_6\|\tilde\theta\|_{H^r_\sigma}|\hat\theta(0;\cdot)|_\infty\Big)
\le\epsilon,
\end{equation*}
if $c_2c_6\epsilon<\frac14$.

The statements for the differences of $\mathcal{S}$,
 for different $({\tilde \theta}, {\hat \theta} (0;
t))$ follows from parallel statements in Lemma \ref{NN0bounds}.

\end{proof}

\begin{note} Constants $c_1, c_2, c_3, c_4$ and $c_5$ depend on $\sigma$.
\end{note}

\noindent{\bf Proof of Proposition \ref{translationprop5.2}:} If
$c_5 \epsilon < 1$, then it is clear that the right sides of
(\ref{evoltiltheta}) and (\ref{evolhat}) define a contraction map in
a small ball $\mathcal{V}_\epsilon$ in the Banach space
$\mathcal{D}$. Therefore, there exists a unique solution
$\big({\tilde \theta}, {\hat \theta} (0;t) \big) $ satisfying
equations (\ref{evoltiltheta}) and (\ref{evolhat}), hence (B.1). The
local uniqueness of solutions (see Appendix \S \ref{A7.2}) implies that this is
the only solution. The $e^{-\sigma t/2 }$ exponential decay of
${\tilde \theta}$ and hence of $\theta$ implies that the steady
circle is approached exponentially in time. The constraint condition
(B.4) implies that $L-2\pi$ decays exponentially.

\begin{note} It is easy to show that 
given any $j$,  $\tilde\theta(\cdot,t)\in
\dot{H}^{r+3j/2}$ for $t\ge j$ in the following manner. 
Since $(\tilde\theta,\hat\theta(0;t))\in
\mathcal{V}_\epsilon$ for some $r\ge3$, there exists $t_0\in[0,1]$
such that $\|\tilde\theta(\cdot,t_0)\|_{r+3/2}<\epsilon$. 
So we can restart clock at $t=t_0$ and prove global solution in
$H_{\sigma}^{r+3/2} $.
It follows that there exists $t_1\in(t_0,t_0+1]$ so that 
$\|\tilde\theta(\cdot,t_1)\|_{r+3}<\epsilon$. By
bootstrapping, we obtain $\tilde\theta(\cdot,t)\in
\dot{H}^{r+3j/2}$ for $t\ge j$.

Indeed more can be shown to be true. The contraction argument in
Proposition \ref{translationprop5.2} can be carried out for
arbitrary sized initial condition over small sized time interval.
Through bootstrapping and using Sobolev embedding theorem, we can
conclude that the solution is in   
$C^{\infty}$ in space for $t \in (0,S]$.
The property of smoothing of initial conditions 
is similar to other dissipative equations like Navier-Stokes.
\end{note}

\section{Steady translating bubble in the channel with sidewalls
($\beta > 0$)}

For a steadily traveling bubble solution,
in the frame of an appropriately moving bubble,
we have to require the normal interface speed $U=0$. This would imply
(A.1) is automatically satisfied for a time-independent
$\theta^{(s)} (\alpha) $ and $L=L^{(s)} = 2 \pi$,
where $z(\alpha)=z^{(s)}(\alpha)$ describes the geometry shape of the steady bubble and $ \gamma (\alpha, t) = \gamma^{(s)}(\alpha)$ is determined in terms of
$\theta$ through (A.2).

Earlier, we have shown  that
 for the bubble with the invariant area,
\begin{equation}\label{7.1}
\int_0^{2\pi} U(\alpha)d\alpha=0.
\end{equation}
\medskip
Further, there is no loss of generality in the steady problem to choose
${\hat \theta}^{(s)} (0) =0$  since this corresponds to a choice of
origin for $\alpha$, and make $\alpha=0$ correspond to $y^{(s)}(0)=0$.
Thus, from (\ref{1.12}), the steady bubble problem reduces to
\begin{equation*}
\mathfrak{U} \left [{\tilde \theta}^{(s)}, u_0, \beta \right ]
\equiv
\mathcal{Q}_0\Big(\frac{1}{2}\mathcal{H}(\gamma^{(s)})+\frac{1}{2}\Re\big(\mathcal{G}[z^{(s)}]\gamma^{(s)}\big)
+(u_0+1)\cos\big(\alpha+\theta^{(s)}(\alpha)\big)\Big)=0
\tag{C.1},
\end{equation*}
with vortex sheet strength $\gamma^{(s)}$ and $\hat{\theta}^{(s)}(\pm1)$
 determined by
\begin{equation*}
\left(I+a_\mu\mathcal{F}[z^{(s)}]\right)\gamma^{(s)}=2\big(1+\frac{\mu_2}{\mu_1+\mu_2}u_0\big)\sin\big(\alpha+\theta^{(s)}(\alpha)\big)+\sigma\theta_{\alpha\alpha}^{(s)},\tag{C.2}
\end{equation*}
\begin{equation*}
\int_0^{2\pi}\exp\Big(i\alpha+i\big(\hat{\theta}^{(s)}(-1)e^{-i\alpha}
+\hat{\theta}^{(s)}(1)e^{i\alpha}+\tilde\theta^{(s)}(\alpha)\big)\Big)d\alpha=0,
\tag{C.3}\end{equation*} where
$\theta^{(s)}=\tilde\theta^{(s)}+\hat{\theta}^{(s)}(1)e^{i\alpha}+\hat{\theta}^{(s)}(-1)e^{-i\alpha}$.
Hence  we seek solutions  $(\tilde\theta^{(s)},u_0,\beta)\in
\dot{H}^r\times(-1,1)\times(-\Upsilon,\Upsilon)$  \footnote{We choose small $\epsilon$ and $\Upsilon$ such that Proposition \ref{prop2.6} can be  applied in (C.3) and  Proposition \ref{propgamma} can also be applied in (C.2).}.

Recently, \cite{Xie1}, \cite{Xie2} and \cite{Xie3} obtained selection results for steady finger for small  non-zero surface tension.

\begin{remark} For $r\ge3$, by
Propositions \ref{prop2.6} and \ref{propgamma}, we know that
$\|\mathfrak{U}\|_{r-2}\leq C$ with $C$ depending on $\Upsilon$ and the diameter
of $\mathcal{B}^r_\epsilon$. Hence, $\mathfrak{U}$ maps an open set of
$H^{r}_p\times\mathbb{R}^2$ into the space $H^{r-2}_p$.
\end{remark}

\begin{note}\label{note7.3}
We know that $\mathfrak{U}[0,0,0](\alpha)=0$ with the corresponding vortex sheet strength
$\gamma^{(s)}(\alpha)=2\sin{\alpha}$ and $\hat{\theta}^{(s)}(\pm 1)=0$. We also
see that $\frac{\partial \mathfrak{U}}{\partial u_0}[0,0,0](\alpha)=\frac{\mu_1}{\mu_1+\mu_2}\cos\alpha$ and the
Fr$\acute{e}$chet derivative $\mathfrak{U}_{\tilde\theta^{(s)}}[0,0,0]h$ 
(see Appendix \S \ref{FrechU}) for $h\in
\dot{H}^r$ is given by:
\begin{equation}\label{7.2}
\mathfrak{U}_{\tilde\theta^{(s)}}[0,0,0]h=\frac{\sigma}{2}\mathcal{H}(h_{\alpha\alpha})
-i\sum_{k=1}^\infty(1+a_\mu)\frac{k+1}{k+2}\hat{h}(k+1)e^{ik\alpha}+c.c..
\end{equation}
\end{note}

It is convenient to recast the steady state problem in a contraction mapping
problem using smallness of $\beta$ and the knowledge that
$\left ( {\tilde \theta}^{(s)} , \gamma^{(s)} \right ) = \left (0,
2 \sin \alpha \right )$ is the steady state solution for $\beta=0$.
We rewrite $\mathfrak{U} = 0 $ as
\begin{equation}
\label{stead1}
\mathfrak{U}_{{\tilde \theta}^{(s)}} [0, 0, 0] {\tilde \theta}^{(s)}
+
\mathfrak{U}_{u_0} [0, 0, 0] u_0 + \frac{\beta^2}{2}
\mathfrak{U}_{\beta\beta}
[ 0, 0, 0]
= \mathfrak{N}^{(s)} \left [ {\tilde \theta}^{(s)}, u_0, \beta \right ],
\end{equation}
where
\begin{multline}
\label{stead2}
\mathfrak{N}^{(s)} \left [ {\tilde \theta}^{(s)}, u_0, \beta \right ]
=  -\mathfrak{U} [ {\tilde \theta}^{(s)}, u_0, \beta ] +
\mathfrak{U}_{{\tilde\theta}^{(s)}} [0, 0, 0] {\tilde \theta}^{(s)}
+
\mathfrak{U}_{u_0} [0, 0, 0] u_0 + \frac{\beta^2}{2}
\mathfrak{U}_{\beta\beta}
[ 0, 0, 0] \\
=\mathfrak{A}[\tilde\theta^{(s)}](\alpha)+\mathfrak{B}[\tilde\theta^{(s)},u_0]+\mathfrak{C}[\tilde\theta^{(s)},u_0,\beta]
\end{multline}
with
$$\mathfrak{A}[\tilde\theta^{(s)}](\alpha)=\mathfrak{U}\big[\tilde\theta^{(s)},0,0\big](\alpha)-\mathfrak{U}_{\tilde\theta^{(s)}}[0,0,0]\tilde\theta^{(s)}(\alpha),$$
$$\mathfrak{B}[\tilde\theta^{(s)},u_0]=\mathfrak{U}\big[\tilde\theta^{(s)},u_0,0\big]-\mathfrak{U}\big[\tilde\theta^{(s)},0,0\big]-\mathfrak{U}_{u_0}[0,0,0]u_0,$$
$$\mathfrak{C}[\tilde\theta^{(s)},u_0,\beta]=\mathfrak{U}\big[\tilde\theta^{(s)},u_0,\beta\big]-\mathfrak{U}\big[\tilde\theta^{(s)},u_0,0\big]-\frac{\beta^2}{2}
\mathfrak{U}_{\beta\beta}
[ 0, 0, 0].$$
It will be shown that $\mathfrak{N}^{(s)}$
is either nonlinear
in $\left({\tilde \theta}^{(s)}, u_0 \right )$ or at least $O(\beta^4)$.

\begin{lemma}\label{lemathNs1}
For any $r \ge 3$, let $\| {\tilde \theta}^{(s)} \|_r $ and $u_0$
sufficiently small, then there exists $C$ independent of  $u_0$ and
${\tilde \theta}^{(s)}$ so that
$$\big\|\mathfrak{A}[\tilde\theta^{(s)}]\big\|_{r-1}\leq C\|\tilde\theta^{(s)}\|_{r-1}\|\tilde\theta^{(s)}\|_r.$$
Further, let $\mathfrak{A}^{(1)}$ and $\mathfrak{A}^{(2)}$ correspond to $\tilde\theta^{(s)}_1$ and $\tilde\theta^{(s)}_2$ respectively, each in $\dot H^r$. Then there exists $C$ independent of $\beta$, $u_0$ and
${\tilde \theta}^{(s)}$ so that
$$\big\|\mathfrak{A}^{(1)}-\mathfrak{A}^{(2)}\big\|_{r-1}\leq C\big(\|\tilde\theta^{(s)}_1\|_{r-1}\|\tilde\theta^{(s)}_1-\tilde\theta^{(s)}_2\|_r+\|\tilde\theta^{(s)}_1\|_{r}\|\tilde\theta^{(s)}_1-\tilde\theta^{(s)}_2\|_{r-1}\big).$$
\end{lemma}
\begin{proof}
We identify $\mathfrak{A}[\tilde\theta^{(s)}]$ as the nonlinear part of normal velocity $U$ for $\beta=0$ in \ref{translating6.6}).
By Lemma \ref{N1lem4.6}, the statements of the Lemma follow.
\end{proof}

\begin{lemma}\label{lemathNs2}
For any $r \ge 3$, let $\| {\tilde \theta}^{(s)} \|_r $ and $u_0$
sufficiently small, then there exists $C$ independent of $u_0$ and
${\tilde \theta}^{(s)}$ so that
$$\big\|\mathfrak{B}[\tilde\theta^{(s)},u_0]\big\|_{r}\leq C|u_0|\|\tilde\theta^{(s)}\|_{r}.$$
Further, let $\mathfrak{B}^{(1)}$ and $\mathfrak{B}^{(2)}$ correspond to $(\tilde\theta^{(s)}_1,u_0^{(1)})$ and $(\tilde\theta^{(s)}_2,u_0^{(2)})$ respectively, each in $\dot H^r$. Then there exists $C$ independent of $\beta$, $u_0$ and
${\tilde \theta}^{(s)}$ so that
$$\big\|\mathfrak{B}^{(1)}-\mathfrak{B}^{(2)}\big\|_{r}\leq C\big(|u_0^{(1)}|\|\tilde\theta^{(s)}_1-\tilde\theta^{(s)}_2\|_r+\|\tilde\theta^{(s)}_1\|_{r}|u_0^{(1)}-u_0^{(2)}|\big).$$
\end{lemma}
\begin{proof}
Let $\gamma^{(u_0)}$  correspond to $(\tilde\theta^{(s)},u_0,0)$, while $\gamma^{(u_0)}$ corresponds to $(\tilde\theta^{(s)},0,0)$. Then by (\ref{1.12}), we obtain
\begin{multline}\label{mathfrakB1}
\mathfrak{B}[\tilde\theta^{(s)},u_0]=\frac{1}{2}\mathcal{H}[\gamma^{(u_0)}-\gamma^{(u_0)}_0]+\frac{1}{2}\Re\big(\mathcal{G}[z^{(s)}](\gamma^{(u_0)}-\gamma^{(u_0)}_0)\big)\\+u_0\left(\cos\big(\alpha+\tilde\theta^{(s)}(\alpha)\big)-\frac{\mu_1}{\mu_1+\mu_2}\cos\alpha\right).
\end{multline}
For (C.2) and the relation between $\mathcal{F}$ and $\mathcal{G}$, we also have
\begin{multline}\label{mathfrakB2}
\gamma^{(u_0)}-\gamma^{(u_0)}_0=-a_\mu\Re\big(\frac1i\mathcal{G}[z^{(s)}](\gamma^{(u_0)}-\gamma^{(u_0)}_0)-\frac1i\mathcal{G}[\omega_0](\gamma^{(u_0)}-\gamma^{(u_0)}_0)\big)\\+2u_0\frac{\mu_2}{\mu_1+\mu_2}\sin\big(\alpha+\tilde\theta^{(s)}\big).
\end{multline}
By Lemma \ref{lem3.11} (for $\beta=0$ and $L^{(1)}=L^{(2)}=2\pi$), from (\ref{mathfrakB2}), we have
$$\|\gamma^{(u_0)}-\gamma^{(u_0)}_0\|_1\leq C (\|\tilde\theta^{(s)}\|_r\|\gamma^{(u_0)}-\gamma^{(u_0)}_0\|_1+|u_0|\big).$$
Hence for sufficient small $\|\tilde\theta^{(s)}\|_r$, we have
\begin{equation}\label{gammau0}
\|\gamma^{(u_0)}-\gamma^{(u_0)}_0\|_1\leq C|u_0|.
\end{equation}
Plugging (\ref{mathfrakB2}) into (\ref{mathfrakB1}), we have
\begin{multline*}
\mathfrak{B}[\tilde\theta^{(s)},u_0]=\frac12\mathcal{H}\left[a_\mu\Re\big(\frac1i\mathcal{G}[z^{(s)}](\gamma^{(u_0)}-\gamma^{(u_0)}_0)-\frac1i\mathcal{G}[\omega_0](\gamma^{(u_0)}-\gamma^{(u_0)}_0)\big)\right]\\
+\frac12\Re\big(\mathcal{G}[z^{(s)}](\gamma^{(u_0)}-\gamma^{(u_0)}_0)-\mathcal{G}[\omega_0](\gamma^{(u_0)}-\gamma^{(u_0)}_0)\big)\\+u_0\left(\Big(\cos\big(\alpha+\tilde\theta^{(s)}(\alpha)\big)-\cos\alpha\Big)+\frac{\mu_2}{\mu_1+\mu_2}\mathcal{H}\Big[\sin\big(\eta+\tilde\theta^{(s)}\big)-\sin(\eta)\Big](\alpha)\right).
\end{multline*}
Hence, by Lemmas \ref{lem3.1}, \ref{lem3.11} (for $\beta=0$ and $L^{(1)}=L^{(2)}=2\pi$) and (\ref{gammau0}), we have the first statement.

For the difference term, by Lemmas \ref{lem3.1}, \ref{lem3.11} (for $\beta=0$ and $L^{(1)}=L^{(2)}=2\pi$) and Proposition \ref{propgamma}.
\end{proof}

\begin{lemma}\label{lemathNs3}
For any $r \ge 3$, assume $\| {\tilde \theta}^{(s)} \|_r $, $u_0$ and $\beta$
are sufficiently small. Then there exists $C$ independent of $\beta$, $u_0$ and
${\tilde \theta}^{(s)}$ so that
$$\big\|\mathfrak{C}[\tilde\theta^{(s)},u_0]\big\|_{r}\leq C\big(\beta^2|u_0|+\beta^2\|\tilde\theta^{(s)}\|_{r}+\beta^4\big).$$
Further, suppose $\mathfrak{C}^{(1)}$ and $\mathfrak{C}^{(2)}$ correspond to $(\tilde\theta^{(s)}_1,u_0^{(1)},\beta)$ and $(\tilde\theta^{(s)}_2,u_0^{(2)},\beta)$ respectively, each in $\dot H^r$. Then there exists $C$ independent of $\beta$, $u_0$ and
${\tilde \theta}^{(s)}$ so that
$$\big\|\mathfrak{C}^{(1)}-\mathfrak{C}^{(2)}\big\|_{r}\leq C\beta^2\big(\|\tilde\theta^{(s)}_1-\tilde\theta^{(s)}_2\|_r+|u_0^{(1)}-u_0^{(2)}|\big).$$
\end{lemma}
\begin{proof}
Suppose $\gamma^{(s)}_0$ satisfying (C.2)   corresponds to $(\tilde\theta^{(s)},u_0,0)$. Then for (\ref{1.12}),
\begin{multline}\label{mathfrakC1}
\mathfrak{C}[\tilde\theta^{(s)},u_0,\beta]=\frac{1}{2}\mathcal{H}[\gamma^{(s)}-\gamma^{(s)}_0]+\frac{1}{2}\Re\big(\mathcal{G}_1[z^{(s)}](\gamma^{(s)}-\gamma^{(s)}_0)\big)+\frac{1}{2}\Re\big(\mathcal{G}_2[z^{(s)}]\gamma^{(s)}\big)-\frac{\beta^2}{6}(1+a_\mu)\cos\alpha\\
=\frac{1}{2}\mathcal{H}\left[\gamma^{(s)}-\gamma^{(s)}_0+a_\mu\frac{\beta^2}{12}\sin\eta\right](\alpha)+\frac12\Re\big(\mathcal{G}[z^{(s)}](\gamma^{(s)}-\gamma^{(s)}_0)-\mathcal{G}_1[\omega_0](\gamma^{(s)}-\gamma^{(s)}_0)\big)\\+\frac12\Re\big(\mathcal{G}_2[z^{(s)}]\gamma^{(s)}_0-\mathcal{G}_2[i\omega_0]\gamma^{(s)}_0\big)+\frac12 \Re\big(\mathcal{G}_2[i\omega_0]\big(\gamma^{(s)}_0-2\sin\eta)(\alpha)\big)\\
+\frac12\Re\big(\mathcal{G}_2[i\omega_0](2\sin\eta)(\alpha)-\frac{\beta^2}{6}\cos\alpha\big).
\end{multline}
For (C.2), we also have
\begin{equation}\label{mathfrakC2}
\gamma^{(s)}-\gamma^{(s)}_0=-a_\mu\Re\big(\frac1i\mathcal{G}_1[z^{(s)}](\gamma^{(s)}-\gamma^{(s)}_0)-\frac1i\mathcal{G}_1[\omega_0](\gamma^{(s)}-\gamma^{(s)}_0)\big)-a_\mu\Re\big(\frac1i\mathcal{G}_2[z^{(s)}]\gamma^{(s)}\big).
\end{equation} Proposition \ref{propgamma} gives us
$$\|\gamma^{(s)}\|_1\leq C,\,\,\|\gamma_0^{(s)}\|_1\leq C.$$
By Lemma \ref{lem3.11} (for $\beta=0$ and $L^{(1)}=L^{(2)}=2\pi$), Note \ref{mathcalK2} and (\ref{mathfrakC2}), we have
$$\|\gamma^{(s)}-\gamma^{(s)}_0\|_r\leq C (\|\tilde\theta^{(s)}\|_r\|\gamma^{(s)}-\gamma^{(s)}_0\|_1+\beta^2\big).$$
Hence for sufficient small $\|\tilde\theta^{(s)}\|_r$, we have
\begin{equation}\label{gammau1}
\|\gamma^{(s)}-\gamma^{(s)}_0\|_r\leq C\beta^2.
\end{equation}
(\ref{mathfrakC2}) can be rewritten as
\begin{multline}\label{gammas1}
\gamma^{(s)}-\gamma^{(s)}_0+a_\mu\frac{\beta^2}{6}\sin\alpha=-a_\mu\Re\big(\frac1i\mathcal{G}[z^{(s)}](\gamma^{(s)}-\gamma^{(s)}_0)-\frac1i\mathcal{G}_1[\omega_0](\gamma^{(s)}-\gamma^{(s)}_0)\big)\\-a_\mu\Re\big(\frac1i\mathcal{G}_2[z^{(s)}]\gamma^{(s)}_0-\frac1i\mathcal{G}_2[i\omega_0]\gamma^{(s)}_0\big)-a_\mu\Re\big(\frac1i\mathcal{G}_2[i\omega_0]\big(\gamma^{(s)}_0-2\sin\eta)(\alpha)\big)\\
-a_\mu\Re\big(\frac1i\mathcal{G}_2[i\omega_0](2\sin\eta)(\alpha)-\frac{\beta^2}{6}\sin\alpha\big).
\end{multline}
We see from (C.2) that
\begin{multline*}
\gamma^{(s)}_0-2\sin\alpha=-a_\mu\Re\big(\frac1i\mathcal{G}_1[z^{(s)}]\gamma^{(s)}_0-\frac1i\mathcal{G}_1[\omega_0]\gamma^{(s)}_0\big)\\+2\Big(\sin(\alpha+\tilde\theta^{(s)}\big)-2\sin\alpha\Big)+2u_0\frac{\mu_2}{\mu_1+\mu_2}\sin(\alpha+\tilde\theta^{(s)}\big)+\sigma\tilde\theta^{(s)}_{\alpha\alpha}.
\end{multline*}
Hence by Lemmas \ref{lem3.1} and \ref{lem3.11} (for $\beta=0$ and $L^{(1)}=L^{(2)}=2\pi$), we have from above
\begin{equation}\label{gammas0}
\|\gamma^{(s)}_0-2\sin(\cdot)\|_1\leq C\big(\|\tilde\theta^{(s)}\|_r+|u_0|\big).
\end{equation}
We know the first derivative of $\mathcal{G}_2[i\omega_0](2\sin\eta)(\alpha)$ with respect to $\beta$ at $\beta=0$ is equal to 0. On calculation,
\begin{equation*}
\Big(\mathcal{G}_2[i\omega_0](2\sin\eta)(\alpha)\Big)_{\beta\beta}\Big|_{\beta=0}=\frac{e^{i\alpha}}{3}.
\end{equation*}
Hence for sufficiently small $\beta$, by Taylor expansion, we have
\begin{equation}\label{mathcalG2}
\left\|\mathcal{G}_2[i\omega_0](2\sin\eta)(\alpha)-\frac{\beta^2}{6}e^{i\alpha}\right\|_r\le C\beta^4.
\end{equation}
By Lemmas \ref{lem6.3}, \ref{lem3.11} (for $\beta=0$ and $L^{(1)}=L^{(2)}=2\pi$), Note \ref{mathcalK2}, (\ref{gammas0}) and (\ref{mathcalG2}), from (\ref{gammas1}) we get
\begin{equation}\label{mathcalG21}
\left\|\gamma^{(s)}-\gamma^{(s)}_0+a_\mu\frac{\beta^2}{6}\sin(\cdot)\right\|_r\leq C\big(\beta^2\|\tilde\theta^{(s)}\|_r+\beta^2u_0+\beta^4\big).
\end{equation}
Hence, by Lemma \ref{lem3.11}, (\ref{gammas0}), (\ref{mathcalG2}) and (\ref{mathcalG21}), the first statement is obtained.

The proof for the second statement follows similarly.
\end{proof}

Hence we have
\begin{lemma}
\label{lemathfrakNs}
For any $r \ge 3$, assume $\| {\tilde \theta} \|_r $, $u_0$ and $\beta$
are sufficiently small. Then there exists $C$ independent of $\beta$, $u_0$ and
${\tilde \theta}$ so that
\begin{equation}
\| \mathfrak{N}^{(s)} \|_{r-1}
\le C \left [ |u_0| \| {\tilde \theta}\|_r +|u_0|\beta^2+ \beta^4 + \beta^2\|\tilde\theta\|_r+
\| {\tilde \theta} \|_r \| {\tilde \theta} \|_{r-1} \right ].
\end{equation}
Further, suppose $\mathfrak{N}^{(s)}_1$ and $\mathfrak{N}^{(s)}_2$ correspond to $(\tilde\theta^{(s)}_1,u_0^{(1)},\beta)$ and $(\tilde\theta^{(s)}_2,u_0^{(2)},\beta)$ respectively, each in $\dot H^r$. Then there exists $C$ independent of $\beta$, $u_0$ and
${\tilde \theta}^{(s)}$ so that
\begin{multline*}
\big\|\mathfrak{N}^{(s)}_1-\mathfrak{N}^{(s)}_2\big\|_{r-1}\leq C\Big(\beta^2\big(\|\tilde\theta^{(s)}_1-\tilde\theta^{(s)}_2\|_r+|u_0^{(1)}-u_0^{(2)}|\big)
+\|\tilde\theta^{(s)}_1\|_{r-1}\|\tilde\theta^{(s)}_1-\tilde\theta^{(s)}_2\|_r\\+\|\tilde\theta^{(s)}_1\|_{r}\|\tilde\theta^{(s)}_1-\tilde\theta^{(s)}_2\|_{r-1}+|u_0^{(1)}|\|\tilde\theta^{(s)}_1-\tilde\theta^{(s)}_2\|_r+\|\tilde\theta^{(s)}_1\|_{r}|u_0^{(1)}-u_0^{(2)}|\Big).
\end{multline*}
\end{lemma}
\begin{proof}
Combining Lemmas \ref{lemathNs1}, \ref{lemathNs2} and \ref{lemathNs3}, the two statements are obtained.
\end{proof}
\begin{definition}\label{def7.4}We  define the linear operator $A$ on $u\in \dot{H}^r$ by
\begin{equation}\label{7.3}
Au=-\frac{\sigma}2
u_{\alpha\alpha}-\sum_{k=2}^\infty(1+a_\mu)\frac{k+1}{k+2}\hat{u}(k+1)e^{ik\alpha}-\sum_{k=-2}^{-\infty}(1+a_\mu)\frac{k-1}{k-2}\hat{u}(k-1)e^{ik\alpha}.
\end{equation}
\end{definition}
\begin{prop}\label{prop7.5} For $r\ge3$,
the linear operator $A: \dot{H}^r\rightarrow \dot{H}^{r-2}$,
 is invertible. Further, $\|A^{-1}f\|_r\leq C_r\|f\|_{r-2}$, for any $f\in\dot{H}^{r-2}$.
\end{prop}
\begin{proof}
For any surface tension $\sigma$, there exists the integer $K>2$
such that $n^2 \geq\frac{8}{\sigma}$ for any $|n|\geq K$.  Let us
define a family of  the spaces $Z_r:=\big\{u\in \dot{H}^r|
\mathcal{Q}_Ku=u\big\}$ with $r\geq 0$. We define the linear
operator $A_K:=\mathcal{Q}_KA$, which maps from $Z_r$ to $Z_{r-2}$.
The corresponding bilinear mapping $ E_K: Z_1\times Z_1 \rightarrow
\mathbb{R}$ is defined  by\begin{align*}
E_K[u,v]=2\Re\left(\sum_{k=K}^\infty\big[\frac{\sigma}{2}k^2\hat{u}(k)-(1+a_\mu)\frac{k+1}{k+2}\hat{u}(k+1)\big]\hat{v}(-k)\right)
\end{align*} for any $u,v\in Z_1$.

\medskip

It is easy to see that  there exist $a>0$  such that
$$\big|E_K[u,v]\big|\leq a \|u\|_{1}\|v\|_{1},$$
and \begin{eqnarray*} E_K[u,u]&\geq&\frac{\sigma}{2}
\|u\|_{1}^2-3\Big|\sum_{k=K}^\infty \hat{u}(k+1)\hat{u}(-k)
+\sum_{k=-K}^{-\infty}\hat{u}(k-1)\hat{u}(-k)\Big|\nonumber\\
&\geq& \frac{\sigma}2 \|u\|_{1}^2-2\|u\|_0^2\geq
\frac{\sigma}{4}\|u\|_1^2,
\end{eqnarray*}
the last inequality is the reason that for $|n|\geq K$, we have
$\frac{\sigma}{4}n^2\geq 2$.

\medskip

Hence by Lax-Milgram theorem, we see
that for any $f\in \dot{H}^{r-2}$, there exists only one $u_K\in Z_1$ such
that $E_K[u_K,v]=(\mathcal{Q}_Kf,v)_{L^2}$ for any $v\in Z_1$ and so
$A_Ku_K=\mathcal{Q}_Kf$ for some $u_K\in Z_1$. We also have
\begin{multline}\label{Ainverse}
\|\mathcal{Q}_Kf\|_r^2
\ge2\sum_{k=K}^\infty\frac{\sigma^2}{4}k^{2r}|\hat{u}_K(k)|^2-4\sum_{k=K}^\infty k^{2r-2}\frac{(k+1)^2}{(k+2)^2}|\hat{u}_K(k+1)|^2\\
\geq\frac{\sigma^2}{4}\|u_K\|_{r}^2-2\|u_K\|_{r-2}^2\ge \frac{\sigma^2}{8}\|u_K\|_{r}^2,
\end{multline}
for $\frac{\sigma}{4}n^2\geq 2$.

 Let us consider the linear operator $A$. It  can be written as
 \begin{equation*}
 Au=\sum_{k=2}^{K-1}\big(\frac{\sigma}{2}k^2\hat{u}(k)-(1+a_\mu)\frac{k+1}{k+2}\hat{u}(k+1)\big)e^{ik\alpha}
+A_KQ_ku+c.c.
\end{equation*}
for $u\in \dot{H}^{r}$.

\medskip

For any $f\in \dot{H}^{r-2}$, there exists only one solution $u_K\in
Z_{r}$ such that $A_Ku_K=\mathcal{Q}_Kf$. Then using $u_K$, we
consider the following finite linear equation system for
$\left ( b_{K-1}, b_{K-2} .. b_2, b_{-2}, \cdots,b_{-K+1} \right )^T$
\begin{multline}
\label{19}
{\begin{pmatrix} \frac{\sigma}{2} (K-1)^2 & 0 &0 &\cdots\cr
                -\frac{K-1}{K} (1+a_\mu) &  \frac{\sigma}{2} (K-2)^2 &0 &\cdots\cr
                  0  & - (1+a_\mu)\frac{K-2}{K-1} &
                      \frac{\sigma}{2} (K-3)^2 &\cdots\cr
. & . & . & \cdots\cr
. & . & . & \cdots\cr
. & . & . & \cdots\cr
               \end{pmatrix}}
{\begin{pmatrix} b_{K-1} \cr
                 b_{K-2} \cr
                 b_{K-3}      \cr
                 .\cr
                 .       \cr
                 b_{-K+1}       \cr
                \end{pmatrix}} \\
=
{\begin{pmatrix} {\hat f} (K-1) + (1+a_\mu) \frac{K}{K+1} {\hat u}_K (K)  \cr
                 {\hat f} (K-2)
                 \cr
                . \cr
                . \cr
{\hat f} (-K+1) + (1+a_\mu) \frac{-K}{-K+1} {\hat u}_K (-K)  \cr
                \end{pmatrix}}  .
\end{multline}
It is easy  to from the triangle structure see that there exists only one solution
$(b_{K-1},\cdots, b_2,b_{-2},\cdots,b_{-K+1})$. Then we choose
$u=\sum_{k=2}^{K-1}b_ke^{ik\alpha}+\sum_{k=-2}^{-K+1}b_k
e^{ik\alpha}+u_K$ and $Au=f$. Since $u_K\in
H^r_p $, we induce $u\in
H^r_p $.

\medskip

 Hence, for any $f\in \dot{H}^{r-2}$,
there is only one $u=A^{-1}f\in \dot{H}^r$. By (\ref{Ainverse}), $\|A^{-1}f\|_r\le C_r\|f\|_{r-2}$.
\end{proof}

\begin{prop}\label{prop5.4new}
For any  surface tension $\sigma>0$,
$r \ge 3$,  and sufficiently small $\epsilon$,
there exists  a neighborhood $O$ of (0,0) and  a ball  $\mathcal{B}_\epsilon^r\subset\dot{H}^r$ such that  $\tilde\theta^{(s)}:O\rightarrow \mathcal{B}_\epsilon^r$ with
$\mathcal{Q}_1\mathfrak{U}\big[\tilde\theta^{(s)}(u_0,\beta),u_0,\beta\big]=0$. Further, $\tilde\theta^{(s)}(u_0,\beta;\alpha)$ is odd with respect to $\alpha$ for any $(u_0,\beta)\in O$.
\end{prop}
\begin{proof}  We define the operator $\mathcal{T}$ by
$$\mathcal{T}\tilde\theta^{(s)}\equiv A^{-1}\mathcal{Q}_1\mathfrak{N}^{(s)}[\tilde\theta^{(s)},u_0,\beta].$$
By Lemma \ref{lemathfrakNs} and Proposition \ref{prop7.5}, for sufficient small $\epsilon$, there exists a neighborhood $O$ of $(0,0)$, such that the operator $\mathcal{T}$ is the contraction map in the ball $\mathcal{B}_\epsilon^r$ for $(u_0,\beta)\in O$.

Hence, by contraction mapping theorem,  there exist
open sets $ O\subset\mathbb{R}^2$  such that $\tilde\theta^{(s)}=\mathcal{T}\tilde\theta^{(s)}$ in the ball  $ \mathcal{B}_\epsilon^r\subset
\dot{H}^r$ for $(u_0,\beta)\in O$. By (\ref{stead1}), we have
$$\mathcal{Q}_1\mathfrak{U}\big[\tilde\theta^{(s)}(u_0,\beta;\alpha),u_0,\beta\big]=0.$$

For any $(u_0,\beta)\in  O$,   we define $\eta(\alpha)=-\tilde\theta^{(s)}(u_0,\beta;-\alpha)-\hat{\theta}^{(s)}(-1)e^{i\alpha}-\hat{\theta}^{(s)}(1)e^{-i\alpha}$, $v(\alpha)=-(z^{(s)}(-\alpha))^{\ast}$, and $\xi(\alpha)=-\gamma^{(s)}(-\alpha)$. Then it is easy to check that
 \begin{multline*}
\Re\Big(\frac{z^{(s)}_\alpha(-\alpha)}{2\pi i}\mbox{PV}\int_0^{2\pi}\gamma^{(s)}(\alpha')
K(-\alpha,\alpha')d\alpha'\Big)\\
=-\Re\Big(\frac{v_\alpha(\alpha)}{2\pi i}\mbox{PV}\int_0^{2\pi}\xi(\alpha')
\Big\{\frac{\beta}{4}\coth\Big[\frac{\beta}{4}\big(v(\alpha)-v(\alpha')\big)\Big]-\frac{\beta}{4}\tanh\Big[\frac{\beta}{4}\big(v(\alpha)-v^{\ast}(\alpha')\big)\Big]\Big\}d\alpha'\Big).
\end{multline*}
Hence, $\mathcal{Q}_1\mathfrak{U}[\mathcal{Q}_1\eta(\alpha),u_0,\beta]=\mathcal{Q}_1\mathfrak{U}[\tilde\theta^{(s)}(\alpha),u_0,\beta]=0$ with $\xi(\alpha)$ satisfying (C.2).

Also by uniqueness, we have $\tilde\theta^{(s)}(u_0,\beta;\alpha)=\mathcal{Q}_1\eta(\alpha)\equiv-\tilde\theta^{(s)}(u_0,\beta;-\alpha)$.
\end{proof}

\begin{note}
Note that the $\tilde\theta^{(s)}$ of Proposition \ref{prop5.4new}
is not the steady state since we only required $\mathcal{Q}_1U=0$
instead of $\mathcal{Q}_0 U=U=0$. Here $u_0$ is arbitrary.  The
additional condition $(\mathcal{Q}_0-\mathcal{Q}_1)U=0$ can be
satisfied by constraining $u_0$ appropriately. The usefulness of
this Proposition is to show any steady solution $\theta^{(s)}$ that
actually satisfies $\mathcal{Q}_0U=0$ must be an odd function since
this is true for any sufficient small $u_0$.
\end{note}
\begin{definition} \label{def5.5}
We define  a family of  Banach spaces $\{X_r\}_{r\geq0}$   by
\begin{align*}
X_r&=\big\{u\in \dot{H}^{r}|
 u(-\alpha)=-u(\alpha)\big\},\\
Y_r&=\big\{u\in
H^{r}_p  \big| \mathcal{Q}_0 u=u,
u(-\alpha)=u(\alpha)\big\}.
\end{align*}
\end{definition}
\begin{remark}
 Proposition \ref{prop5.4new} shows us that the shape of the steady bubble must be symmetric with the center of the channel. Also $\mathfrak{U}:X_r\times\mathbb{R}^2\rightarrow Y_{r-2}$.
 Hence it is reasonable to consider the  solution $(\tilde\theta^{(s)}, u_0, \beta)$ to (C.1)-(C.3) in the space $X_r\times \mathbb{R}^2$.
 \end{remark}

\noindent{\bf Proof of Theorem \ref{theo5.6}:} Let $f=\mathfrak{N}^{(s)}[\tilde\theta^{(s)},u_0,\beta]-\frac{\beta^2}2\mathfrak{U}_{\beta\beta}[0,0,0]$ and $g=A^{-1}\mathcal{H}(\mathcal{Q}_1f)$.
Actually it is easy to check that $f(-\alpha)=f(\alpha)$ and $g(-\alpha)=-g(\alpha)$ for $\tilde\theta^{(s)}\in X_r$.

We define
  an operator $\mathfrak{T}$ in $X_r\times\mathbb{R}$ by
$$\mathfrak{T}[\tilde\theta^{(s)},u_0]=\Big(A^{-1}\mathcal{H}(\mathcal{Q}_1f),2\hat{f}(1)+\frac{\mu_1+\mu_2}{\mu_1}\big(\frac{4}{3}+\frac{4}{3}a_\mu\big)i\hat{g}(2)\Big)^{T}.$$
By  Lemma \ref{lemathfrakNs} and Proposition \ref{prop7.5}, for sufficient small $\epsilon$, there exist an open set  $ O_1\subset\mathbb{R}$ and  a ball $ O_2\subset
X_r\times\mathbb{R}$ such that $\mathfrak{T}$ is the contraction map in the ball $O_2$ for any $\beta\in O_1$.
Hence, by contraction mapping theorem, we have $(\tilde\theta^{(s)},u_0)^{T}=\mathfrak{T}[\tilde\theta^{(s)},u_0]$ for any $\beta\in O_1$. By (\ref{stead1}), we have
$$\mathcal{Q}_0\mathfrak{U}\big[\tilde\theta^{(s)}(\beta;\alpha),u_0(\beta),\beta\big]=0.$$
By Lemma \ref{lemathfrakNs} and Proposition \ref{prop7.5}, for sufficiently small $\epsilon$ and $\Upsilon$, there exists $C$ independent of $\epsilon$ and $\Upsilon$, such that
\begin{equation}\label{thetasbeta}
\|\tilde\theta^{(s)}\|_r+|u_0|\le C\beta^2.\end{equation}

We deduce from (C.2) that
\begin{align*}
\gamma^{(s)}(\alpha)=&2\sin\alpha-a_\mu\Re\Big(\frac{z^{(s)}_{\alpha}}{\pi
i} \mbox{PV}\int_{\alpha-\pi}^{\alpha+\pi} \frac{\gamma^{(s)}(\alpha')}{
z^{(s)}(\alpha)-z^{(s)}(\alpha')}d\alpha'-\frac{e^{i\alpha}}{\pi
i} \mbox{PV}\int_{\alpha-\pi}^{\alpha+\pi}\frac{\gamma^{(s)}(\alpha')}{
\int_{\alpha'}^\alpha e^{i\zeta}d\zeta}d\alpha'\Big) \nonumber\\
&-a_\mu\Re\Big(\frac{\omega_{s_\alpha}}{2\pi i}\sum_{n=1}^\infty
\frac{2B_{2n}}{(2n)!}(-1)^n\beta^{2n}\int_0^{2\pi}\gamma^{(s)}(\alpha')\big(z^{(s)}(\alpha)-z^{(s)}(\alpha')
\big)^{2n-1}d\alpha'\Big)\nonumber\\
&+2\Big(\sin\big(\alpha+\theta^{(s)}(\alpha)\big)-\sin\alpha\Big)+\sigma\theta_{{\alpha\alpha}}^{(s)}
+\frac{2\mu_2}{\mu_1+\mu_2}u_0\sin\big(\alpha+\theta^{(s)}(\alpha)\big),
\end{align*}
where $B_n$ is $n$th Bernoulli number.
By (\ref{thetasbeta}) and Lemma \ref{lem3.11}, we have
\begin{eqnarray*}
\|\gamma^{(s)}-2\sin\alpha\|_{r-2}\leq C\beta^2,
\end{eqnarray*}
where $C$ depends on  $\epsilon$ and $\Upsilon$.

\begin{remark}
Since we consider the steady solution in $\dot{H}^r$ for $r\ge3$,
 where $r$ is arbitrary, by uniqueness shown in Theorem \ref{theo5.6}, the steady solution is in $H^{\infty}$, and hence in $C^\infty$.
 The result is consistent with analyticity results for arbitrary channel width in the small $\sigma$ limit \cite{Xie3}.
\end{remark}

\begin{note}
Actually for $\mu_2=0$, by formal expansion in correspondence to earlier calculation using conformal mapping \cite{Tanveer2}, we have
\begin{eqnarray*}
\theta^{(s)}(\alpha)&=&\beta^4\big(\frac{1}{54\sigma}\sin{(3\alpha)}+\frac{1}{72\sigma^2}\sin{(2\alpha)}\big)+O(\beta^6),\nonumber\\
u_0&=&-\frac{\beta^2}{6}+\beta^4\big(\frac7{180}+\frac1{216\sigma^2}\big)+O(\beta^6),\nonumber\\
\gamma^{(s)}(\alpha)&=&2\sin\alpha-\frac{\beta^2}6\sin\alpha+\beta^4\Big(\big(-\frac{19}{120}+\frac{1}{72\sigma^2}\big)\sin(3\alpha)+\big(\frac{1}{72}+\frac{7}{216\sigma^2}\big)
\sin\alpha\\&&\hspace{4cm}+\frac{1}{54\sigma}\sin(4\alpha)
-\frac{1}{54\sigma}\sin(2\alpha)\Big) +O(\beta^6).\nonumber
\end{eqnarray*}
For steady states, two fluid flows  can  be related to one fluid flow by transform variables \cite{tanveer5}.
\end{note}

\section{Evolution of symmetric bubble with sidewalls ($\beta > 0$) }

\begin{lemma}
\label{symmetry}
If initial conditions satisfy the symmetry condition
\begin{equation*}
\theta_0 (-\alpha) = -\theta_0 (\alpha) ~~,~~y_0 = 0,
\end{equation*}
then the corresponding solution $\left ( \theta (\alpha, t), L(t),
y(0, t) \right ) $ in $H^r_p \times \mathbf{C}^1 \times \mathbf{C}^1 $
to (A.1)
and (\ref{y0}) satisfy symmetry condition
for all time, {\it i.e.} $\theta (-\alpha, t) = -\theta (\alpha, t)$ and
$y(0, t) = 0$.
The corresponding vortex sheet strength $\gamma (\alpha, t)$, determined
from (A.2) also obeys the symmetry condition
$\gamma (-\alpha, t) = - \gamma (\alpha, t)$  and the bubble shape
is symmetric about the channel centerline ($x$-axis).
\end{lemma}

\begin{proof}
If $\theta_0$ is odd and $y(0, 0)=0$,
it follows from
(\ref{eqomega}) and (\ref{zomega}), that $z^* (\alpha, 0) = z(-\alpha, 0)$
and we have a symmetric bubble to start with. The corresponding vortex
sheet strength determined from (A.2) $\gamma (\alpha, t)$
is easily to be odd.
Again, it is readily checked that
that if $\left ( \theta (\alpha, t), \gamma (\alpha, t),
L(t),  y(0, t) \right )$ solve
(A.1)-(A.3) and (\ref{y0}), then so does
$\left ( -\theta (-\alpha, t), -\gamma (-\alpha, t),
L(t),  -y(0, t) \right )$. Since the initial condition is symmetric,
it follows from local uniqueness of solution
(see Appendix \S \ref{A7.2} ) that symmetry is preserved in time. From the geometric
relation
$$ z (\alpha, t) = \frac{i L(t)}{2\pi} \int_0^\alpha
e^{i \alpha + i \theta (\alpha, t) } d\alpha + z(0, t),$$
symmetry about the $x$ ($Re ~z$) axis follows.
\end{proof}

\begin{remark}
\label{remsteadyshape}
Symmetry implies ${\hat \theta} (0; t) =0 = y(0, t)$.
and so the evolution equation for ${\hat \theta} (0; t)$ in
(B.1) and $y(0,t)$ in (\ref{y0}) can be ignored. For the symmetry bubble, we also have
$$K(\alpha,\alpha')=\frac{\beta}{2}\coth\Big[\frac{\beta}{2}\big(z(\alpha)-z(\alpha')\big)\Big].$$
Proposition \ref{prop5.4new}
 implies that the steady bubble solution
$\left (\theta^{(s)} (\alpha), \gamma^{(s)} (\alpha) \right )$
are also odd functions of time.
 \end{remark}
\medskip

\subsection{Main results for the translating bubble in the strip}

In this section, we first state the main results for the translating bubble.

It is convenient to define
\begin{definition}
\label{def5.8}
\begin{align*}
\Gamma(\alpha,t)&=
\gamma(\alpha,t)-\gamma^{(s)}(\alpha),
\end{align*}
\begin{equation}
\theta(\alpha,t)
=\widetilde\Theta(\alpha,t)+{\tilde\theta}^{(s)}(\alpha)
+\hat{\theta}(-1;t) e^{-i\alpha}+\hat{\theta}(1;t)e^{i\alpha}.
\end{equation}
\end{definition}

In this section, we will find
solutions $\widetilde\Theta$ to satisfy (B.1) with initial condition
with the initial  condition
\begin{equation}
\label{2p5.11}
\widetilde{\Theta} (\alpha, 0) = \widetilde{\Theta}_0 (\alpha)
\equiv \mathcal{Q}_1 \left [ \theta_0 - \theta^{(s)} \right ] (\alpha).
\end{equation}
We will also consider the motion of the interface
with small symmetric perturbation around the steady  bubble.
Since the bubble area is invariant with time,
we take  $V$ to be  the steady bubble area, i.e..
\begin{eqnarray} \label{2p5.10}
 V=\frac{1}{2}\Im\int_0^{2\pi}z^{(s)}_{{\alpha}}(z^{(s)})^{\ast}d\alpha.
\end{eqnarray}

The main result in this section is the following proposition:
\begin{prop}\label{prop8}{ For $\sigma>0$,  there exist $\epsilon,\,\Upsilon>0$ such that for
$r\geq3$, if $\|\widetilde\Theta(\cdot,0)\|_r<\epsilon$,
$0<\beta<\Upsilon$,  then for  initial  shape symmetric about channel
centerline, i.e. $\widetilde\Theta(-\alpha,0)=-\widetilde\Theta(\alpha,0)$,
there
exists a global solution $\widetilde\Theta\in {\dot H}_r $
to the Hele-Shaw initial value problem
with initial  condition (\ref{2p5.11}).
Furthermore, $\|\mathcal{Q}_0\Theta\|_r$
decays exponentially as $t \rightarrow \infty$.
Thus the  translating steady bubble is
asymptotically stable for sufficiently small symmetric initial disturbances in the $H^r_p $ space.}
\end{prop}

\begin{note}\label{translatingnote5.7}
Proposition \ref{prop8} and Lemma \ref{lem2.7} imply  Theorem \ref{theo8}.
\end{note}

\subsection{Evolution equation in integral form}

It is readily checked that $\Gamma (\alpha, t)$ satisfies
\begin{multline}
\label{D.2}
\big(I+a_\mu\mathcal{F}[z]\big)\Gamma=
-a_\mu\mathcal{F}[z]\gamma^{(s)}+a_\mu\mathcal{F}[z^{(s)}]\gamma^{(s)}
+\frac{2\pi-L}{L}\sigma\theta_{\alpha\alpha}
+\sigma(\theta-\theta^{(s)})_{\alpha\alpha}\\
+\frac{L-2\pi}{\pi}\big(1+\frac{\mu_2}{\mu_1+\mu_2}u_0\big)
\sin\big(\alpha+\theta\big)\\
+2\big(1+\frac{\mu_2}{\mu_1+\mu_2}u_0\big)
\Big(\sin\big(\alpha+\theta)-\sin\big(\alpha+\theta^{(s)}\big)\Big).
\end{multline}
Hence, we have
\begin{prop}
\label{translatingGamma}
If  $\widetilde\Theta\in\dot{H}^r$ with
$\|\widetilde\Theta\|_1<\epsilon_1$, and $0\le\beta<\Upsilon$
then for sufficiently small $\epsilon_1$ and $\Upsilon$,
there
exists a unique solution $ \Gamma \in \{u\in H^{r-2}_p
|\hat{u}(0)=0\}$ for
$r\ge 3$ satisfying (\ref{D.2}).
This solution $\Gamma$ satisfies the estimates
\begin{eqnarray*}
\| \Gamma \|_0  &\leq&  C\|
\widetilde\Theta \|_2, \\
\|\Gamma\|_{r-2}&\leq &C_1\exp(C_2\|\widetilde\Theta\|_{r-2})
\|\widetilde\Theta\|_r,
\end{eqnarray*}
where  $C_1$ and  $C_2$ depend on $r$.

Let $\Gamma^{(1)} $  and $\Gamma^{(2)} $ correspond to ${\widetilde
\Theta}^{(1)}$ and ${\widetilde \Theta}^{(2)}  $ respectively. Assume $\|\widetilde\Theta^{(1)}\|_1<\epsilon_1$ and $\|\widetilde\Theta^{(2)}\|_1<\epsilon_1$. If $\widetilde\Theta^{(1)},\widetilde\Theta^{(2)}\in\dot{H}^r$ with $r\ge3$,   then for sufficient small $\epsilon_1$,
\begin{equation}\label{gamma4.61}
 \| \Gamma^{(1)} - \Gamma^{(2)} \|_{r-2} \le C_1\exp\big(C_2(\|\widetilde\Theta^{(1)}\|_r+\|\widetilde\Theta^{(2)}\|_r)\big)
\big \| {\widetilde \Theta}^{(1)} - {\widetilde \Theta}^{(2)}\big \|_r
\end{equation}
where  $C_1$ and $C_2$ depend on $r$ alone.
\end{prop}
\begin{proof}
In Proposition \ref{propgamma}, we take
$\gamma^{(2)} = \gamma$, ${\tilde \theta}^{(1)} =
{\tilde \theta}$, $L^{(1)} = L$,
$$\gamma^{(2)} = \gamma^{(s)} ,~~{\tilde \theta}^{(2)}
= \tilde\theta^{(s)}, ~~L^{(2)} = 2 \pi
$$
and use Lemma  \ref{lem3.11} to obtain the first two statements.
The statement (\ref{gamma4.61}) follows in a similar
manner from (\ref{eq2prop319}).

\end{proof}

\medskip
The evolution equation (B.1) translates into the following
equation for
$\Theta$:
\begin{equation}
\label{eqTheta}
\widetilde{\Theta}_t(\alpha,t)=
\frac{2\pi}{L}\mathcal{Q}_1\big(U_\alpha+T(1+\theta_\alpha)\big)
= \mathcal{A} \big [\widetilde{\Theta} \big]
+ \mathcal{L}_\beta [\widetilde{\Theta} ] + \mathcal{N} [
\widetilde{\Theta} ].
\end{equation}
where $L$ is determined from (B.4) with $V$ determined from (\ref{2p5.10}).

We can integrate the evolution equation (\ref{eqTheta}) and
rewrite it as the following integral equation:
\begin{equation}
\label{eqfixed}
\widetilde{\Theta} (\alpha, t) = e^{t \mathcal{A}} \widetilde{\Theta}_0 +
\int_0^t e^{(t-\tau) \mathcal{A}}
\left ( \mathcal{L}_\beta [ \widetilde{\Theta} ] + \mathcal{N} [ \widetilde{
\Theta} ] \right ) (\alpha, \tau) d\tau \equiv \mathcal{R} [ \widetilde{\Theta}
] (\alpha, t).
\end{equation}
We will eventually show that $\mathcal{R}$ defines a contraction in
a sufficiently small ball in the $X_r$ space for $r \ge 3$.
For that purpose we need some properties.
\begin{prop}
\label{mathcalNprop}
If for $r\ge 3$, ${\widetilde \Theta} \in \dot{H}^r $ with
$\|\widetilde\Theta\|_1 <\epsilon_1$, and $0\le\beta<\Upsilon$,
then for sufficiently small $\epsilon_1$ and $\Upsilon$,
the functions $\mathcal{L}_\beta $,
and  $\mathcal{N} $, defined in Appendix (\S 7.3),
satisfy the following estimates
\begin{eqnarray*}
\Big\|\mathcal{L}_\beta\Big\|_{r-1}&\leq&C_1\beta^2\exp\big(C_2\|\widetilde\Theta\|_r\big)\|\widetilde\Theta\|_r,\nonumber\\
\Big\|\mathcal{N}\Big\|_{r-1}&\leq&C_1\exp\big(C_2\|\widetilde\Theta\|_r\big)\|\widetilde\Theta\|_r\|\widetilde\Theta\|_{r+1},\nonumber
\end{eqnarray*}
where $C_1$ and  $C_2$ depend only on $r$.
Further, let $\left (\mathcal{L}_\beta^{(1)}, \mathcal{N}^{(1)} \right ) $
and $\left (\mathcal{L}_\beta^{(2)}, \mathcal{N}^{(2)} \right ) $
correspond to
$\widetilde \Theta^{(1)}$ and ${\widetilde \Theta}^{(2)}$ respectively,
each in $\dot{H}^r$ with $\| {\widetilde \Theta}^{(1)} \|_1$ and
$\| {\widetilde \Theta}^{(2)} \|_1 $ $ < \epsilon_1$.
Then for sufficiently small $\epsilon_1$,
\begin{eqnarray*}
\left\| \mathcal{L}^{(1)}_\beta - \mathcal{L}^{(2)}_\beta
 \right\|_{r-1}
&\le& C_1\beta^2\exp{\Big(C_2\big(\| {\widetilde
\Theta}^{(1)}\|_r + \|{\widetilde \Theta}^{(2)}\|_r\big)\Big)}\big\| {\widetilde
\Theta}^{(1)} - {\widetilde \Theta}^{(2)}\big \|_r ,\\
\left\| \mathcal{N}^{(1)} - \mathcal{N}^{(2)}
 \right\|_{r-1}
&\le& C_1\exp{\Big(C_2\big(\| {\widetilde
\Theta}^{(1)}\|_r + \|{\widetilde \Theta}^{(2)}\|_r\big)\Big)}\Big\{\big(\| {\widetilde
\Theta}^{(1)}\|_r + \|{\widetilde \Theta}^{(2)}\|_r\big) \big\| {\widetilde
\Theta}^{(1)} - {\widetilde \Theta}^{(2)}\big \|_{r+1} \\
&&+\big(\| {\widetilde
\Theta}^{(1)}\|_{r+1} + \|{\widetilde \Theta}^{(2)}\|_{r+1}\big) \big\| {\widetilde
\Theta}^{(1)} - {\widetilde \Theta}^{(2)}\big \|_r\Big\},
\end{eqnarray*}
where $C_1$ and $C_2$ depend on $r$.
\end{prop}
\begin{proof} On using Lemmas  \ref{coroXi} (see Note \ref{notecoroXi}),   \ref{insertnew}, \ref{insertnew2},
\ref{leminsertnew4}, \ref{lem3.11}, \ref{lem6.3} and Proposition \ref{translatingGamma},
the proof follows from the expressions of $\mathcal{L}_\beta $ and
$\mathcal{N}$.
\end{proof}

\begin{remark} It is easily to check that $\big(\mathcal{L}_\beta[\widetilde\Theta]+\mathcal{N}[\widetilde\Theta]\big)(-\alpha)
=-\big(\mathcal{L}_\beta[\widetilde\Theta]+\mathcal{N}[\widetilde\Theta]\big)(\alpha)$.
\end{remark}

\subsection{Contraction properties of $\mathcal{R}$ and global
existence for symmetric disturbances}

\begin{note}\label{notesym}
For the linear evolution equation  (\ref{eqnv}), if $f$ and $v_0$ are odd with respect to $
\alpha$, then by uniqueness of the linear equation (\ref{eqnv}), $v(-\alpha,t)=-v(\alpha,t)$.
\end{note}
 First, by Proposition \ref{mathcalNprop}, we have
\begin{lemma} \label{lemmathcalN}
Assume $0\le\beta<\Upsilon$. Suppose for $r \ge 3$
$  {\widetilde \Theta} (\alpha, t)\in X_r$
satisfy the condition $
\| {\widetilde \Theta} \|_{H_\sigma^r} \le \epsilon$.
Then for $\mathcal{L}_\beta [{\widetilde \Theta}] (\alpha,t)$ and
$\mathcal{N} [ {\widetilde \Theta} ] (\alpha, t)$ determined
from the Appendix (\S 7.3), as $\epsilon$ and $\Upsilon$ are small enough,  we have
\begin{equation*}
\big\| \mathcal{L}_\beta [{\widetilde \Theta} ] +\mathcal{N} [ {\widetilde \Theta} ]
\big\|_{H_\sigma^{r-3}} \le
C\| {\widetilde \Theta} \|_{H_\sigma^r}
\left (
\| {\widetilde \Theta} \|_{H_\sigma^r} + \beta^2\right ) .
\end{equation*}
Further, if both
$ {\widetilde \Theta}^{(1)} (\alpha, t)$
and
$ {\widetilde \Theta}^{(2)} (\alpha, t)$
satisfy
(\ref{eqfixed}), then the corresponding
$\left (\mathcal{L}_\beta^{(1)}, \mathcal{N}^{(1)}\right )$
and $\left ( \mathcal{L}_\beta^{(2)},\mathcal{N}^{(2)}\right )$ satisfy
\begin{equation*}
\|\mathcal{L}_\beta^{(1)}-\mathcal{L}_\beta^{(2)}\|_{H_\sigma^{r-3}}+\| \mathcal{N}^{(1)}
- \mathcal{N}^{(2)}
\|_{H_\sigma^{r-3}} \le
C(\epsilon +\beta^2)
\| {\widetilde \Theta}^{(1)}-{\widetilde \Theta}^{(2)}
\|_{H_\sigma^r} .
\end{equation*}
\end{lemma}

\medskip

Hence, by Lemmas \ref{nonhomogeneous} and \ref{lemmathcalN}, we have
\begin{lemma}\label{lem8.11} Assume $0\le\beta<\Upsilon$.
 Let $ r\geq3$, $\|\widetilde\Theta_0\|_{w,r}<\frac{\epsilon}2$ and $\widetilde\Theta\in X_r$ with $\|\widetilde\Theta\|_{H^r_\sigma}\le\epsilon$. For sufficiently small  $\epsilon$ and $\Upsilon$,   the operator $\mathcal{R}$ defined in (\ref{eqfixed}) satisfies
the following estimate:
\begin{eqnarray*}
\big\|\mathcal{R}\big[\widetilde\Theta\big]\big\|_{H_\sigma^r}&\leq&
C\epsilon.
\end{eqnarray*}
Further, if $\|\widetilde\Theta^{(1)}\|_{H^r_\sigma}\le\epsilon$ and $\|\widetilde\Theta^{(2)}\|_{H^r_\sigma}\le\epsilon$, then
\begin{eqnarray*}
\big\|\mathcal{R}\big[\widetilde\Theta^{(1)}\big]-\mathcal{R}\big[\widetilde\Theta^{(2)}\big]\big\|_{H_\sigma^r}&\leq&
C\epsilon\|\widetilde\Theta^{(1)}-\widetilde\Theta^{(2)}\|_{H_\sigma^r}.
\end{eqnarray*}
 Further,  $\mathcal{R}[\widetilde\Theta](-\alpha)=-\mathcal{R}[\widetilde\Theta](\alpha)$.
\end{lemma}

\medskip

\noindent{\bf Proof of Proposition \ref{prop8}:}
If
$C \epsilon < 1$, then it is clear that
the right side of (\ref{eqfixed}) define a contraction map in an $\epsilon$ ball in
the Banach
space $X_r\cap H_\sigma^r$. Therefore, there exists a unique solution
${\widetilde \Theta} $
satisfying the equation (\ref{eqfixed}), hence (B.1).
The local uniqueness of solutions (see Appendix \S \ref{A7.2}) 
implies that this
is the only solution. The $e^{-\sigma t/2}$ exponential decay of
${\widetilde \Theta}$ and hence of $\Theta$ implies that the steady symmetric translating bubble
is approached exponentially in time. The constraint condition (B.4)
shows that $L-2\pi$ decays exponentially.

\vspace{1cm}
{\bf{Acknowledgements:}}
Partial support for this research was provided by the U.S. National Science
(DMS-0733778, DMS-0807266).

\section{appendix}

\subsection{Proof of Lemma \ref{coroXi}}\label{A7.1}
 Consider $j_0=1$ firstly.
Let $F(u)=uh(u)$. Then $h(u)$ is also an entire function of order 1.
\begin{eqnarray}\label{translating6.2}
\left\|F\big(u(\cdot)\big)\right\|_\infty\leq
C_1\exp(C_2\|u\|_\infty)\|u\|_\infty\leq C_1\exp\left(C_2\|u\|_1\right)\|u\|_1.
\end{eqnarray}
We see
$$\left\|D_\alpha F(u(\cdot))\right\|_0=\left\|u_\alpha D_uF\right\|_0\le C_1\exp\left(C_2\|u\|_1\right)\|u\|_1.$$
 For $k\ge2$, by Banach Algebra property, we also have
\begin{multline}\label{new7.2}
\Big\|D_\alpha F\big(u(\alpha)\big)\Big\|_{k-1}\le C\|D_\alpha
u\|_{k-1}\|D_uF(u(\alpha))\|_{k-1}\leq C\|u\|_k\sum_{j=1}^\infty
|a_j|j\|u\|_{k-1}^{j-1}\\
\leq C_1\|u\|_k\exp\big(C_2\|u\|_{k-1}\big).
\end{multline}
Hence, by (\ref{translating6.2}) and (\ref{new7.2}), we have for
$k\ge2$,
\begin{equation}\label{new7.3}
\left\|F\big(u(\cdot)\big)\right\|_k\leq
C_1\|u\|_k\exp\big(C_2\|u\|_{k-1}\big),
\end{equation}
with $C_1$ and $C_2$ depending only on $k$.

Let $F(u)=u^2g(u)$. Then $g(u)$ is also an entire function of order $1$.
\begin{eqnarray*}
\left\|F\big(u(\cdot)\big)\right\|_\infty\leq
C\exp(\|u\|_\infty)\|u\|^2_\infty\leq C\exp\left(\|u\|_1\right)\|u\|_1^2.
\end{eqnarray*}
 And
$D_u F(u)$ is the entire function of order 1 with $j_0=1$ , so for
$k\ge 2$, by  Bananch Algebra and (\ref{new7.3}),  we have
\begin{eqnarray*}
\left\|D_\alpha F\big(u(\cdot)\big)\right\|_{k-1}\leq
C\|u_\alpha\|_{k-1}\|D_uF(u(\alpha))\|_{k-1}\leq
C_1\|u\|_k\|u\|_{k-1}\exp\big(C_2\|u\|_{k-1}\big)
\end{eqnarray*}
with $C_1$ and $C_2$ depending only on $k$. Hence, for $k\ge2$,
\begin{equation}\label{new7.4}
\left\|F\big(u(\cdot)\big)\right\|_k\leq
C_1\|u\|_{k-1}\|u\|_k\exp\big(C_2\|u\|_{k-1}\big),
\end{equation}
with $C_1$ and $C_2$ depending only on $k$.

By the same technique, we obtain the difference results.

\subsection{Local uniqueness of Hele-Shaw bubble solutions}\label{A7.2}

We  have the local uniqueness theorem for the system (B.1)-(B.6) as follows:
\begin{theorem}\label{uniqueness}Let $0\leq\beta<\Upsilon$ and  $|u_0|<1$, where $\Upsilon$ is small enough for Lemmas
\ref{lem6.2}, \ref{lem3.11} and Proposition \ref{propgamma} to apply.
Let $\big(\tilde\theta_1(\alpha,t), \hat{\theta}_1(0;t),y_1(0,t)\big)$ and $\big(\tilde\theta_2(\alpha,t), \hat{\theta}_2(0;t),y_2(0,t)\big)$ be solutions of the system (B.1)-(B.6) with the same initial condition (\ref{translating5.1}) in  the space $C\left([0,S],\mathcal{B}^r_\epsilon\times\mathbb{R}\times S_M\right)$ with $r\ge4$.
 Suppose $\|\tilde\theta_1\|_1<\epsilon_1$ and $\|\tilde\theta_2\|_1<\epsilon_2$ such that $|L_1-2\pi|<\frac{1}{2}$ and $|L_2-2\pi|<\frac{1}{2}$ by  (\ref{6.14}).
Then for sufficient small $\epsilon_1$ and $\Upsilon$, the two solutions are the same in $\dot{H}^2\times\mathbb{R}\times S_M$.
\end{theorem}

 \begin{proof}We define the energy function $E^d(t)$ for the  difference of two solutions by
\begin{multline}\label{differenceenergy}
E^d(t)=\frac{1}{2}\int_0^{2\pi}\big( D^2_\alpha\tilde\theta_1- D^2_\alpha\tilde\theta_2\big)^2d\alpha+\frac{1}{2}\big(\hat{\theta}_1(0;t)-\hat{\theta}_2(0;t)\big)^2\\+\frac{1}{2}\big(y_1(0,t)-y_2(0,t)\big)^2.
\end{multline}
Taking derivatives on both sides with respect to $t$, and using (B.1)-(B.6), we have using (\ref{1.12})
\begin{multline}\label{derienergy}
\frac{dE^d(t)}{dt}=\int_0^{2\pi}D^2_\alpha\big(\tilde\theta_1-\tilde\theta_2\big)D^3_\alpha\mathcal{Q}_1\big(\frac{2\pi}{L_1}U_1-\frac{2\pi}{L_2}U_2\big)d\alpha\\
+\int_0^{2\pi}D^2_\alpha\big(\tilde\theta_1-\tilde\theta_2\big)D_\alpha\mathcal{Q}_1\big(\frac{2\pi}{L_1}(1+\theta_{1,\alpha})U_1-\frac{2\pi}{L_2}(1+\theta_{2,\alpha})U_2\big)d\alpha\\
+\int_0^{2\pi}D^2_\alpha\big(\tilde\theta_1-\tilde\theta_2\big)D^2_\alpha\mathcal{Q}_1\big(\frac{2\pi}{L_1}T_1\theta_{1,\alpha}-\frac{2\pi}{L_2}T_2\theta_{2,\alpha}\big)d\alpha\\
+\big(\hat{\theta}_1(0;t)-\hat{\theta}_2(0;t)\big)\int_0^{2\pi}\big[\frac{2\pi}{L_1}T_1(1+\theta_{1,\alpha})-\frac{2\pi}{L_2}T_2(1+\theta_{2,\alpha})\big]d\alpha\\+\big(y_1(0,t)-y_2(0,t)\big)\big[-U_1(0,t)\sin\big(\theta_1(0,t)\big)+U_2(0,t)\sin\big(\theta_2(0,t)\big)\big]
\end{multline}
\begin{multline*}
=\int_0^{2\pi}D^2_\alpha\big(\tilde\theta_1-\tilde\theta_2\big)(D^3_\alpha+D_\alpha)\mathcal{Q}_1\big(\frac{2\pi}{L_1}U_1-\frac{2\pi}{L_2}U_2\big)d\alpha\\
+\int_0^{2\pi}D^2_\alpha\big(\tilde\theta_1-\tilde\theta_2\big)D_\alpha\mathcal{Q}_1\big(\frac{2\pi}{L_1}\theta_{1,\alpha}U_1-\frac{2\pi}{L_2}\theta_{2,\alpha}U_2\big)d\alpha\\
+\int_0^{2\pi}D^2_\alpha\big(\tilde\theta_1-\tilde\theta_2\big)D_\alpha^2\mathcal{Q}_1\big(\frac{2\pi}{L_1}T_1\theta_{1,\alpha}-\frac{2\pi}{L_2}T_2\theta_{2,\alpha}\big)d\alpha\\
+\big(\hat{\theta}_1(0;t)-\hat{\theta}_2(0;t)\big)\int_0^{2\pi}\big[\frac{2\pi}{L_1}T_1(1+\theta_{1,\alpha})-\frac{2\pi}{L_2}T_2(1+\theta_{2,\alpha})\big]d\alpha\\+\big(y_1(0,t)-y_2(0,t)\big)\big[-U_1(0,t)\sin\big(\theta_1(0,t)\big)+U_2(0,t)\sin\big(\theta_2(0,t)\big)\big]
=I_1+I_2+I_3+I_4+I_5.
\end{multline*}

By (\ref{1.12}), we have
\begin{eqnarray*}
I_1&=&\int_0^{2\pi}D^2_\alpha\big(\tilde\theta_1-\tilde\theta_2\big)(D^3_\alpha+D_\alpha)\mathcal{Q}_1\big(\frac{2\pi^2}{L_1^2}\mathcal{H}[\gamma_1]-\frac{2\pi^2}{L_2^2}\mathcal{H}[\gamma_2]\big)d\alpha\\
&&+\int_0^{2\pi}D_\alpha^2\big(\tilde\theta_1-\tilde\theta_2\big)(D_\alpha^3+D_\alpha)\mathcal{Q}_1\Big(\frac{2\pi^2}{L_1^2}\Re\big(\mathcal{G}[z_1]\gamma_1\big)-\frac{2\pi^2}{L_2^2}\Re\big(\mathcal{G}[z_2]\gamma_2\big)\Big)d\alpha\\
&&+(u_0+1)\int_0^{2\pi}D^2_\alpha\big(\tilde\theta_1-\tilde\theta_2\big)(D_\alpha^3+D_\alpha)\mathcal{Q}_1\Big(\frac{2\pi}{L_1}\cos\big(\alpha+\theta_1(\alpha)\big)-\frac{2\pi}{L_2}\cos\big(\alpha+\theta(\alpha)\big)\Big)d\alpha.
\end{eqnarray*}
Using  (B.3) and by Lemma \ref{lem3.11} and Proposition \ref{propgamma}, we have
\begin{multline*}
I_1=-\sigma\int_0^{2\pi}D_\alpha^3\big(\tilde\theta_1-\tilde\theta_2\big)\Lambda D_\alpha^3\big(\frac{4\pi^3}{L_1^3}\tilde\theta_1-\frac{4\pi^3}{L_2^3}\tilde\theta_2\big)d\alpha+\sigma\int_0^{2\pi}D_\alpha^2\big(\tilde\theta_1-\tilde\theta_2\big)\Lambda D^2\big(\frac{4\pi^3}{L_1^3}\tilde\theta_1-\frac{4\pi^3}{L_2^3}\tilde\theta_2\big)d\alpha\\
+\big(1+\frac{\mu_2}{\mu_1+\mu_2}u_0\big)\int_0^{2\pi}D^2_\alpha\big(\tilde\theta_1-\tilde\theta_2\big)\Lambda D^2_\alpha\mathcal{Q}_1\Big(\frac{2\pi}{L_1}\sin\big(\alpha+\theta_1(\alpha)\big)-\frac{2\pi}{L_2}\sin\big(\alpha+\theta_2(\alpha)\big)\Big)d\alpha\\
-a_\mu\int_0^{2\pi}D_\alpha^2\big(\tilde\theta_1-\tilde\theta_2\big)\Lambda  D^2_\alpha\mathcal{Q}_1\Big(\frac{2\pi^2}{L_1^2}\mathcal{F}[z_1]\gamma_1-\frac{2\pi^2}{L_2^2}\mathcal{F}[z_2]\gamma_2\Big)d\alpha\\
+\big(1+\frac{\mu_2}{\mu_1+\mu_2}u_0\big)\int_0^{2\pi} D^2_\alpha\big(\tilde\theta_1-\tilde\theta_2\big)\Lambda \mathcal{Q}_1\Big(\frac{2\pi}{L_1}\sin\big(\alpha+\theta_1(\alpha)\big)-\frac{2\pi}{L_2}\sin\big(\alpha+\theta_2(\alpha)\big)\Big)d\alpha
\\-a_\mu\int_0^{2\pi} D^2_\alpha\big(\tilde\theta_1-\tilde\theta_2\big)\Lambda \mathcal{Q}_1\Big(\frac{2\pi^2}{L_1^2}\mathcal{F}[z_1]\gamma_1-\frac{2\pi^2}{L_2^2}\mathcal{F}[z_2]\gamma_2\Big)d\alpha\\
+\int_0^{2\pi} D^2_\alpha\big(\tilde\theta_1-\tilde\theta_2\big)(D_\alpha^3+D_\alpha)\mathcal{Q}_1\Big(\frac{2\pi^2}{L_1^2}\Re\big(\mathcal{G}[z_1]\gamma_1\big)-\frac{2\pi^2}{L_2^2}\Re\big(\mathcal{G}[z_2]\gamma_2\big)\Big)d\alpha\\
+(u_0+1)\int_0^{2\pi} D^2_\alpha\big(\tilde\theta_1-\tilde\theta_2\big)(D^3_\alpha+D_\alpha)\mathcal{Q}_1\Big(\frac{2\pi}{L_1}\cos\big(\alpha+\theta_1(\alpha)\big)-\frac{2\pi}{L_2}\cos\big(\alpha+\theta_2(\alpha)\big)\Big)d\alpha\\
\leq-\sigma\int_0^{2\pi}D^3_\alpha\big(\tilde\theta_1-\tilde\theta_2\big)\Lambda D^3_\alpha\big(\frac{4\pi^3}{L_1^3}\tilde\theta_1-\frac{4\pi^3}{L_2^3}\tilde\theta_2\big)d\alpha\\
+C\|\tilde\theta_1-\tilde\theta_2\|_2\Big(\|\theta_1-\theta_2\|_3+\beta|y_1(0,t)-y_2(0,t)|\Big),
\end{multline*}
where $C$ depends on $\epsilon$.
For $I_2, I_3, I_4$ and $I_5$, by (\ref{U3.32}) and (\ref{T3.33}) in Proposition \ref{propgamma},  we obtain
\begin{multline}\label{diffenergy}
I_2+I_3+I_4+I_5\leq C\|\tilde\theta_1-\tilde\theta_2\|_2\Big(\|\theta_1-\theta_2\|_3+\beta|y_1(0,t)-y_2(0,t)|\Big)\\
+C\big|\hat{\theta}_1(0;t)-\hat{\theta}_2(0;t)\big|\Big(\|\theta_1-\theta_2\|_3+\beta|y_1(0,t)-y_2(0,t)|\Big)\\+\big|y_1(0,t)-y_2(0,t)\big|\|U_1(\cdot,t)\sin\big(\cdot+\theta_1(\cdot,t)\big)-U_2(\cdot.t)\sin\big(\cdot+\theta_2(\cdot,t)\big)\big\|_1\\
\leq
C\|\tilde\theta_1-\tilde\theta_2\|_2\Big(\|\theta_1-\theta_2\|_3+\beta|y_1(0,t)-y_2(0,t)|\Big)\\
+C\big|\hat{\theta}_1(0;t)-\hat{\theta}_2(0;t)\big|\Big(\|\theta_1-\theta_2\|_3+\beta|y_1(0,t)-y_2(0,t)|\Big)\\+\big|y_1(0,t)-y_2(0,t)\big|\Big(\|\theta_1-\theta_2\|_3+\beta|y_1(0,t)-y_2(0,t)|\Big),\\
\end{multline}
where $C$ depends on $\epsilon$.
Actually, combining the estimates for $I_1$, $I_2$, $I_3$, $I_4$ and $I_5$, by Cauchy inequality, we have
\begin{equation*}
\frac{dE^d(t)}{dt}\leq CE^d(t).
\end{equation*}
That is
$$E^d(t)\leq E^d(0)e^{Ct}.$$
Hence, $E^d(t)=0$ if $E^d(0)=0$.
\end{proof}

\subsection{The Fr$\acute{e}$chet derivative  
$\mathfrak{U}_{\tilde\theta^{(s)}}[0,0,0]$ in \S 5}\label{FrechU}
From (C.2), $\gamma^{(s)}$ is the result of an
operator acting on $(\tilde\theta^{(s)},u_0,\beta)$.
From substituting ${\tilde \theta}^{(s)} = \epsilon h$ and
taking the $\epsilon$ derivative at $\epsilon=0$ 
and using Proposition \ref{prop2.6},  
we have
 \begin{equation}\label{frechetU}
 \mathfrak{U}_{\tilde\theta^{(s)}}[0,0,0]h=\frac12\mathcal{H}\big[\gamma^{(s)}_{\tilde\theta^{(s)}}[0,0,0]h\big](\alpha)+i\sum_{k=1}^\infty\frac{1}{k+2}\hat{h}(k+1)e^{ik\alpha}-h (\alpha) \sin{\alpha} +c.c..
 \end{equation}
From (C.2) and Proposition \ref{prop2.6},  we have
\begin{multline}\label{frechetgamma}
\gamma^{(s)}_{\tilde\theta^{(s)}}[0,0,0]h(\alpha)=2(1+a_\mu)
h(\alpha) \cos{\alpha} +\sigma h_{\alpha\alpha}(\alpha)+
2a_\mu\mathcal{H}[h \sin{\alpha}](\alpha)\\-a_\mu\sum_{k=1}^\infty\frac{2}{k+2}\hat{h}(k+1)e^{ik\alpha}+c.c..
\end{multline}

Hence, combining (\ref{frechetU}) and (\ref{frechetgamma}), using $\mathcal{H}^2=-I$, we obtain
\begin{multline*}
 \mathfrak{U}_{\tilde\theta^{(s)}}[0,0,0]h=\frac{\sigma}2\mathcal{H}\big[h_{\alpha\alpha}\big](\alpha)+(1+a_\mu)i\sum_{k=1}^\infty\frac{1}{k+2}\hat{h}(k+1)e^{ik\alpha}\\
 -(1+a_\mu) h(\alpha) \sin{\alpha} +(1+a_\mu)\mathcal{H}\big[ h \cos\alpha 
\big](\alpha)+c.c.\\
 =\frac{\sigma}2\mathcal{H}\big[h_{\alpha\alpha}\big](\alpha)-i(1+a_\mu)\sum_{k=1}^\infty\frac{k+1}{k+2}\hat{h}(k+1)e^{ik\alpha}+c.c..
 \end{multline*}

  \subsection{Expressions for $\mathcal{L}_\beta$ and $\mathcal{N}$}\label{A7.4}

\begin{definition}
We define the function
\begin{align*}
\omega_s(\alpha)=\int_0^{\alpha}e^{i\tau+i\theta^{(s)}(\tau)}d\tau.
\end{align*}
\end{definition}
\begin{multline*}
\mathcal{L}_\beta\big[\widetilde\Theta\big](\alpha,t)=\mathcal{Q}_1\Big\{\Big(\frac{1}{2}\mathcal{H}\Big(\mathcal{L}_{\beta_1}
\big[\widetilde\Theta\big]\Big)(\alpha,t)
+\mathcal{L}_{\beta_2}\big[\widetilde\Theta\big](\alpha,t)\Big)_\alpha+\mathcal{L}_{\beta_3}\big[\widetilde\Theta\big](\alpha,t)\Big\},\\
\mathcal{N}\big[\widetilde\Theta\big](\alpha,t)=\frac{2\pi}{L}\mathcal{Q}_1\Big\{\Big(\frac{1}{2}\mathcal{H}\Big(\mathcal{N}_1\big[\widetilde\Theta\big]\Big)(\alpha,t)
+\mathcal{N}_2\big[\widetilde\Theta\big](\alpha,t)\Big)_\alpha+\mathcal{N}_3\big[\widetilde\Theta\big](\alpha,t)\Big\}\nonumber\\
+\frac{2\pi-L}{L}\Big\{ \sum_{k=2}^\infty
(1+a_\mu)\frac{(k^2-1)(k+1)}{k(k+2)}\widehat\Theta(k+1)
e^{ik\alpha}\nonumber\\
-\sum_{k=-2}^{-\infty}
(1+a_\mu)\frac{(k^2-1)(k-1)}{k(k-2)}\widehat\Theta(k-1)
e^{ik\alpha}
+\mathcal{L}_\beta\big[\widetilde\Theta\big](\alpha,t)\Big\},
\end{multline*}
where
\begin{eqnarray*}
&&\mathcal{L}_{\beta_1}\big[\widetilde\Theta\big](\alpha)\\
&=&a_\mu\Re\Big(-\frac{1}{i}\mathcal{G}[z^{(s)}]\Gamma\Big)+a_\mu\Re\Big(-\frac{1}{i}\mathcal{G}_1[z]
(\gamma^{(s)}-2\sin{\alpha})+\frac{1}{i}\mathcal{G}_1[\omega_s]
(\gamma^{(s)}-2\sin\alpha)\Big)\nonumber\\
 &&-a_\mu \Re\Big(z_{\alpha}\mathcal{K}_2[z]\gamma^{(s)}(\alpha)-i\omega_{s_\alpha}\mathcal{K}_2[z^{(s)}]\gamma^{(s)}(\alpha)\Big) -4a_\mu D_\alpha\Re\Big(\mathfrak{B}[\Theta](\alpha)-\mathfrak{W}[\Theta](\alpha)\Big)\\
&&+\frac{L-2\pi}{\pi}\big(\sin(\alpha+\theta^{(s)})-\sin{\alpha}\big)+2\Theta\big(\cos(\alpha+\theta^{(s)})-\cos\alpha\big)+\frac{2\pi-L}{L}\sigma\theta^{(s)}_{{\alpha\alpha}}\\
&&+\frac{L-2\pi}{\pi}\frac{\mu_2}{\mu_1+\mu_2}u_0\sin\big(\alpha+\theta\big)+2\frac{\mu_2}{\mu_1+\mu_2}u_0\Big(\sin\big(\alpha+\theta)-\sin\big(\alpha+\theta^{(s)}\big)\Big),
\end{eqnarray*}
\begin{multline*}
\mathcal{L}_{\beta_2}\big[\widetilde\Theta\big]=\frac{2\pi-L}{2L}\mathcal{H}[\gamma^{(s)}-2\sin\alpha']
+\Re\Big(\frac{1}{2}\mathcal{G}[z^{(s)}]\Gamma\Big)
+\Re\Big(\frac{\pi}{L}\mathcal{G}_1[z]\big(\gamma^{(s)}(\alpha)-2\sin{\alpha}\big)\\
-\frac{1}{2}\mathcal{G}_1[\omega_s]\big(\gamma^{(s)}(\alpha)-2\sin{\alpha}\big)\Big)+2 \Re\Big(-\omega_{\alpha}\mathcal{K}_2[z]\gamma^{(s)}(\alpha)+\omega_{s_\alpha}\mathcal{K}_2[z^{(s)}]\gamma^{(s)}(\alpha)\Big)\\
 +\frac{2\pi-L}{L}\Re\Big(\mathcal{G}_1[\omega_s]\sin\alpha\Big)
+\Re\Big(2i\frac{\partial}{\partial\alpha}\mathfrak{B}[\Theta](\alpha)-2i\frac{\partial}{\partial\alpha}\mathfrak{W}[\Theta](\alpha)\Big)\\
+u_0\big[\cos\big(\alpha+\theta(\alpha)\big)-\cos\big(\alpha+\theta^{(s)}(\alpha)\big)\big],
\end{multline*}
\begin{eqnarray*}
\mathcal{L}_{\beta_3}\big[\widetilde\Theta\big]
&=&\Big(\int_0^\alpha\theta^{(s)}_{\alpha}(\alpha')\big(\frac{2\pi^2}{L^2}\sigma\mathcal{H}(\Theta_{\alpha\alpha})(\alpha')+\mathcal{L}\big[\widetilde\Theta\big](\alpha')\big)d\alpha'\\
&&-\frac{\alpha}{2\pi}\int_0^{2\pi}\theta^{(s)}_{\alpha}(\alpha)\big(\frac{2\pi^2}{L^2}\sigma\mathcal{H}(\Theta_{\alpha\alpha})(\alpha')+\mathcal{L}\big[\widetilde\Theta\big](\alpha')\big)d\alpha\Big)(1+\theta^{(s)}_{\alpha})\nonumber\\
&& +\theta_{\alpha}^{(s)}\Big(\int_0^\alpha\big(\frac{2\pi^2}{L^2}\sigma\mathcal{H}(\Theta_{\alpha\alpha})(\alpha')+\mathcal{L}\big[\widetilde\Theta\big](\alpha')\big)d\alpha'
-\frac{\alpha}{2\pi}\int_0^{2\pi}
\mathcal{L}\big[\widetilde\Theta\big](\alpha)d\alpha\Big)\nonumber\\&&+\Big\{\int_0^\alpha\big(1+\theta^{(s)}_{\alpha}(\alpha')\big)
\Big[\frac{1}{2}\mathcal{H}\Big(\mathcal{L}_{\beta_1}\big[\widetilde\Theta\big](\cdot)\Big)(\alpha')
+\mathcal{L}_{\beta_2}\big[\widetilde\Theta\big](\alpha')\Big]d\alpha'\nonumber\\
&&-\frac{\alpha}{2\pi}\int_0^{2\pi}\big(1+\theta_{\alpha}^{(s)}(\alpha)\big)
\Big[\frac{1}{2}\mathcal{H}\Big(\mathcal{L}_{\beta_1}\big[\widetilde\Theta\big](\cdot)\Big)(\alpha)
+\mathcal{L}_{\beta_2}\big[\widetilde\Theta\big](\alpha)\Big]
d\alpha\Big\}(1+\theta^{(s)}_{\alpha}),\nonumber\\
\end{eqnarray*}
\begin{multline*}\mathcal{N}_1\big[\widetilde\Theta\big]=a_\mu\Re\Big(-\frac{1}{
i}\mathcal{G}[z]\Gamma+\frac{1}{
i}\mathcal{G}[z^{(s)}]\Gamma\Big)
+\frac{L-2\pi}{\pi}\Big(\sin(\alpha+\theta)
-\sin\big(\alpha+\theta^{(s)}(\alpha)\big)\Big)\\
+2\Big(\sin(\alpha+\theta)-\sin\big(\alpha+\theta^{(s)}(\alpha)\big)-\Theta\cos\big(\alpha+\theta^{(s)}(\alpha)\big)\Big) -2a_\mu\Re\Big(\frac{1}{i}\frac{\partial}{\partial\alpha}\big\{\mathfrak{B}[\Xi_e[\Theta]](\alpha)\big\}\Big)\\
+2a_\mu\Re\Big(i(e^{i\Theta}-1)\Big\{\frac{\omega_{s_\alpha}}{\omega_\alpha}\big(\mathcal{G}_1[\omega]\sin\alpha-\cos\alpha\big)-\big(\mathcal{G}_1[\omega_s]\sin\alpha-\cos\alpha\big)\Big\}\Big)
\\+2a_\mu\Re\Big(\frac{\omega_{s_\alpha}}{\pi i
}\mbox{PV}\int_{\alpha-\pi}^{\alpha+\pi}\sin(\alpha')\frac{q_1[\omega-\omega_s](\alpha,\alpha')}{q_1[\omega_s](\alpha,\alpha')}
\Big(\frac{1}{\omega(\alpha)-\omega(\alpha')}-\frac{1}{\omega_s(\alpha)-\omega_s(\alpha')}
\Big)d\alpha'\Big),
\end{multline*}
\begin{multline}
\mathcal{N}_2\big[\widetilde\Theta\big]
=\frac{2\pi-L}{2L}\mathcal{H}[\Gamma-\frac{2\pi}{L}\sigma\Theta_{\alpha\alpha}]+\Re\Big(\frac{\pi}{L}\mathcal{G}[z]\Gamma-\frac{1}{2}\mathcal{G}[z^{(s)}]\Gamma
\Big)\\
+\frac{2\pi-L}{L}\Re\Big(\mathcal{G}_1[\omega]\sin\alpha-\mathcal{G}_1[\omega_s]\sin\alpha\Big)
+\Re\Big(\frac{\partial}{\partial\alpha}\big(\mathfrak{B}[\Xi_e[\Theta]](\alpha)\big)\Big)\\
+\Re\Big((e^{i\Theta}-1)\big\{\frac{\omega_{s_\alpha}}{\omega_\alpha}\big(\mathcal{G}_1[\omega]\sin(\alpha)-\cos\alpha\big)-\big(\mathcal{G}_1[\omega_s]\sin(\alpha)-\cos\alpha\big)\big\}\Big)\\
+\Big(\cos\big(\alpha+\theta(\alpha)\big)-\cos\big(\alpha+\theta^{(s)}\big)+\Theta\sin\big(\alpha+\theta^{(s)}\big)\Big)\\
-\Re\Big(\frac{\omega_{0_\alpha}}{\pi
}\mbox{PV}\int_{\alpha-\pi}^{\alpha+\pi}\sin(\alpha')\frac{q_1[\omega-\omega_s](\alpha,\alpha')}{q_1[\omega_s](\alpha,\alpha')}\cdot
\Big(\frac{1}{\omega(\alpha)-\omega(\alpha')}
-\frac{1}{\omega_s(\alpha)-\omega_s(\alpha')}
\Big)d\alpha'\Big),\nonumber
\end{multline}
\begin{eqnarray}
\mathcal{N}_3\big[\widetilde\Theta\big]&=&\Big\{\int_0^\alpha\big(1+\theta^{(s)}_{\alpha}(\alpha')\big)
\Big[\frac{1}{2}\mathcal{H}\Big(\mathcal{N}_1\big[\widetilde\Theta\big](\cdot)\Big)(\alpha')
+\mathcal{N}_2\big[\widetilde\Theta\big](\alpha')\Big]d\alpha'\nonumber\\
&& -\frac{\alpha}{2\pi}\int_0^{2\pi}
\big(1+\theta^{(s)}_{\alpha}(\alpha)\big)
\Big[\frac{1}{2}\mathcal{H}\Big(\mathcal{N}_1\big[\widetilde\Theta\big](\cdot)\Big)(\alpha)
+\mathcal{N}_2\big[\widetilde\Theta\big](\alpha)\Big]d\alpha\nonumber\\
&&+\int_0^\alpha\Theta_\alpha(\alpha')
U(\alpha')d\alpha'-\frac{\alpha}{2\pi}\int_0^{2\pi}\Theta_\alpha(\alpha)U(\alpha)d\alpha\Big\}(1+\theta^{(s)}_{\alpha})\nonumber\\
&&+\Big(\int_0^\alpha\Theta_\alpha(\alpha')U(\alpha')d\alpha'-\frac{\alpha}{2\pi}\Theta_\alpha(\alpha)U(\alpha)d\alpha\Big)
\Theta_\alpha(\alpha).\nonumber
\end{eqnarray}
\begin{multline*}
\mathcal{L}_\Gamma\big[\widetilde\Theta\big](\alpha)=2\Theta(\alpha,t)\cos{\alpha}+\frac{L-2\pi}{\pi}\sin{\alpha}-4a_\mu\Re\Big(\frac{\partial}{\partial\alpha}\big\{\mathfrak{W}[\Theta](\alpha)\big\}\Big),
\end{multline*}
\begin{eqnarray*}\mathcal{L}\big[\widetilde\Theta\big]
&=&\frac{1}{2}\mathcal{H}[\mathcal{L}_\Gamma](\alpha,t)+\frac{L-2\pi}{L}\cos{\alpha}
-\mathcal{Q}_0\theta\sin\alpha+\Re\Big(i\frac{\partial}{\partial\alpha}\big(\mathfrak{W}[\Theta](\alpha)\big)\Big).
\end{eqnarray*}

\end{document}